\documentclass{amsart}
\usepackage[utf8]{inputenc}
\usepackage{amsfonts, amsmath, amsthm, amssymb}
\usepackage{graphicx}

\usepackage{a4wide}
\usepackage{rotating}
\usepackage{bbm}
\usepackage[all,cmtip]{xy}

\usepackage{dsfont}
\usepackage{color}
\usepackage{hyperref}
\usepackage{verbatim}
\usepackage{mathtools}
\usepackage{enumitem}
\usepackage{epstopdf}
\usepackage{booktabs}

\usepackage{tikz}
\usepackage[titletoc]{appendix}
\usetikzlibrary{patterns}
\usetikzlibrary{decorations.pathreplacing}
\usetikzlibrary{arrows.meta}
\usetikzlibrary{cd}
\usepackage[margin=1.0in]{geometry}
\usepackage{caption}
\usepackage{float}
\usepackage{mathtools}
\usepackage{pgfplots}
\usepackage{tkz-euclide}

\newcommand{\Z}{\mathbb{Z}}
\newcommand{\R}{\mathbb{R}}
\newcommand{\Q}{\mathbb{Q}}
\newcommand{\C}{\mathbb{C}}

\newcommand{\GL}{\mathrm{GL}}

\newcommand{\U}{\mathrm{U}}

\DeclareMathOperator{\im}{im}

\DeclareMathOperator{\id}{id}

\numberwithin{equation}{section}

\newtheorem{thm}{Theorem}[section]

\newtheorem{lemma}[thm]{Lemma}
\newtheorem{proposition}[thm]{Proposition}

\newtheorem{defn}[thm]{Definition}
\newtheorem{remark}[thm]{Remark}
\newtheorem{corollary}[thm]{Corollary}

\begin{document}

\title[Simple homotopy theory for Fukaya categories]{Simple homotopy theory for Fukaya categories}
\author{Yonghwan Kim}
\email{yonghkim@mit.edu}

\maketitle
\vspace{-1.5em}
\begin{abstract}
We develop a categorical framework for simple homotopy theory in Fukaya categories, based on the fundamental group of the ambient symplectic manifold. When the first Chern class vanishes, we show that any isomorphism in the Fukaya category of a Weinstein manifold has trivial Whitehead torsion. As an application, we prove that any pair of closed connected Lagrangians that are isomorphic in the Fukaya category of such Weinstein manifolds are simple homotopy equivalent, provided one of the Lagrangians is homotopy equivalent to the ambient symplectic manifold and their fundamental groups are isomorphic.
\end{abstract}

\section{Introduction\label{sec:intro}}

\subsection{Results}

Lens spaces provide examples of closed manifolds that are homotopy equivalent but not homeomorphic; a well-known case is the pair $L(7,1)$ and $L(7,2)$. This example is particularly striking in symplectic geometry because their cotangent bundles $T^*L(7,1)$ and $T^*L(7,2)$ are diffeomorphic \cite{MilnorLens}.

Abouzaid--Kragh \cite{AbouzaidKragh} showed that the cotangent bundles of two 3-dimensional lens spaces are symplectomorphic if and only if the two lens spaces themselves are diffeomorphic. In contrast, Karabas--Lee \cite{KarabasLee} computed that the wrapped Fukaya categories of $T^*L(7,1)$ and $T^*L(7,2)$ are quasi-equivalent. This raises the question: can Fukaya categories distinguish the simple homotopy types of Lagrangian submanifolds?

In this paper, we develop a categorical framework for simple homotopy theory in the setting of Fukaya categories, based on the fundamental group of the ambient symplectic manifold. This allows us to extract information about the simple homotopy types of closed exact Lagrangian submanifolds.

Before turning to the technical aspects, we illustrate several applications of this framework. These examples are built around a key geometric construction: the Weinstein 1-handle connect sum. Given two Liouville manifolds $X$ and $Y$, one can construct their \emph{Weinstein 1-handle connect sum}, which we denote by $X\natural Y$. Topologically, this operation corresponds to forming a connected sum along a 1-handle. 

\begin{thm}[Theorem \ref{thm:auteq-reprise}] \label{thm:auteq}
    Let $X=T^*L(7,1)\natural T^*L(7,2)$ be the Weinstein 1-handle connect sum of cotangent bundles. Then any exact symplectomorphism $\phi:X\to X$ acts on middle-dimensional homology as one of the four possibilities $(\pm\id,\pm\id)$. In particular, it preserves the direct sum decomposition
    \begin{equation}
        H_3(X)\cong H_3(T^*L(7,1))\oplus H_3(T^*L(7,2)).
    \end{equation}
\end{thm}

By Lemma \ref{lem:swap-diffeo}, there exists a self-diffeomorphism of $X$ which swaps the two summands of $H_3(X)$. Thus, the obstruction exhibited above is purely symplectic, not smooth.

Our second application determines the diffeomorphism types of certain Lagrangian submanifolds:

\begin{thm}[Theorem \ref{thm:lens-space-reprise}] \label{thm:lens-space}
    Let $X$ be a simply-connected Weinstein manifold of dimension 6 with $c_1(X)=0$, and let $L(p,q)$ be a 3-dimensional lens space. Then any connected closed exact Maslov zero Spin Lagrangian submanifold $L$ in $M=T^*L(p,q)\natural X$, for which the induced map $\pi_1(L)\to\pi_1(M)$ is an isomorphism, must be diffeomorphic to $L(p,q)$.
\end{thm}

The Maslov zero and Spin conditions ensure that these Lagrangians admit brane structures and hence define objects in the Fukaya category.

Our main application shows that any two closed exact Lagrangians in a Weinstein manifold that are isomorphic in the Fukaya category are simple homotopy equivalent, assuming one of them is homotopy equivalent to the ambient symplectic manifold and their fundamental groups are isomorphic. This extends the result of Abouzaid--Kragh \cite{AbouzaidKragh} that any closed exact Lagrangian submanifold in the cotangent bundle of a closed smooth manifold is simple homotopy equivalent to the zero section to general Weinstein manifolds.

\begin{thm}[Theorem \ref{thm:simple-he-Lagrangians-reprise}] \label{thm:simple-he-Lagrangian}
    Let $X$ be a Weinstein manifold with $c_1(X)=0$, and let $L$ be a connected closed exact Maslov zero Spin Lagrangian such that the inclusion $L\xhookrightarrow{}X$ is a homotopy equivalence. Suppose that $K$ is another connected closed exact Maslov zero Spin Lagrangian for which the induced map $\pi_1(K)\to\pi_1(X)$ is an isomorphism, and assume that there exist brane structures on $K$ and $L$ such that they define isomorphic objects in the compact Fukaya category $\mathcal{F}(X)$. Then the inclusion $K\xhookrightarrow{}X$ is also a homotopy equivalence, and the composition with any homotopy inverse of $L\xhookrightarrow{}X$
    \begin{equation}
        K\xhookrightarrow{}X\to L
    \end{equation}
    is a simple homotopy equivalence.
\end{thm}

As an application, we consider \textit{exotic Weinstein balls} $\Sigma^{2n}$, which are nonequivalent Weinstein manifolds whose underlying topological spaces are homeomorphic to the standard ball $B^{2n}\cong T^*\R^n$. Examples were constructed by Seidel--Smith \cite{SeidelSmith}, McLean \cite{McLean}, and Abouzaid--Seidel \cite{AbouzaidSeidelLefschetz}.

\begin{thm}[Theorem \ref{thm:exotic-ball-reprise}] \label{thm:exotic-ball}
    Let $X=T^*Q\natural\Sigma$ be the Weinstein 1-handle connect sum of an exotic Weinstein ball $\Sigma$ with the cotangent bundle of a non-simply connected Spin manifold $Q$. Then any connected closed exact Maslov zero Lagrangian submanifold $L\subset X$ with $\pi_1(L) \cong \pi_1(X)$ is simple homotopy equivalent to the zero section of $T^*Q$.
\end{thm}

The proof also shows that any such Lagrangian $L$ admits a brane structure for which it defines an object isomorphic to $Q$ in $\mathcal{F}(X)$, which allows us to apply Theorem \ref{thm:simple-he-Lagrangians-reprise}.

These applications all rely on the main outcome of our study of simple homotopy theory in Fukaya categories, which shows that the simple homotopy types of certain closed Lagrangian submanifolds can be recovered from their isomorphism class in the Fukaya category.

\begin{thm}[Theorem \ref{thm:Weinstein-equivalence-simple-reprise}] \label{thm:Weinstein-equivalence-simple}
    Let $X$ be a Weinstein manifold with $c_1(X)=0$, and let $K$ and $L$ be two connected closed exact Maslov zero Lagrangians in $X$, equipped with brane structures. If $K$ and $L$ define isomorphic objects in the Fukaya category of $X$ with $\Z$-coefficients, then the isomorphism has trivial Whitehead torsion. In particular, if the induced maps $\pi_1(K)\to\pi_1(X)$, $\pi_1(L)\to\pi_1(X)$ are injective, the Reidemeister torsions of the cellular cochain complexes of $K$ and $L$, computed using any representation $\pi_1(X)\to\C$ via inclusions, must agree when defined.
\end{thm}

An interesting class of Lagrangians $L$ in a Weinstein manifold $X$ with injective $\pi_1(L)\to\pi_1(X)$ is provided by connected closed exact Lagrangians in Weinstein manifolds constructed by attaching subcritical handles to a cotangent bundle over a small contractible patch of the base \cite[Corollary 14.2]{HusinKragh}.

In the remainder of the introduction, we explain how to associate a Whitehead torsion element to each pair of isomorphic Lagrangians in the Fukaya category, and outline the strategy to proving Theorem \ref{thm:Weinstein-equivalence-simple}.

\subsection{Whitehead torsion in Fukaya categories}

Our study of simple homotopy theory in Fukaya categories builds on the $A_\infty$-bimodule $\underline{CF^*(K,L)}$, originally introduced in \cite{AbouzaidKragh}. For expository purposes, we restrict to the case when the given pair of Lagrangians $K$ and $L$ meet transversely in a symplectic manifold $X$. This bimodule is a free $\Z$-module generated by all lifts $\Tilde{x}$ of intersection points $x\in K\cap L$ to the universal cover $\Tilde{X}$ of $X$. The differential on $\underline{CF^*(K,L)}$ counts all lifts of pseudoholomorphic strips in $X$ with boundary on $K$ and $L$. After choosing a lift $\Tilde{x}$ for each $x$, this bimodule as a cochain complex can be identified with a free left $\Z\pi_1(X)$-module generated by the chosen lifts $\Tilde{x}$, with differential given by summing over rigid holomorphic strips $u$, each weighted by an element $g(u)\in\pi_1(X)$ representing the homotopy class of a path traced by $u$ in $X$.

Cochain complexes equipped with a chosen basis admit a refined classification up to simple homotopy equivalence. In particular, to any acyclic cochain complex over a group ring $\Z G$ equipped with a basis, one can associate a \emph{Whitehead torsion} element in the Whitehead group $Wh(G)$ of the group $G$. The Whitehead torsion of the mapping cone of a homotopy equivalence between based cochain complexes vanishes if and only if the map is a simple homotopy equivalence.

Since Floer and Morse theoretic cochain complexes come with natural geometric bases, it is compelling to consider their simple homotopy types. Related approaches to Whitehead torsion and Reidemeister torsion in Floer-theoretic and Morse-theoretic contexts have appeared in the work of Hutchings--Lee \cite{HutchingsLee}, Sullivan \cite{Sullivan}, Suarez \cite{Suarez}, Abouzaid--Kragh \cite{AbouzaidKragh}, Charette \cite{Charette}, Alvarez-Gavela--Igusa \cite{AlvarezGavelaIgusa}, Kenigsberg--Porcelli \cite{KenigsbergPorcelli}, and Courte--Porcelli \cite{CourtePorcelli}.

To incorporate the above bimodule into a categorical framework, we extend the bimodule to twisted complexes of Lagrangians, which we denote by $\underline{CF^*(\mathcal{K},\mathcal{L})}$. For technical reasons, we work with twisted complexes twisted by cochain complexes over $\Z$, rather than free $\Z$-modules. Accordingly, we introduce the $A_\infty$-category $Tw_{Ch}\mathcal{A}$ of such twisted complexes. When the original $A_\infty$-category $\mathcal{A}$ is cohomologically unital, $Tw_{Ch}\mathcal{A}$ is also cohomologically unital and triangulated.

The simple homotopy type of the $A_\infty$-bimodule $\underline{CF^*(\mathcal{K},\mathcal{L})}$, viewed as a cochain complex over $\Z\pi_1(X)$, is independent of the choice of lifts $\Tilde{x}$. In particular, if $\underline{CF^*(\mathcal{K},\mathcal{L})}$ is acyclic, its Whitehead torsion is well-defined. This observation leads us to introduce the following simple homotopy-theoretic refinement of standard categorical notions. 

Recall that an object $K$ in a category $\mathcal{C}$ is defined to be left acyclic if $\hom^*_\mathcal{C}(K,L)$ is acyclic for every object $L$. We define an object $\mathcal{K}$ in the category of twisted complexes $Tw_{Ch}\mathcal{F}(X)$ over a Fukaya category $\mathcal{F}$ to be \emph{left simply acyclic}, if for any other twisted complex $\mathcal{L}$, the $A_\infty$-bimodule $\underline{CF^*(\mathcal{K},\mathcal{L})}$, regarded as a based cochain complex, is acyclic with trivial Whitehead torsion. The notion of \emph{right simply acyclic} objects is defined analogously.

Many categorical notions, such as isomorphisms and generation, can be reformulated in terms of acyclicity of certain twisted complexes. This allows us to define \emph{simple} analogues of these notions. For example, a collection of Lagrangians $\{L_i\}$ is said to \emph{generate} the Fukaya category $\mathcal{F}(X)$ if for every object $K$, there exists a twisted complex $\mathcal{K}$ built from $\{L_i\}$ such that $K\cong\mathcal{K}$, or equivalently the mapping cone $[K\to\mathcal{K}]$ is acyclic. We say $\{L_i\}$ \emph{simply generate} $\mathcal{F}(X)$ if for every $K$, the twisted complex $[K\to\mathcal{K}]$ can be chosen to be both left and right simply acyclic.

A key structural result about these refined notions is the following ``automatic simplicity lemma'', which states that under certain assumptions on the generators, any isomorphism in the Fukaya category is automatically simple.

\begin{lemma}[Lemma \ref{prop:automaticsimplicity}] \label{lem:autsimp-intro}
    Let $X$ be an exact symplectic manifold, and let $\mathcal{F}$ be a version of a Fukaya category of $X$. If $\mathcal{F}$ admits simple generators $\{L_i\}$ such that each $L_i$ is simply connected, then every isomorphism in $\mathcal{F}$ is a simple isomorphism.
\end{lemma}

We state the lemma for a general Fukaya category without specifying a particular model, as the result applies to any version in which the morphism spaces are proper, i.e. finite-dimensional in the chain level so that Whitehead torsion is well-defined. 

Now let $K$ and $L$ be two closed exact Lagrangians which define isomorphic objects in the compact Fukaya category $\mathcal{F}(X)$ when equipped with brane structures. To apply Lemma \ref{lem:autsimp-intro}, we need a larger Fukaya category which contains the category of closed exact Lagrangians as a full subcategory, and admits simply connected simple generators. The version used in this paper is the Fukaya category of Lefschetz fibrations as in \cite{Seidel18}:

\begin{proposition}[Proposition \ref{prop:simple-gen-fuk-lef}] \label{prop:simple-gen-Lef-fib}
    Let $X$ be the total space of a Lefschetz fibration $\pi:X\to\C$. Then the Lefschetz thimbles associated to $\pi$ simply generate $\mathcal{F}(\pi)$, the Fukaya category of the Lefschetz fibration $\pi$.
\end{proposition}

By the result of Giroux and Pardon \cite{GirouxPardon}, any Weinstein manifold can be presented as the total space of a Lefschetz fibration after a Weinstein deformation. Thus, given a Weinstein manifold $X$, we may deform it to one admitting a Lefschetz fibration $\pi:X\to\C$. By Lemma \ref{lem:autsimp-intro} and Proposition \ref{prop:simple-gen-Lef-fib}, it follows that any isomorphism between two closed exact Lagrangians $K$, $L$ in the Fukaya category $\mathcal{F}(X)$ must be a simple isomorphism. In particular, the homotopy equivalences of cochain complexes
\begin{equation}
    \underline{CF^*(K,K)}\simeq\underline{CF^*(K,L)}\simeq\underline{CF^*(L,L)}
\end{equation}
induced from the module actions of the isomorphism elements from $K\cong L$ are simple homotopy equivalences. When the maps $\pi_1(K)\to\pi_1(X)$ and $\pi_1(L)\to\pi_1(X)$ are injective, and the representations $\rho:\pi_1(K)\to\C$ and $\rho':\pi_1(L)\to\C$ both factor through a common representation $\psi:\pi_1(X)\to\C$, the cochain complexes $\underline{CF^*(K,K)}$ and $\underline{CF^*(L,L)}$ recover the Reidemeister torsions of the cellular cochain complexes of $K$ and $L$. Since Reidemeister torsion is invariant under simple homotopy equivalences, this proves Theorem \ref{thm:Weinstein-equivalence-simple}.

\subsection{Organization of the paper}

In Section 2, we review Whitehead torsion and Reidemeister torsion, including an explicit computation of Reidemeister torsion for lens spaces $L(p,q)$. Section 3 introduces the category of twisted complexes twisted by cochain complexes, denoted by $Tw_{Ch}\mathcal{C}$, and recalls the Fukaya category of the Lefschetz fibration as constructed in \cite{Seidel18}. Section 4 forms the core of the paper: we define the $A_\infty$-bimodule $\underline{CF^*(K,L)}$ and its extension to twisted complexes $\underline{CF^*(\mathcal{K},\mathcal{L})}$. In Subsection 4.2, we use this framework to define ``simple'' versions of categorical notions such as acyclic objects, isomorphism, and generation. In Subsection 4.3, we show that the Fukaya category of a Lefschetz fibration is simply generated by the Lefschetz thimbles. Finally, Section 5 presents applications of the theory developed above.

\subsection{Acknowledgements}

First and foremost, I would like to thank my advisor Paul Seidel for his continued support and guidance throughout this project. I am especially grateful to Zihong Chen and Thomas Massoni for many helpful discussions. I would also like to thank Mohammed Abouzaid, Shaoyun Bai, Kenny Blakey, Thomas Kragh, Jae Hee Lee, Gheehyun Nahm, Noah Porcelli, and Jonathan Zung for helpful comments and discussions at different stages of this project. I would also like to thank the anonymous referee for helpful feedback. This work was partially supported by an MIT Peterson Fellowship award.

\section{Whitehead torsion and Reidemeister torsion} \label{sec:torsion}

In this section, we review the notions of Whitehead torsion and Reidemeister torsion. The exposition will mostly follow Milnor's survey \cite{MilnorSurvey}. For further details, the reader may consult Turaev \cite{Turaev}, or see Abouzaid--Kragh \cite[Section 2]{AbouzaidKragh} for a quick overview of the key definitions and properties.

\subsection{Whitehead torsion}\label{ssec:Whitehead-torsion}

Whitehead torsion is defined for chain complexes and cochain complexes equipped with a basis that freely generates each degree component. With respect to such a basis, the differentials can be represented by matrices with entries in the coefficient ring. To make this precise, we introduce the definition of a based (co)chain complex. For our purposes of defining Whitehead torsion, we will restrict to complexes which are finitely generated in each degree, and supported in finitely many degrees.

\begin{defn}\label{def:based-complex}
    A \emph{based (co)chain complex} over a ring $R$ is a (co)chain complex of free modules $(C_*,\partial_*)$ equipped with a choice of finite ordered basis $b_i=\{b_i^\alpha\}$ for each degree component $C_i$. We also require that the degree components $C_i$ are nonzero for only finitely many degrees.
\end{defn}

To match our constructions in later sections using Floer cohomology, we will focus on based \emph{cochain} complexes for the remainder of this subsection.

For any ring $R$, we define $\GL_n(R)$ to be the multiplicative group of invertible $n\times n$ matrices with entries in $R$. We also define an inclusion $\GL_n(R)\xhookrightarrow{}\GL_{n+1}(R)$ by mapping a matrix $A\in\GL_n(R)$ to
\begin{equation}
\begin{pmatrix}
    A &0\\
    0 &1_R
\end{pmatrix} \in \GL_{n+1}(R).    
\end{equation}

We define the general linear group $\GL(R)$ as the direct limit of the direct system of inclusions given by
\begin{equation}
    \GL_1(R)\xhookrightarrow{}\GL_2(R)\xhookrightarrow{}\cdots.
\end{equation}

Inside the general linear group, there is a subgroup $E(R)$ generated by the elementary matrices, which are the matrices that agree with the identity matrix up to a possible difference only at exactly one non-diagonal entry. It is known (\cite[Lemma 1.1]{MilnorSurvey}) that this subgroup $E(R)$ is the commutator subgroup of $\GL(R)$.

\begin{defn}\label{def:Whitehead-group-K_1}
    For a ring $R$, we define the group \emph{$K_1(R)$} to be the quotient
    \begin{equation}
        K_1(R) = \GL(R)/E(R).
    \end{equation}
\end{defn}

By the previous discussion, $K_1(R)$ is an abelian group. For example, when $R$ is a field $\mathbb{F}$, the group $K_1(\mathbb{F})$ can be identified with the group of units of $\mathbb{F}$
\begin{equation}
    K_1(\mathbb{F}) \cong \mathbb{F}^\times,
\end{equation}
and the isomorphism is induced by the determinant homomorphism $\det:K_1(\mathbb{F})\to \mathbb{F}^\times$.

As another example, one can verify that $K_1(\Z)$ is isomorphic to $\{\pm1\}$ using the Euclidean algorithm. In practice, we will want to ignore the sign ambiguity in the matrices representing the differentials, which arises from the ordering of the basis elements. Thus, we define a reduced version of $K_1(R)$ that quotients out this sign indeterminacy.

\begin{defn}\label{def:Reduced-Whitehead-group}
    For a ring $R$, we define the \emph{reduced $K_1(R)$} to be the quotient
    \begin{equation}
        \overline{K_1}(R) = K_1(R)/\{\pm1\}.
    \end{equation}
\end{defn}

From the above discussion, it follows that $\overline{K_1}(\Z)$ is the trivial group, and likewise $\overline{K_1}(\mathbb{F})\cong\mathbb{F}^\times/\pm 1$.

Now we define the Whitehead torsion of a based cochain complex. The assumption on our ring $R$ is that it satisfies the \emph{invariant basis number} property:
\
\begin{equation} \label{eqn: IBN_property}
    \mathrm{For~each~nonnegative~integer~}m\neq n,~ R^m \not\cong R^n.
\end{equation}

This condition ensures that the rank of a free $R$-module is independent of the choice of basis. Examples of rings which do not satisfy this assumption arise as Leavitt algebras, introduced in \cite{Leavitt}.

Suppose we have a free $R$-module $R^n$ equipped with two different bases $b=\{b_1,\cdots,b_n\}$ and $c=\{c_1,\cdots,c_n\}$. Then there exists a matrix $A$ representing the change of basis, whose entries $\{a_{ij}\}$ are determined by the formula $c_i=\sum a_{ij}b_j$. We denote the corresponding equivalence class of the matrix $A$ in $\overline{K_1}(R)$ as $[c/b]$.

Now suppose that we have a short exact sequence of finitely generated free $R$-modules
\begin{equation}
\begin{tikzcd}
    0 \arrow[r] &F_0 \arrow[r,"\iota"] &F_1 \arrow[r] &F_2 \arrow[r] &0,
\end{tikzcd}    
\end{equation}
and bases $b=\{b_1,\cdots,b_k\}$ for $F_0$ and $c=\{c_1,\cdots,c_l\}$ for $F_2$. By choosing lifts $\Tilde{c_i}\in F_1$, we can construct a basis $bc$ for $F_1$ as
\begin{equation}
    bc = \{\iota(b_1),\cdots,\iota(b_k),\Tilde{c_1},\cdots,\Tilde{c_l}\}.
\end{equation}
For different choices of lifts $c'=\{\Tilde{c_i}'\}$ for the basis $c$, one can show that the class $[bc'/bc]\in \overline{K_1}(R)$ is trivial. To connect this to cochain complexes, suppose that $(C_*,\partial_*,\{c_i\})$ is a based cochain complex which is acyclic. Then for the submodule $B_i=\im(\partial_{i-1})=\ker(\partial_i)$ of $C_i$, we have a short exact sequence
\begin{equation}\label{eqn-chain-cpx-SES}
\begin{tikzcd}
    0 \arrow[r] &B_i \arrow[r] &C_i \arrow[r, "\partial_i"] &B_{i+1} \arrow[r] &0.
\end{tikzcd}
\end{equation}
From the above discussion, we can construct a basis $q_iq_{i+1}$ for $C_i$ from a preferred basis $q_i$ of $B_i$ and $q_{i+1}$ of $B_{i+1}$, when both $B_i$ and $B_{i+1}$ are free. Moreover, the change of basis matrix $[q_iq_{i+1}/c_i]$ with respect to the given preferred basis $c_i$ of $C_i$ will be independent of choices of the lifts. 

In \cite[Section 4]{MilnorSurvey}, Milnor explains how to generalize the above definition to the case where the modules $B_i$ are only stably free. This condition is always satisfied for cochain complexes supported in finitely many degrees, since the short exact sequence (\ref{eqn-chain-cpx-SES}) implies that each $B_i$ is stably free. Our definition of based cochain complexes incorporates this assumption. When comparing two stable bases $b$, $c$ of different sizes, the symbol $[b/c]$ is extended to denote the change of basis matrix obtained after appropriately enlarging the smaller basis. With this in place, we can now state the definition of torsion for based cochain complexes:

\begin{defn}\label{def:whitehead-torsion-K1}
    Let $(C_i,\partial_i,\{c_i\})$ be a based acyclic cochain complex over a ring $R$. Choose stable bases $q_i$ for each $B_i=\partial_{i-1}(C_{i-1})$. The \emph{Milnor torsion} of the based cochain complex $(C_i,\partial_i,\{c_i\})$ is defined as
    \begin{equation}\label{eqn:def-torsion}
        \eta(C_i,\partial_i,\{c_i\}) = \prod_i ~[q_iq_{i+1}/c_i]^{(-1)^{i-1}}\in \overline{K_1}(R),
    \end{equation}
    where $q_iq_{i+1}$ denotes the concatenated basis of $C_i$. This definition is independent of the choice of basis $q_i$.
\end{defn}

Although this definition a priori depends on the choice of basis $q_i$ for each $B_i$, one can see that when we pick a different basis $q_i'$, the contribution of the change of basis matrix $[q_i'/q_i]$ cancels in the product. Thus the definition of the Milnor torsion only depends on the choice of bases for each $C_i$.

We now discuss some properties of Milnor torsion.

\begin{proposition}[Theorem 3.1, \cite{MilnorSurvey}]\label{prop:torsion-additive}
    Suppose that we have a short exact sequence $ 0 \to C' \to C \to C'' \to 0 $ of based acyclic cochain complexes over $R$, with bases $\{c_i\},\{c_i'\},\{c_i''\}$ that are compatible in the sense that $[c_i'c_i''/c_i]=[\id]\in\overline{K_1}(R)$ for all $i$. Then the Milnor torsion is multiplicative under extensions:
    \begin{equation}
        \eta(C) = \eta(C')\eta(C'')\in \overline{K_1}(R).
    \end{equation}
\end{proposition}

We generalize the above situation to the case where we have a filtered cochain complex, equipped with a basis for each degree term that descends to the associated graded complexes.

\begin{proposition}[Theorem 5.2, \cite{MilnorSurvey}]\label{prop:Whitehead-torsion-filtration}
    Let $(C^*, \partial, \{c^i\})$ be a based acyclic cochain complex over a ring $R$, equipped with a finite filtration by subcomplexes
    \[
        0 = F^{-1}C^* \subset F^0C^* \subset F^1C^* \subset \cdots \subset F^nC^* = C^*.
    \]
    Suppose the following hold:
    \begin{itemize}
        \item Each subquotient $F^kC^* / F^{k-1}C^*$ is an acyclic cochain complex.
        \item The chosen basis $\{c^i\}$ of $C^i$ decomposes as a disjoint union $\{c^i\} = \bigsqcup_k \{c^i_k\}$, and the image of $\{c^i_k\}$ freely generates $F^kC^i / F^{k-1}C^i$.
        \item The Milnor torsion of each graded piece $F^kC^* / F^{k-1}C^*$, computed with respect to the induced basis $\{c^i_k\}$, is trivial in $\overline{K_1}(R)$.
    \end{itemize}
    Then the total cochain complex $(C^*, \partial, \{c^i\})$ has trivial Milnor torsion with respect to the basis $\{c^i\}$.
\end{proposition}

Now consider the following situation: let $(X,Y)$ be a CW pair such that the relative cellular chain complex $C_*^{cell}(X,Y)$ is acyclic. Pick a universal cover $\pi:\Tilde{X}\to X$, and define $\Tilde{Y}=\pi^{-1}(Y)$. After choosing lifts $\Tilde{e}$ of the cells $e$ generating $C_*^{cell}(X,Y)$ to $\Tilde{X}$, we identify the other cells generating $C_*^{cell}(\Tilde{X},\Tilde{Y})$ as $g\cdot \Tilde{e}$ for some $g\in\pi_1(X)$. This gives $C_*^{cell}(\Tilde{X},\Tilde{Y};\Z)$ the structure of a based acyclic chain complex over the group ring $\Z\pi_1(X)$, its Milnor torsion is well-defined.

However, this Milnor torsion depends on the initial choice of lifts $\{\Tilde{e}\}$. To obtain a torsion independent of such choices, we must quotient by the ambiguity arising from the action of $\pi_1(X)$, which parametrizes the possible choices of lifts.

This motivates us to consider the case when our coefficient ring $R$ is a group ring $\Z G$ for some group $G$. These rings satisfy (\ref{eqn: IBN_property}), because they admit a surjective map $\Z G\to \Z$, and the invariant basis number property pulls back under surjective maps: any isomorphism of $\Z G$-modules $(\Z G)^m\cong (\Z G)^n$ induces an isomorphism $\Z^m\cong\Z^n$ as $\Z$-modules, implying $m=n$.

Since $K_1(\Z G)$ is defined as the quotient of the general linear group $\GL(\Z G)$ by its commutator subgroup, there is a group homomorphism $G/[G,G]\to K_1(\Z G)$ induced by the inclusion $G\xhookrightarrow{}\GL_1(\Z G)\to \GL(\Z G)$. 

\begin{defn} \label{def:Whitehead-group-Wh}
    Define the \emph{Whitehead group $Wh(G)$ of a group $G$} to be the cokernel of the map $G/[G,G]\to \overline{K_1}(\Z G)$. This can alternatively be described as the quotient of $K_1(\Z G)$ by the trivial units $\pm G$, understood as $1\times1$ matrices.
\end{defn}

For example, since $\overline{K_1}(\Z)$ is the trivial group, it follows that the Whitehead group $Wh(\{e\})$ of the trivial group $\{e\}$ is trivial.





We now define the Whitehead torsion of an acyclic based cochain complex over the group ring $\Z G$.

\begin{defn}\label{def:Whitehead-torsion-Wh}
    Let $(C_i,\partial_i,\{c_i\})$ be an acyclic based cochain complex over the group ring $\Z G$. We define its \emph{Whitehead torsion} by the image of its Milnor torsion under the canonical projection $\pi:\overline{K_1}(\Z G)\to Wh(G)$:
    \begin{equation}\label{eqn:def-whitehead-torsion}
        \tau(C_i,\partial_i,\{c_i\})=\pi(\eta(C_i,\partial_i,\{c_i\}))\in Wh(G).
    \end{equation}
\end{defn}

Since the Whitehead torsion is defined as the image of Milnor torsion under the projection $\overline{K_1}(\Z G)\to Wh(G)$, it is invariant under replacing a basis element $c_i^\alpha$ by $g\cdot c_i^\alpha$ for any $g\in G$. Moreover, both the multiplicativity of torsion for short exact sequences (Proposition \ref{prop:torsion-additive}) and its vanishing for filtered complexes with acyclic graded pieces with trivial torsion (Proposition \ref{prop:Whitehead-torsion-filtration}) hold for Whitehead torsion.

We also introduce the following definition:

\begin{defn}\label{def:simply-acyclic-chaincpx}
    An acyclic based cochain complex $(C_*,\partial_*,\{c_i\})$ is \emph{simply acyclic} if its associated Whitehead torsion $\tau(C_*,\partial_*,\{c_i\})$ is trivial.
\end{defn}

We now define the Whitehead torsion of a quasi-isomorphism between based complexes. Suppose $f:(C_*,\{c_i\})\to (D_*,\{d_i\})$ is a quasi-isomorphism between based cochain complexes over a group ring $\Z G$. The mapping cone of $f$ is defined as 
\begin{equation}
    cone(f)=C[1]\oplus D,~ \partial_{cone}=\begin{pmatrix}
    -\partial_C &0\\
    -f &\partial_D
\end{pmatrix}.
\end{equation}
Since $f$ is a quasi-isomorphism, $cone(f)$ is acyclic. The bases of $C_*$ and $D_*$ together naturally induce a basis for $cone(f)$, which we denote as $c_{cone}$.

\begin{defn} \label{def:simple-homotopy}
    The \emph{Whitehead torsion $\tau(f)$ of a quasi-isomorphism} $f:(C_*,\{c_i\})\to (D_*,\{d_i\})$ between two based cochain complexes over $\Z G$ is defined to be the Whitehead torsion $\tau(cone(f),c_{cone})$ of the mapping cone $cone(f)$. If $f$ is a chain homotopy equivalence such that $\tau(f)$ is trivial, we say that $f$ is a \emph{simple homotopy equivalence}.
\end{defn}

We now extend this definition to the topological setting, by defining the Whitehead torsion of a homotopy equivalence between CW complexes.

\begin{defn} \label{def:torsion-CW-cpx}
    Let $X$ be a finite connected CW complex, and let $Y \subset X$ be a subcomplex such that $X$ deformation retracts onto $Y$. Choose a universal cover $\pi \colon \widetilde{X} \to X$, equip $\widetilde{X}$ with the induced CW structure, and define $\tilde{Y}\coloneqq\pi^{-1}(Y)$. Fix a lift $\widetilde{e}_k$ of each cell $e_k$ in $X \setminus Y$, and view the relative cellular chain complex $C_*^{\mathrm{cell}}(\widetilde{X}, \widetilde{Y})$ as a based acyclic chain complex over the group ring $\mathbb{Z}\pi_1(X)$.

    The \emph{Whitehead torsion} of the CW pair $(X, Y)$ is defined by
    \[
        \tau(X, Y) := \tau(C_*^{\mathrm{cell}}(\widetilde{X}, \widetilde{Y}), \{\widetilde{e}_k\}) \in Wh(\pi_1(X)).
    \]

    Given a cellular homotopy equivalence $f \colon X \to Y$ between finite connected CW complexes, its \emph{Whitehead torsion} is defined as
    \[
        \tau(f) := \tau(M_f, X),
    \]
    where $M_f$ is the mapping cylinder of $f$. We say that $f$ is a \emph{simple homotopy equivalence} if $\tau(f)$ is trivial.
\end{defn}

Chapman (\cite[Theorem 1]{Chapman}) showed that any homeomorphism between finite connected $CW$ complexes induces a simple homotopy equivalence between the cellular chain complexes. Also, the Whitehead torsion of a chain homotopy equivalence $f$ is invariant under chain homotopies of $f$. The following two propositions, which formalize this invariance and multiplicativity of torsion, can be found in \cite[Lemma 7.2]{Turaev}.

\begin{proposition} \label{prop:homotopy-invariance-torsion}
Let $f$ and $g$ be chain homotopy equivalences between based cochain complexes over a group ring $\Z G$. If $f$ and $g$ are chain homotopic, their Whitehead torsions agree:
\begin{equation}
    \tau(f) = \tau(g).
\end{equation}
\end{proposition}

Whitehead torsion is also multiplicative under composition of chain homotopy equivalences:

\begin{proposition} \label{prop:Whitehead-torsion-additive}
    Let $f:C\to C'$, $g:C'\to C''$ be chain homotopy equivalences between based cochain complexes $C,C',C''$ over a group ring $\Z G$. Then
\begin{equation}
    \tau(g\circ f)=\tau(g)\cdot \tau(f).
\end{equation}
\end{proposition}

There is an alternate definition of a simple homotopy equivalence, motivated from the ``simple expansion'' and ``simple collapse'' operations for CW complexes. We first introduce its algebraic counterpart.

\begin{defn} \label{def:simple-operations}
For based cochain complexes $(C_*,\{c_*^i\},\partial_*)$ over a group ring $\Z G$, we define three types of \emph{elementary simple operations}:
\begin{enumerate}
    \item Elementary expansion/retraction: we may take a direct sum with the short exact sequence 
    \begin{equation*}
        \begin{tikzcd}
        0 \arrow[r] &\Z G \arrow[r, "\id"] &\Z G \arrow[r] &0,
    \end{tikzcd}
    \end{equation*}
    shifted in any degree, and add a basis element at each according degree. The inverse operation of deleting such a direct summand is also allowed.
    \item Handle Slide: We may replace a basis element $c_i^\alpha$ by adding a linear sum of basis elements $c_i^\beta$ in the same degree to obtain a new basis element $c_i^\alpha+\sum_\beta g_\beta c_i^\beta$.
    \item Deck transformation: We may replace a basis element $c_\lambda^i$ with $g\cdot c_\lambda^i$, for some element $g\in G$.
\end{enumerate}
\end{defn}

We define the Whitehead group $Wh(G)$ to be the equivalence classes of acyclic based cochain complexes over $\Z G$ under the above three relations. A classic result in simple homotopy theory is that these two definitions of the Whitehead group agree (\cite[Theorem 8.7]{Turaev}). 

Now we turn to the geometric counterpart of the above discussion. The following is the geometric description of elementary expansions and retractions for CW complexes.

\begin{defn}
    Let $X$ be a finite CW complex, and suppose that $f:\partial D^k\to X$ is a continuous cellular map. Then we add a $k$-cell $e^k$ by the attaching map $f$, and further attach a $k+1$-cell by a map $g:\partial D^{k+1}\to X\cup e^k$ such that $g^{-1}(e^k)$ is the upper hemisphere of $S^k=\partial D^{k+1}$ and $g$ restricted to $g^{-1}(e^k)$ is a homeomorphism. Call the resulting CW complex $Y$, and define the inclusion $X\xhookrightarrow{}Y$ to be an \emph{elementary expansion}. The homotopy inverse of the above operation is called an \emph{elementary collapse}.
\end{defn}

One can check that elementary expansions and collapses are simple homotopy equivalences. The converse also holds:

\begin{proposition} [22.2, \cite{Cohen}] \label{prop:CW_cpx_he_simple}
    A cellular homotopy equivalence $f:X\to Y$ between finite CW complexes has trivial Whitehead torsion if and only if it can be decomposed into a finite composition of elementary expansions and collapses.
\end{proposition}

To end the subsection, we prove three lemmata about the Whitehead torsion of a cochain complex of the form $(C_*,\partial_*)\otimes_\Z D_*$, where $D_*$ is a based cochain complex over $\Z G$.

\begin{lemma} \label{lem:zero-torsion}
    Let $(C_*,\partial_*,\{c_i\})$ be an acyclic based cochain complex over $\Z$, and let $G$ be any group. Then the base-changed cochain complex $(\Tilde{C}_*,\Tilde{\partial}_*,\{\Tilde{c_i}\})=C_*\otimes_\Z\Z G$ over $\Z G$ is acyclic and has trivial Whitehead torsion.
\end{lemma}
\begin{proof}
    The Whitehead torsion is functorial under base change: the map $\{e\}\to G$ induces a group homomorphism $Wh(\{e\})\to Wh(G)$. Since $Wh(\{e\})$ is trivial, the torsion of $C_*$ vanishes, and therefore the torsion of $\Tilde{C}_*$ also vanishes. In other words, the matrix representing the differentials $\Tilde{\partial}_i$ with respect to the basis $\{\Tilde{c_i}\}$ has entries in $\Z\subset\Z G$.
\end{proof}

The argument generalizes to complexes of the form $C_*\bigotimes D_*$, where $C_*$ is an acyclic based cochain complex over $\Z$, and $D_*$ is any based cochain complex over $\Z G$.

\begin{lemma} \label{lem:zero-torsion-general}
        Let $(C_*,\partial^C_*,\{c_i\})$ be an acyclic based cochain complex over $\Z$, and let $(D_*,\partial^D_*,\{d_i\})$ be a based cochain complex over $\Z G$ for some group $G$. Then the based cochain complex
        \begin{equation}
            (C_*\otimes D_*, \partial^C_*\otimes\id+(-1)^{\deg}\id\otimes\partial^D_*,\{c_i\otimes d_j\})
        \end{equation}
        is acyclic and has trivial Whitehead torsion.
    \end{lemma}
    \begin{proof}
        The idea is to first prove the statement for the case when $C_*$ is a two-step complex $C_0\to C_1$, then reduce the general case to this case. 

        For the two-step case, suppose $C_*$ is isomorphic to 
        \begin{equation}
        \begin{tikzcd}
            0 \arrow[r] &C_0\cong\Z^r \arrow[r,"A"] &C_1\cong\Z^r \arrow[r] &0,
        \end{tikzcd}
        \end{equation}
        with identifications given by the preferred basis of $C_*$. Then the tensor product cochain complex $C_*\otimes D_*$ admits a filtration whose graded pieces are also two-step based cochain complexes
    \begin{equation}
        \mathrm{Gr}\,\mathcal{F}^j(C_*\otimes D_*) \cong 
    \left[ C_0\otimes D_j \xrightarrow{A\otimes\mathrm{id}} C_1\otimes D_j \right].
    \end{equation}
    With respect to the basis $\{c_i \otimes d_j\}$, this map is represented by a matrix with entries in $\Z\subset\Z G$, which has trivial class in $\overline{K_1}(\Z G)$. Each graded piece $\mathrm{Gr}\,\mathcal{F}^j(C_*\otimes D_*)$ is acyclic with trivial Whitehead torsion, and thus the filtration lemma (Proposition \ref{prop:Whitehead-torsion-filtration}) shows that the same holds for $(C_*\otimes D_*,\{c_i\otimes d_j\})$.
    
    We now turn to the general case. Since $C_*$ is an acyclic cochain complex over $\Z$, one can see from the short exact sequences
    \begin{equation}
    \begin{tikzcd}
        0 \arrow[r] &B_i \arrow[r] &C_i \arrow[r] &B_{i+1} \arrow[r] &0
    \end{tikzcd}   
    \end{equation}
    that each $B_i=\im (\partial_{i-1})$ is a projective $\Z$-module and thus is free. Therefore we may choose a basis $b_i$ for each $B_i$, and obtain a new basis $\{b_{i-1}b_i\}$ for $C_*$. Since $Wh(*)$ is trivial, \cite[Theorem 8.7]{Turaev} and Definition \ref{def:simple-operations} provides a sequence of elementary simple operations relating the two based acyclic cochain complexes $(C_*,\{c_i\})$ and $(C_*,\{b_{i-1}b_i\})$. Tensoring with $D_*$ induces a sequence of elementary simple operations between $(C_*\otimes D_*,\{c_i\otimes d_j\})$ and $(C_*\otimes D_*,\{b_{i-1}b_i\otimes d_j\})$, and thus their Whitehead torsions agree.

    Therefore, it is enough for us to show that the based cochain complex
     \begin{equation}
            (C_*\otimes D_*, \partial^C_*\otimes\id+(-1)^{\deg}\id\otimes\partial^D_*,\{b_{i-1}b_i\otimes d_j\})
        \end{equation}
    is acyclic with trivial Whitehead torsion. By the previous argument, we can find a filtration whose graded pieces are isomorphic to
    \begin{equation}
        \mathrm{Gr}\,\mathcal{F}^i(C_*\otimes D_*) \cong \left[ B_i \xrightarrow{\mathrm{id}} B_i \right] \otimes D_*,
    \end{equation}
    each of which is an acyclic based cochain complex with trivial Whitehead torsion by the two-step case. The filtration lemma then implies that $(C_*\otimes D_*,\{b_{i-1}b_i\otimes d_j\})$ is acyclic with trivial Whitehead torsion.
\end{proof}

We also show that tensoring a simply acyclic based cochain complex with a cochain complex over $\Z$ preserves simple acyclicity.

\begin{lemma} \label{lem:zero-torsion-tensor-complex}
    Let $(C_*,\partial_*^C,\{c_i\})$ be a based cochain complex over $\Z$, and let $(D_*,\partial^D_*,\{d_i\})$ be a simply acyclic based complex over $\Z G$ for some group $G$. Then the tensor product complex
     \begin{equation}
        (C_*\otimes D_*, \partial^C_*+(-1)^{\deg}\partial^D_*,\{c_i\otimes d_j\})
    \end{equation}
    is acyclic and has trivial Whitehead torsion.
\end{lemma}
\begin{proof}
    This time, filter $C_*$ by degree, such that the associated graded piece $F^iC_*/F^{i+1}C_*$ is concentrated in degree $i$ and isomorphic to $C_i$. Then the induced filtration on $C_*\otimes D_*$ has associated graded pieces
    \begin{equation}
        gr^i(C_*)\otimes D_*\cong C_i\otimes D_*
    \end{equation}
    with differential $\pm\partial^d_*$ and basis $\{c_i\otimes d_*\}$. Since $D_*$ is acyclic with trivial Whitehead torsion, so is each complex $C_i\otimes D_*$ with the induced basis. 
    
    Therefore, each associated graded piece of the filtered complex $C_*\otimes D_*$ is acyclic with trivial Whitehead torsion. The result now follows from the filtration lemma for Whitehead torsion (Proposition \ref{prop:Whitehead-torsion-filtration}).
\end{proof}

\subsection{Reidemeister torsion} \label{ssec:R-torsion}

In this subsection, we recall the definition of Reidemeister torsion. Like Whitehead torsion, it is defined for based cochain complexes. When the complex is acyclic, Reidemeister torsion coincides with the image of the corresponding Whitehead torsion under the natural functorial maps. However, unlike Whitehead torsion, Reidemeister torsion can also be defined for non-acyclic based complexes. This flexibility is useful in distinguishing simple homotopy types of based complexes.

Let $(C_i,\partial_i,\{c_i\})$ be a based cochain complex over a group ring $\Z G$, which need not be acyclic, and let $\rho: \Z G\to \mathbb{F}$ be a ring homomorphism to a field $\mathbb{F}$. The map $\rho$ induces a chain complex $(C_i\otimes_{\Z G}\mathbb{F},\partial_i\otimes\id)$, and the basis $\{c_i\}$ induces a basis $\{c_i\otimes 1\}$ of $C_i\otimes_{\Z G}\mathbb{F}$ as an $\mathbb{F}$-vector space.

\begin{defn}\label{def:R-torsion}
    Suppose that $(C_i,\partial_i,\{c_i\})$ is a based complex over $\Z G$ and there exists a ring homomorphism $\rho:\Z G\to\mathbb{F}$ such that $(C_i\otimes_{\Z G}\mathbb{F},\partial_i\otimes\id)$ is acyclic. In such a case, we pick $\mathbb{F}$-bases $q_i$ for the images of $\partial_{i-1}\otimes\id$, and define the \emph{Reidemeister torsion} $\Delta_\rho(C_i,\partial_i,\{c_i\})$ as
    \begin{equation}
        \Delta_\rho(C_i,\partial_i,\{c_i\})=\prod_i
        ~\det([q_iq_{i-1}/(c_i\otimes 1)])^{(-1)^{i-1}}\in\mathbb{F}^\times/\{\pm\rho(G)\}.
    \end{equation}
    This definition is independent of the choices of bases $q_i$, as one can see that for a different choice of basis $q_i'$, the determinant of the change of basis matrix $[q_i'/q_i]$ is canceled in the alternating product.
\end{defn}

The Reidemeister torsion $\Delta_\rho$ depends on the choice of the ring homomorphism $\rho:\Z G\to\mathbb{F}$.

\begin{proposition} \label{prop:R-torsion-agrees-with-Wh}
    Suppose that $(C_i,\partial_i,\{c_i\})$ is an acyclic based complex over $\Z G$. Then for any ring homomorphism $\rho: \Z G\to \mathbb{F}$, the identity
\begin{equation}
    \Delta_\rho(C_i,\partial_i,\{c_i\}) = \rho_*(\tau(C_i,\partial_i,\{c_i\}))
\end{equation}
    holds, where $\rho_*: Wh(G) \to \mathbb{F}^\times/\{\pm\rho(G)\}$ denotes the map which first applies $\rho$ to each entry, then takes the determinant.
\end{proposition}

Though Reidemeister torsion is a weaker invariant than Whitehead torsion, we note in the following proposition that Reidemeister torsion is also a simple homotopy invariant of based cochain complexes.

\begin{proposition}[Corollary 9.2, \cite{Turaev}] \label{prop:R-torsion-simple-invariant}
    Suppose that $f:(C_*,\{c_i\})\to (D_*,\{d_j\})$ is a chain homotopy equivalence between two based cochain complexes over $\Z G$, where $C_*$ and $D_*$ need not be acyclic. Moreover, assume that $f$ is a simple homotopy equivalence. Now suppose that for a ring homomorphism $\rho:\Z G\to\mathbb{F}$ as before, the Reidemeister torsion $\Delta_\rho(C_*,\{c_i\})$ is well-defined. Then the chain complex $(D_*\otimes_{\Z G} \mathbb{F})$ is also acyclic, and the Reidemeister torsions of $C_*$ and $D_*$ agree:
    \begin{equation}
        \Delta_\rho(C_*,\{c_i\})=\Delta_\rho(D_*,\{d_j\}).
    \end{equation}
\end{proposition}
\begin{proof}
    Since $f:C_*\to D_*$ is a chain homotopy equivalence, the map
    \begin{equation}
        f\otimes_{\Z G}\mathbb{F}: C_*\otimes_{\Z G}\mathbb{F}\to D_*\otimes_{\Z G}\mathbb{F}
    \end{equation}
    is a chain homotopy, and thus induces an isomorphism on cohomology. Since the Reidemeister torsion $\Delta_\rho(C_*,\{c_i\})$ is well-defined, $C_*\otimes_{\Z G}\mathbb{F}$ is acyclic, and it follows that $D_*\otimes_{\Z G}\mathbb{F}$ is also acyclic.
    
    Now consider the mapping cone $cone(f)=(C[1]\oplus D, \partial_{cone})$ of $f$. Since $f$ was assumed to be a simple homotopy equivalence, $cone(f)$ is a based acyclic chain complex over $\Z G$ that has trivial Whitehead torsion. Then after taking the tensor product, $cone(f)\otimes_{\Z G}\mathbb{F}$ will again be acyclic, and have trivial Reidemeister torsion
    \begin{equation}
        \Delta_\rho(cone(f)) = 1.
    \end{equation}
    But since the mapping cone has a natural two-step filtration whose graded pieces are $C[1]$ and $D$ which are both based acyclic, we apply Proposition \ref{prop:torsion-additive} to conclude that
    \begin{equation}
        \Delta_\rho(C_*[1],\{c_i[1]\})\cdot\Delta_\rho(D_*,\{d_j\}) = \Delta_\rho(cone(f))=1.
    \end{equation}
    Since the shift operation to the chain complex satisfies 
    \begin{equation}
        \Delta_\rho(C_*[1],\{c_i[1]\})=\Delta_\rho(C_*,\{c_i\})^{-1},
    \end{equation}
    comparing these two equations, our proof is complete.
\end{proof}

\subsection{Reidemeister torsion for lens spaces}\label{ssec:lens-spaces}

Reidemeister torsion provides a complete classification of lens spaces up to diffeomorphism. In this subsection, we recall the computation of the Reidemeister torsion for three-dimensional lens spaces. 

A couple of remarks are in order before we begin the computation. We will use cellular chain complexes rather than cochain complexes, justified by Milnor's duality theorem for Reidemeister torsion \cite[Theorem 1]{MilnorDuality}. By results of Moise \cite[Theorems 3, 4]{Moise}, every 3-manifold admits a triangulation, and any two such triangulations admit a common simplicial subdivision. Thus we may choose a convenient CW structure for the computation. More generally, Chapman \cite[Theorem 1]{Chapman} shows that different CW structures on the same manifold of any dimension give cellular chain complexes that are simple homotopy equivalent, and therefore have the same Reidemeister torsion. For CW structures arising from Morse functions, the invariance of Reidemeister torsion can be seen by following a generic one-parameter family of Morse functions, where only handle slides and birth-death bifurcations occur (see \cite[Chapter 8]{KreckLück}).

\begin{defn} \label{def:lens-space}
    The \emph{three-dimensional lens spaces} $L(p,q)$ are defined for $(p,q)$ coprime as quotients of $S^3=\{(z_1,z_2)\mid \lvert z_1\rvert^2+\lvert z_2\rvert^2=1\} \subset \C^2$ under the group action of $\Z/p\subset \U(1)$ given by
    \begin{equation}
        \xi\cdot(z_1,z_2)=(e^{2\pi i/p}z_1,e^{2\pi iq/p}z_2),
    \end{equation}
    where $\xi$ is the generator of the group $\Z/p$.
\end{defn}


The following classification of lens spaces is due to Reidemeister \cite{Reidemeister}:
 
\begin{thm}[Classification of lens spaces] \label{thm:lens-space-classification}
\begin{enumerate}
    \item The lens spaces $L(p,q)$ and $L(p,q')$ are homotopy equivalent iff $qq'\equiv \pm m^2~(mod~p)$ for some integer $m$.
    \item The lens spaces $L(p,q)$ and $L(p,q')$ are simple homotopy equivalent iff $q'\equiv \pm q^{\pm1}~(mod~p)$. In this case, $L(p,q)$ and $L(p,q')$ are actually diffeomorphic.
\end{enumerate}    
\end{thm}

We will briefly review the proof of the second statement by computing the Reidemeister torsion of the lens space $L(p,q)$. Details can be found in \cite[Section 10]{Turaev}.

First, we construct a CW structure on $L(p,q)$ by constructing a cell structure on $S^3$ that is $\Z/p$-equivariant with respect to the $\Z/p$-action defining $L(p,q)$. We regard $S^3=\{(z_1,z_2)\mid \lvert z_1\rvert^2+\lvert z_2\rvert^2=1\}$ as a subset of $\C^2$, and define two subsets
\begin{equation}
    T_0=\{(z_1,z_2)\mid z_2=0\},~T_1=\{(z_1,z_2)\mid z_1=0\}.
\end{equation}
Let $\zeta=\exp(\frac{2\pi i}{p})$ be a $p$th root of unity, and for each $0\leq j\leq p-1$, we define $I_j$ to be the shorter segment in $T_0$ connecting $(\zeta^j,0)$ and $(\zeta^{j+1},0)$ in $T_0$, and similarly define $I_j'$ to be the segment connecting $(0,\zeta^j)$ and $(0,\zeta^{j+1})$ in $T_1$. Now we define the cells
\begin{align}
    e_j^0 &= \{(\zeta^j,0)\}\in T_0, \\ 
    e_j^1 &= \{(z_1,0)\mid z_1\in I_j\}\subset T_0,\\
    e_j^2 &= \{(z_1,t\zeta^j)\mid 0\leq t\leq 1, ~\lvert z_1\rvert^2+t^2=1\},\\
    e_j^3 &= \{(z_1,z_2)\mid z_2\in I_j', ~\lvert z_1\rvert^2+\lvert z_2\rvert^2=1\}.
\end{align}

This cell decomposition descends to a cell decomposition of $L(p,q)$ with 1 cell for each dimension.


The boundary maps in the cellular complex can be computed for each $0\leq j\leq p-1$ to be
\begin{align}
    \partial e_j^0 &= 0,\\
    \partial e_j^1 &= e_{j+1}^0-e_j^0,\\
    \partial e_j^2 &= e_0^1+e_1^1+\cdots+e_{p-1}^1,\\
    \partial e_j^3 &= e_{j+1}^2-e_j^2,
\end{align}
when we consider orientations.

After fixing lifts of the cells $e_0,e_1,e_2,e_3$ in $L(p,q)$ to the universal covers to be $e_0^0,e_1^0,e_2^0,e_3^0$, the cellular chain complex $C^{cell}_*(S^3)$ of $S^3$ can be identified with a based chain complex over the group ring $\Z[\Z/p]$ given by
\begin{equation} \label{eqn:R-torsion-L(p,q)}
\begin{tikzcd}
    \Z[\Z/p]\langle e_3^0\rangle \arrow[r,"1-\xi^r"] &\Z[\Z/p]\langle e_2^0\rangle \arrow[rr,"1+\xi+\cdots+\xi^{p-1}"] &&\Z[\Z/p]\langle e_1^0\rangle \arrow[r,"1-\xi"] &\Z[\Z/p]\langle e_0^0\rangle,
\end{tikzcd}
\end{equation}
where $r$ is an integer that satisfies $qr\equiv 1~(mod~p)$. This chain complex is \emph{not} acyclic over $\Z[\Z/p]$, but becomes acyclic after base change. Consider the ring homomorphism $\rho:\Z[\Z/p]\to\C$ sending the generator $\xi$ of $\Z/p$ to $\zeta$, the $p$th root of unity $e^{2\pi i/p}$. Then, $C^{cell}_*(S^3)\otimes_\rho\C$ can be written as 
\begin{equation}
\begin{tikzcd}
    \C\langle e_3^0\rangle \arrow[rr,"1-\zeta^r"] &&\C\langle e_2^0\rangle \arrow[r,"0"] &\C\langle e_1^0\rangle \arrow[rr,"1-\zeta"] &&\C\langle e_0^0\rangle
\end{tikzcd}
\end{equation}
which is acyclic, and the Reidemeister torsion of this acyclic based complex can be computed to be 
\begin{equation}
    \Delta_\rho(L(p,q))=(1-\zeta^r)(1-\zeta)\in\C^\times/(\pm \zeta^k).
\end{equation}
Since a homeomorphism is always a simple homotopy equivalence (\cite[Theorem 1]{Chapman}), and simple homotopy equivalences preserve Reidemeister torsion, we can compare the Reidemeister torsions of $L(p,q)$ and $L(p,q')$ to detect whether they are homeomorphic.

One subtlety is that the computation of Reidemeister torsion depends on a choice of fixed isomorphism $\pi_1(L(p,q))\cong\Z/p$. Even if we fix identifications $\pi_1(L(p,q))\cong\pi_1(L(p,q'))\cong\Z/p$, a homeomorphism $\phi:L(p,q)\to L(p,q')$ may induce the map $t\mapsto t^d$ on fundamental groups for some $d\in\Z$.

This affects the induced representation: a ring homomorphism $\rho:\Z[\pi_1(L(p,q))]\to\C$ sending $t\mapsto\zeta$ corresponds under $\phi$ to a new representation $\rho':\Z[\pi_1(L(p,q'))]\to\C$ under $\phi$, sending $t\mapsto \zeta^d$. Hence, to distinguish $L(p,q)$ from $L(p,q')$, we must compare the Reidemeister torsions $\Delta_\rho(L(p,q))$ and $\Delta_{\rho'}(L(p,q'))$ for all such $d$.

Although this adds a minor complication to the argument, one can still show that 
\begin{align}
    \Delta_\rho(L(p,q))=\Delta_{\rho'}(L(p,q')) &\iff (1-\zeta^{dr)}(1-\zeta^d)=\pm\zeta^k(1-\zeta^{r'})(1-\zeta)\in\C,~~\exists k,d\in\Z\\
    &\iff q'\equiv\pm q^{\pm1}
\end{align}
by comparing the real parts of the possible values.

The homotopy classification of lens spaces is a little more involved: we refer to \cite[(29.6)]{Cohen} for the proof, which actually gives a homotopy classification for lens spaces in all dimensions. As a corollary, we obtain examples of lens spaces that are homotopy equivalent but not simply so.

\begin{corollary} \label{cor:lens-spaces-homotopic-not-simply}\begin{enumerate}
    \item The two lens spaces $L(7,1)$ and $L(7,2)$ are homotopy equivalent, but not simple homotopy equivalent.
    \item The three lens spaces $L(17,1)$, $L(17,2)$, and $L(17,4)$ are all homotopy equivalent, but no two of these are simple homotopy equivalent.
\end{enumerate}
\end{corollary}


Finally, we consider the following situation, where two manifolds $Q_1\cong L(7,1)$ and $Q_2\cong L(7,2)$ are embedded as submanifolds in a manifold $X$ whose fundamental group $\pi_1(X)$ is isomorphic to the free product $\Z/7*\Z/7$. Denote the inclusions by $\iota_1:Q_1\xhookrightarrow{}X$, $\iota_2:Q_2\xhookrightarrow{} X$, and assume that the inclusion $\iota_2$ induces an injection $\pi_1(L(7,2))\cong\Z/7\xhookrightarrow{}\pi_1(X)$ into the second factor of the free product.

Endow $X$ with a CW structure such that both $Q_1$, $Q_2$ are subcomplexes. We choose a universal cover $\pi:\Tilde{X}\to X$, and denote by $\Tilde{Q_1}, \Tilde{Q_2}$ the lifts of $Q_1,Q_2$. After fixing preferred lifts of cells, we may identify the cellular cochain complexes of $\Tilde{Q_1},\Tilde{Q_2}$ with the cellular cochain complexes of $Q_1,Q_2$ with $\Z\pi_1(X)$-coefficients, using the bases given by these lifts. 

\begin{proposition}\label{prop:R-torsion-7*7}
    In the setup as above, the based cochain complexes $C^*_{cell}(Q_1;\Z\pi_1(X))$, $C^*_{cell}(Q_2;\Z\pi_1(X))$ are not simple homotopy equivalent. In particular, there exists a ring homomorphism $\rho:\Z\pi_1(X)\to\C$ such that the associated Reidemeister torsions
    \begin{equation}
        \Delta_\rho(C^*_{cell}(Q_1;\Z\pi_1(X))), \Delta_\rho(C_{cell}^*(Q_2;\Z\pi_1(X)))
    \end{equation}
    are distinct.
\end{proposition}
\begin{proof}
Suppose for contradiction that there exists a simple homotopy equivalence
\begin{equation}\label{eqn:7*7-contradiction}
    \Phi:C^*_{cell}(Q_1;\Z\pi_1(X))\to C^*_{cell}(Q_2;\Z\pi_1(X)).
\end{equation}
Let $\pi_1(X) = \langle \eta \rangle * \langle \nu \rangle \cong \Z/7 * \Z/7$ for the two generators $\eta,\nu$, and define a ring homomorphism $\rho: \Z\pi_1(X) \to \C$ by setting $\rho(\eta) = \rho(\nu) = \zeta = e^{2\pi i/7}$. Then $\rho$ maps reduced words in $\pi_1(X)$ to powers of $\zeta$, and restricts to the standard representation on each $\Z/7$ factor.

We first compute the Reidemeister torsion of $Q_2$ with respect to $\rho$. The cochain complex $C^*_{cell}(Q_2;\Z\pi_1(X))$ corresponds to the standard cell structure on the lens space $L(7,2)$, with $\pi_1(Q_2) \subset \pi_1(X)$ identified with the subgroup generated by $\nu$. Therefore, we may identify the cochain complex $C^*_{cell}(Q_2;\Z\pi_1(X))\otimes_\rho\C$ with the corresponding one for $L(7,2)$:
\begin{equation} \label{eqn:R-torsion-L(7,2)}
    \begin{tikzcd}
        \C \arrow[r, "1-\zeta"] &\C \arrow[r, "0"] &\C \arrow[r,"1-\zeta^4"] &\C
    \end{tikzcd}
\end{equation}
whose Reidemeister torsion is $\Delta_\rho(Q_2) = (1 - \zeta)(1 - \zeta^4)$.

To compute the Reidemeister torsion of $Q_1 \cong L(7,1)$, let the embedding $\iota:Q_1\xhookrightarrow{}X$ send a fixed generator $\gamma \in \pi_1(Q_1)$ to $\rho(\iota_*(\gamma)) = \zeta^l$ for some integer $l$. The resulting cochain complex $C^*_{cell}(Q_1;\Z\pi_1(X))\otimes_\rho\C$ is:
\begin{equation} \label{eqn:R-torsion-L(7,1)}
    \begin{tikzcd}
        \C \arrow[r, "1-\zeta^l"] &\C \arrow[rr, "1+\zeta^l+\cdots+\zeta^{6l}"] &&\C \arrow[r,"1-\zeta^l"] &\C
    \end{tikzcd},
\end{equation}
coming from the cell structure of $L(7,1)$. Since we assumed (\ref{eqn:7*7-contradiction}) to be a simple homotopy equivalence, the cohomology of the cochain complex (\ref{eqn:R-torsion-L(7,1)}) must vanish, and so it follows that $\zeta^l\neq 1$ and $1+\zeta^l+\cdots+\zeta^{6l}=0$. Therefore its Reidemeister torsion can be computed as $\Delta_\rho(Q_1) = (1 - \zeta^l)^2$.

If the complexes were simple homotopy equivalent, their Reidemeister torsions would agree in $\C^\times/(\pm\zeta^k)$. But by the classification of lens spaces up to simple homotopy equivalence, for all $l = 1,\dots,6$, we have:
\begin{equation}
    (1 - \zeta)(1 - \zeta^4) \neq \pm \zeta^k (1 - \zeta^l)^2 \in \C^\times.
\end{equation}
Hence, $\Delta_\rho(Q_1) \neq \Delta_\rho(Q_2)$, contradicting the existence of a simple homotopy equivalence. 
\end{proof}

\section{Fukaya categories}

\subsection{Categorical notions} \label{ssec:category-theory}

In this subsection, we briefly review $A_\infty$-categorical notions, following the notation of \cite[Chapter 1]{Seidel08}. Our main goal is to define the category of twisted complexes twisted by cochain complexes $Tw_{Ch}\mathcal{C}$, and describe some of its properties.

Our setup is an $A_\infty$-category $\mathcal{C}$, whose morphism spaces are $\Z$-graded cochain complexes of free $R$-modules of finite rank, bounded in both directions. We assume that $\mathcal{C}$ is cohomologically unital, meaning that for every object $L\in\mathcal{C}$, there exists an identity element $e_L$ in $H^*\hom_\mathcal{C}(L,L)$.

\begin{defn} \label{def:acyclic-obj}
    We define an object $L\in\mathcal{C}$ to be \emph{left acyclic} if for any other object $K\in\mathcal{C}$, the cochain complex $\hom_\mathcal{C}(L,K)$ is acyclic. Right acyclic objects are similarly defined. 
\end{defn}


\begin{defn} \label{def:equivalence}
    An \emph{isomorphism element} in an $A_\infty$-category $\mathcal{C}$ is a cocycle $\alpha\in\hom^0_{\mathcal{C}}(K,L)$ for which there exists a cocycle $\beta\in\hom^0_\mathcal{C}(L,K)$ such that the identities
    \begin{equation}
        \mu^2([\beta],[\alpha])=e_K, \mu^2([\alpha],[\beta])=e_L
    \end{equation}
    hold in the cohomological category $H^*(\mathcal{C})$. When such two elements $\alpha,\beta$ exist, we define $K$ and $L$ to be \emph{isomorphic objects} in $\mathcal{C}$.
\end{defn}

Since an isomorphism element $\alpha\in\hom^0_\mathcal{C}(K,L)$ induces isomorphisms between the cohomology groups
\begin{equation}
    H^0\hom_\mathcal{C}(L,K)\cong H^0\hom_\mathcal{C}(K,K) \cong H^0\hom_\mathcal{C}(L,L),
\end{equation}
if either $H^0\hom_\mathcal{C}(K,K)$ or $H^0\hom_\mathcal{C}(L,L)$ is a free $R$-module of rank 1, the cohomology classes of the isomorphism elements $\alpha$ and $\beta$ are uniquely determined up to multiplication by a unit in $R$. In particular, if $R=\Z$, such isomorphisms are uniquely determined up to sign.

Definitions of $A_\infty$-(bi)modules, $A_\infty$-(bi)module homomorphisms and pre-homomorphisms, Yoneda modules, as well as triangulated $A_\infty$-categories and related notions can be found in \cite[Chapter 1]{Seidel08}. We now recall the definition of a twisted complex.

\begin{defn} \label{def:twisted-cpx}
    A \emph{twisted complex} $\mathcal{K}=(I,\{K_i\},\{V_i\})$ in an $A_\infty$-category $\mathcal{C}$ consists of a finite ordered index set $I$, a collection of objects $K_i\in\mathcal{C}$ indexed by $i\in I$, graded free $R$-modules $V_i$ of finite rank indexed by $i\in I$, and a collection of morphisms $\delta_\mathcal{K}=\{\delta_{ij}\in\hom^1(V_i\otimes K_i,V_j\otimes K_j)\}$ such that $\delta_{ij}=0$ if $i\leq j$, and satisfy the following Maurer-Cartan equation:
    \begin{equation}
        \sum_{m\geq1} \mu^m(\delta_{\mathcal{K}},\cdots,\delta_{\mathcal{K}},\delta_{\mathcal{K}})=0,
    \end{equation}
    which is shorthand for the equation
    \begin{equation}
        \sum_{m\geq1}\sum_{p=i_0<\cdots<i_m=q}\mu^m(\delta_{i_{m-1}i_m},\delta_{i_{m-2}i_{m-1}},\cdots,\delta_{i_0i_1})=0
    \end{equation}
    for all $(p,q)$. Because of the lower triangular condition for $\delta_\mathcal{K}$, the above sum is finite.
\end{defn}

When all the $R$-modules $V_i$ are rank $1$, in which case they simply record the degree shift of the objects $K_i$, the twisted complex may be written more concisely as:
\begin{equation}
    \mathcal{K}=[K_0[d_0]\to K_1[d_1]\to\cdots\to K_n[d_n]].
\end{equation}

To define the $A_\infty$-category $Tw\mkern3mu\mathcal{C}$ of twisted complexes in $\mathcal{C}$, we describe the morphisms between twisted complexes and their $A_\infty$-operations. 

\begin{defn} \label{def:twisted-cpx-hom}
    Given two twisted complexes $\mathcal{K}=\bigoplus_i V_i\otimes K_i$, $\mathcal{L}=\bigoplus_j W_j\otimes L_j$, the morphism space is defined as
    \begin{equation}
    \hom_{Tw\mkern2mu\mathcal{C}}~(\bigoplus_i V_i\otimes K_i,\bigoplus_j W_j\otimes L_j) = \bigoplus_{i,j} \hom_R(V_i,W_j)\otimes\hom_\mathcal{C}(K_i,L_j).
\end{equation}
    The $A_\infty$-structure is defined by inserting the differentials $\delta$ wherever possible:
    \begin{equation}
    \mu^k_{Tw\mathcal{C}}(x_k,\dots,x_1)
    = \sum_{i_0,\dots,i_k}
    \mu^{k+i_0+\cdots+i_k}\big(
        \overbrace{\delta_{k},\dots,\delta_{k}}^{i_k},
        x_k,
        \overbrace{\delta_{k-1},\dots,\delta_{k-1}}^{i_{k-1}},
        \dots,
        x_1,
        \overbrace{\delta_{0},\dots,\delta_{0}}^{i_0}
    \big).
\end{equation}
\end{defn}

In the simplified case where all $V_i$ and $W_j$ are rank 1 and record degree shifts $[d_i]$ and $[e_j]$, the morphism space becomes
\begin{equation}
    \hom_{Tw\mkern2mu\mathcal{C}}(\bigoplus_i K_i[d_i],\bigoplus_j L_j[e_j])=\bigoplus_{i,j} \hom_\mathcal{C}(K_i,L_j)[e_j-d_i].
\end{equation}



We note that the notions of left and right acyclicity naturally extend to twisted complexes. Importantly, to check whether a twisted complex $\mathcal{K}$ is left acyclic, it suffices to verify that
\begin{equation}
    \hom_{Tw\mkern2mu\mathcal{C}}(\mathcal{K},L)
\end{equation}
is acyclic for all $L\in \mathcal{C}$, rather than all twisted complexes $\mathcal{L}\in Tw\mathcal{C}$. Similarly, isomorphism elements between twisted complexes are defined analogously to those in $\mathcal{C}$. 

Recall that if $\mathcal{C}$ is cohomologically unital, so is $Tw\mathcal{C}$.

\begin{defn}
    An \emph{isomorphism element} $\alpha\in \hom^0_{Tw\mathcal{C}}(\mathcal{K},\mathcal{L})$ is a degree 0 cocycle such that there exists a cocycle $\beta\in \hom^0_{Tw\mathcal{C}}(\mathcal{L},\mathcal{K})$ satisfying
    \begin{equation}
        [\mu^2_{Tw\mathcal{C}}(\beta,\alpha)]=[e_\mathcal{K}],~[\mu^2_{Tw\mathcal{C}}(\alpha,\beta)]=[e_\mathcal{L}]
    \end{equation}
    holds in the cohomological category $H^*Tw\mathcal{C}$.
\end{defn}


We now introduce a variant of the category of twisted complexes for an $A_\infty$-category $\mathcal{F}$, under the standing assumptions that $\mathcal{F}$ is c-unital and its morphism spaces are $\Z$-graded cochain complexes over $\Z$, free in each degree. In this variant, objects are twisted complexes twisted by cochain complexes over $\Z$ (rather than free $\Z$-modules), and we denote the resulting $A_\infty$-category by $Tw_{Ch}\mathcal{F}$. Unless otherwise specified, all the twisted complexes in the subsequent sections are understood to lie in the category $Tw_{Ch}\mathcal{F}$.

To begin, we construct an enlarged category $\Sigma_{Ch}\mathcal{F}$. The objects of $\Sigma_{Ch}\mathcal{F}$ are formal direct sums of the form
\begin{equation}
    Ob(\Sigma_{Ch}\mathcal{F})=\{\bigoplus_{\alpha\in A} C_\alpha^*\otimes L_\alpha\},
\end{equation}
where each $L_\alpha$ is an object in $\mathcal{F}$, each $C_\alpha^*$ is a cochain complex of finitely generated $\Z$-modules supported in finitely many degrees, and the index set $A$ is finite. The morphism spaces are defined as
\begin{equation}
    \hom_{\Sigma_{Ch}\mathcal{F}} (\bigoplus_\alpha C^*_\alpha\otimes L_\alpha,\bigoplus_\beta D^*_\beta\otimes K_\beta)=\bigoplus_{\alpha,\beta} \hom_\Z(C^*_\alpha,D^*_\beta)\otimes \hom_\mathcal{F}(L_\alpha,K_\beta)
\end{equation}
where $\hom_\Z(C^*_\alpha,D^*_\beta)$ denotes the complex of all $\Z$-linear maps of graded degree between $C^*_\alpha$ and $D^*_\beta$. 

Each morphism space naturally carries a differential defined on pure tensors as
\begin{equation}
    \mu^1(\phi\otimes x)= (-1)^{\lvert x\rvert-1}\partial_{Ch}(\phi)\otimes x + \phi\otimes \mu^1_{\mathcal{F}}(x),
\end{equation}
where $\partial_{Ch}(\phi)=\partial_D\circ\phi-(-1)^{\deg(\phi)}\phi\circ\partial_C$ is the differential from the DG category of cochain complexes. 

Higher $A_\infty$-operations are defined analogously to those in the additive enlargement from the definition of $Tw\mathcal{C}$, and can be written explicitly as
\begin{equation}\label{eqn-signconvention}
    \mu^k(\phi_k\otimes x_k,\cdots,\phi_1\otimes x_1)= (-1)^{\bowtie} (\phi_k\circ\cdots\circ\phi_1) \otimes \mu^k_\mathcal{F}(x_k,\cdots,x_1).
\end{equation}
where $\bowtie=\sum_{i<j}\deg(\phi_i)(\deg(x_j)-1)$ as in \cite[Equation 3.17]{Seidel08}.

We now equip each object of $\Sigma_{Ch}\mathcal{F}$ with a differential. Define $\Xi \mathcal{F}$ to be the category whose objects are pairs
\begin{equation}
    (\mathcal{L}=\bigoplus_\alpha C^*_\alpha\otimes L_\alpha, \delta_\mathcal{L}),
\end{equation}
and $\delta_\mathcal{L}$ is a degree 1 homomorphism in $\hom_{\Sigma_{Ch}\mathcal{F}}(\mathcal{L},\mathcal{L})$. Each such object is also equipped with a finite length decreasing filtration
\begin{equation}
    \mathcal{L}=F^0\mathcal{L}\supset F^1\mathcal{L} \supset\cdots\supset F^n\mathcal{L}=0
\end{equation}
such that the induced map of $\delta_\mathcal{L}$ on each graded piece $F^i\mathcal{L}/F^{i+1}\mathcal{L}$ vanishes.

This category $\Xi\mathcal{F}$ admits the structure of a curved $A_\infty$-category as follows: first, the curvature of an object is defined as
\begin{equation}
    \mu^0(\mathcal{L},\delta_\mathcal{L})=\sum_{k\geq1} \mu^k_{\Sigma_{Ch}\mathcal{F}}(\delta_\mathcal{L},\cdots,\delta_\mathcal{L}).
\end{equation}
The morphism spaces are inherited from $\Sigma_{Ch}\mathcal{F}$:
\begin{equation}
    \hom_{\Xi\mathcal{F}}(\mathcal{L}_0,\mathcal{L}_1)=\hom_{\Sigma_{Ch}\mathcal{F}}(\mathcal{L}_0,\mathcal{L}_1).
\end{equation}
The $A_\infty$-structure maps are defined by inserting $\delta_\mathcal{L}$ whenever possible:
\begin{equation}
    \mu^k_{\Xi\mathcal{F}}(z_k,\cdots,z_1)=\sum_{i_0,\cdots,i_k}\mu_{\Sigma_{Ch}\mathcal{F}}^{k+i_0+\cdots+i_k}\big(~\overbrace{\delta_{\mathcal{L}_k},\cdots,\delta_{\mathcal{L}_k}}^{i_k},x_k,\overbrace{\delta_{\mathcal{L}_{k-1}},\delta_{\mathcal{L}_{k-1}},\cdots,\delta_{\mathcal{L}_{k-1}}}^{i_{k-1}},\cdots,x_1,\overbrace{\delta_{\mathcal{L}_0},\cdots,\delta_{\mathcal{L}_0}}^{i_0}~\big).
\end{equation}

\begin{defn}
    The \emph{$A_\infty$-category of twisted complexes twisted by cochain complexes}, denoted by $Tw_{Ch}\mathcal{F}$, is defined to be the full subcategory of $\Xi\mathcal{F}$ consisting of all objects $(\mathcal{L},\delta_\mathcal{L})$ whose curvature term $\mu^0$ vanishes:
    \begin{equation}
        \sum_{k\geq1}\mu^k(\delta_\mathcal{L},\cdots,\delta_\mathcal{L})=0.
    \end{equation}
\end{defn}

We claim that $Tw_{Ch}\mathcal{F}$ is cohomologically unital.

\begin{proposition}
    If $\mathcal{F}$ is c-unital, then so is $Tw_{Ch}\mathcal{F}$.
\end{proposition}
\begin{proof}
    Our first step is to adapt the usual Yoneda embedding arguments for field coefficients to the setting with $\Z$-coefficients. As shown in \cite{PomerleanoSeidel}, there exists a cohomologically fully faithful embedding of $\mathcal{F}$ into a strictly unital $A_\infty$-category $\mathcal{A}$ defined over $\Z$. Then $Tw_{Ch}\mathcal{A}$ has strict units, given for an object $\mathcal{L}=\bigoplus_\alpha C^*_\alpha\otimes L_\alpha$ by
    \begin{equation}
        e_\mathcal{L}=(\id_{C^*_\alpha}\otimes e_{L_\alpha})\in\hom_{Tw_{Ch}\mathcal{A}}(\mathcal{L},\mathcal{L}).
    \end{equation}
    Furthermore, the filtration argument in \cite[Lemma 3.23]{Seidel08} extends to $Tw_{Ch}$, so the embedding $\mathcal{F}\xhookrightarrow{}\mathcal{A}$ induces a cohomologically fully faithful functor $Tw_{Ch}\mathcal{F}\to Tw_{Ch}\mathcal{A}$. Since $Tw_{Ch}\mathcal{A}$ is strictly unital, it follows that $Tw_{Ch}\mathcal{F}$ is c-unital.
\end{proof}

Having established cohomological unitality of $Tw_{Ch}\mathcal{F}$, we now show that $Tw_{Ch}\mathcal{F}$ is triangulated. 

First, define the shift of an object by
\begin{equation}
    (C^*\otimes L)[1]\coloneqq C^*[1]\otimes L,
\end{equation}
which extends to a shift functor $S:Tw_{Ch}\mathcal{F}\to Tw_{Ch}\mathcal{F}$. Given a degree zero cocycle $c\in\hom_{Tw_{Ch}\mathcal{F}}(\mathcal{K},\mathcal{L})$, we define its mapping cone as the twisted complex
\begin{equation}
    (S\mathcal{K}\oplus\mathcal{L}, \delta=\begin{pmatrix}
        S(\delta_\mathcal{K}) &0\\
        -S(c) &\delta_\mathcal{L}
    \end{pmatrix}),
\end{equation}
where $S$ denotes the shift functor. Combining the filtrations of $\mathcal{
K}$ and $\mathcal{L}$, first from $\mathcal{K}\oplus\mathcal{L}$ to $0\oplus\mathcal{L}$, then from $0\oplus\mathcal{L}$ to $0\oplus 0$, we obtain a finite length filtration verifying that the mapping cone is an object of $Tw_{Ch}\mathcal{F}$. Moreover, its image under the Yoneda embedding fits into an exact triangle involving the Yoneda modules of $\mathcal{K}$ and $\mathcal{L}$. It follows that $Tw_{Ch}\mathcal{F}$ carries the structure of a triangulated category.

Finally, we review how the \emph{algebraic twist} construction from \cite{Seidel08} is defined in this framework.

The $A_\infty$-category $\mathcal{F}$ admits a natural embedding into $Tw_{Ch}\mathcal{F}$, by sending an object $L\in\mathcal{F}$ to $\Z\otimes L\in Tw_{Ch}\mathcal{F}$, where $\Z$ is the cochain complex concentrated in degree zero with trivial differential. Under this inclusion, we have canonical identifications of cochain complexes
\begin{align} \label{eqn-evaluation_map}
    \hom(\hom_\mathcal{F}^*(V,L)\otimes V,\Z\otimes L) &= \hom_\Z(\hom_\mathcal{F}^*(V,L),\Z)\otimes\hom^*_\mathcal{F}(V,L)\\
    &\cong \hom_\Z(\hom^*_\mathcal{F}(V,L),\hom^*_\mathcal{F}(V,L)).
\end{align}
We define the \emph{evaluation map} $ev\in\hom(\hom_\mathcal{F}^*(V,L)\otimes V,\Z\otimes L)$ to be the image of $\id\in \hom_\Z(\hom^*_\mathcal{F}(V,L),\hom^*_\mathcal{F}(V,L))$ under the above isomorphism. Since the above isomorphism is compatible with the differentials, it follows that the evaluation map $ev$ is a degree zero cocycle.

\begin{defn}
    The \emph{algebraic twist} of an object $L$ with respect to $V$ is the twisted complex
    \begin{equation}
        T_VL=\hom_\mathcal{F}^*(V,L)[1]\otimes V\oplus L,
    \end{equation}
    equipped with the filtration whose graded pieces are $\hom_\mathcal{F}^*(V,L)[1]\otimes V$ and $L$, and with the differential $\delta=ev[1]$, the shift of the evaluation map. 
\end{defn}

Using the identification in (\ref{eqn-evaluation_map}) above, and given a basis $\{x_i\}$ of $\hom^*_\mathcal{F}(V,L)$, the evaluation map can be written explicitly as
\begin{equation}
    ev=\sum_i (x_i)^\vee\otimes x_i \in \hom^*_\mathcal{F}(V,L)^\vee\otimes \hom(V,L).
\end{equation}

This construction extends to twisted complexes as follows. Consider a twisted complex twisted by cochain complexes
\begin{equation}
    (\mathcal{L}=\bigoplus_\alpha C^*_\alpha\otimes L_\alpha, \delta_\mathcal{L}).
\end{equation}
We aim to define the evaluation homomorphism
\begin{equation}
    ev\in\hom_{Tw_{Ch}\mathcal{F}}(\hom_{Tw_{Ch}\mathcal{F}}(\Z\otimes V,\mathcal{L})\otimes V,\mathcal{L})
\end{equation}
similarly to the case when $\mathcal{L}$ is a single object $L$. To do this, we define the component maps
\begin{equation}
    (ev)^{\alpha\beta}\in\hom(\hom^*_{Tw_{Ch}\mathcal{F}}(\Z\otimes V,C^*_\alpha\otimes L_\alpha)\otimes V,C^*_\beta\otimes L_\beta)
\end{equation}
for each $\alpha,\beta$ in the index set $A$. 

When $\alpha=\beta$, this recovers the evaluation homomorphism we defined in the previous discussion. For $\alpha\neq\beta$, we analyze the morphism space as
\begin{align}
    \hom_{Tw_{Ch}\mathcal{F}}(\hom^*_{Tw_{Ch}\mathcal{F}}(\Z\otimes V,C^*_\alpha\otimes L_\alpha)\otimes V,C^*_\beta\otimes L_\beta) &\cong \hom_{Tw_{Ch}\mathcal{F}}(\hom_\Z(\Z,C^*_\alpha)\otimes\hom^*_\mathcal{F}(V,L_\alpha)\otimes V,C^*_\beta\otimes L_\beta)\notag \\
    &\cong \hom_\Z(C^*_\alpha\otimes\hom^*_\mathcal{F}(V,L_\alpha),C^*_\beta)\otimes\hom^*_\mathcal{F}(V,L_\beta),
\end{align}
which naturally embeds into
\begin{equation}
    \hom_\Z^*(C^*_\alpha\otimes \hom^*_\mathcal{F}(V,L_\alpha),C^*_\beta\otimes\hom^*_\mathcal{F}(V,L_\beta)).
\end{equation}

Given each off-diagonal component of the differential $\delta_\mathcal{L}^{\alpha\beta}\in\hom(C^*_\alpha,C^*_\beta)\otimes\hom(L_\alpha,L_\beta)$ coming from the twisted complex $\mathcal{L}$, we define the component $(ev)^{\alpha\beta}$ by taking the image of $\delta_\mathcal{L}^{\alpha\beta}$ under the map
\begin{equation}
\hom^*_\mathcal{F}(L_\alpha, L_\beta) \to \hom^*_\Z(\hom^*_\mathcal{F}(V, L_\alpha), \hom^*_\mathcal{F}(V, L_\beta))
\end{equation}
induced by precomposition with $\mu^2$,
\begin{equation}
\mu^2: \hom^*_\mathcal{F}(L_\alpha, L_\beta) \otimes \hom^*_\mathcal{F}(V, L_\alpha) \to \hom^*_\mathcal{F}(V, L_\beta).
\end{equation}

One can check that the differential defined this way on $T_V\mathcal{L}$ satisfies the Maurer-Cartan relation, and thus $T_V\mathcal{L}$ is an object in $Tw_{Ch}\mathcal{F}$.

We conclude this subsection with the following lemma.

\begin{lemma} \label{lem:mu^2-with-ev}
    Let $V$, $L$, $K$ be objects in $\mathcal{F}$, and let $\alpha\in\hom^0_\mathcal{F}(L,K)$. Then,
    \begin{equation}
        \mu^2_{Tw_{Ch}}(\alpha,ev[1])\in\hom^1_\Z(\hom^*_\mathcal{F}(V,L),\Z)\otimes\hom^*_\mathcal{F}(V,K)
    \end{equation}
    corresponds, under the identification (\ref{eqn-evaluation_map}), to the map
    \begin{equation}
        (-1)^{\deg-1}\mu^2_\mathcal{F}(\alpha,~)[1]:\hom^*_\mathcal{F}(V,L)\to\hom^{*+1}_\mathcal{F}(V,K),
    \end{equation}
    where the shift arises from the shifted evaluation map $ev[1]$.
\end{lemma}
\begin{proof}
    Choose a basis $\{x_i\}$ for $\hom^*_\mathcal{F}(V,L)$. Then since $ev[1]=\sum_i (x_i)^\vee[-1]\otimes x_i$, it follows that
    \begin{equation}
        \mu^2_{Tw_{Ch}}(\alpha,ev[1])=\sum_i (-1)^{\lvert x_i\rvert-1} x_i^\vee[-1]\otimes \mu^2_\mathcal{F}(\alpha,x_i),
    \end{equation}
    which corresponds to $(-1)^{\deg-1}\mu^2_\mathcal{F}(\alpha,~)[1]$ under the identification (\ref{eqn-evaluation_map}).
\end{proof}

\subsection{Floer cohomology for exact Lagrangians}

The geometric setup throughout the remainder of this paper will be as follows: we consider an exact symplectic manifold $(X,\omega=d\lambda)$ with vanishing first Chern class, $c_1(X)=0$. This allows us to fix a choice of a fiber bundle $Gr_\Lambda(X)$, whose fiber over $x\in X$ is the fiberwise universal cover $\widetilde{Gr}_\Lambda(T_xX)$ of the oriented Lagrangian Grassmannian $Gr_\Lambda (T_xX)$. 

The Lagrangians we consider are \textit{exact Lagrangian submanifolds}, meaning that each $L$ is equipped with a function $f_L:L\to\R$ such that $df_L=\lambda\vert_L$. To define $\Z$-graded Floer cohomology groups with $\Z$-coefficients, we also equip our Lagrangians with extra data. We define a \textit{Lagrangian brane} to be an exact Lagrangian $(L,f_L)$ together with a Spin structure and a grading structure, where the latter means a consistent choice of lift of the Lagrangian Gauss map $T_xL\xhookrightarrow{}Gr_\Lambda X\vert_x$ to the universal cover $\widetilde{Gr}_\Lambda(X)\vert_x$. Such a structure exists if our underlying Lagrangian $L$ is Maslov zero, meaning the Maslov class $\mu_L\in H^1(L;\Z)$ is zero. From now on, all Lagrangians are assumed to be exact, Maslov zero, and Spin. 

We will also impose further geometric conditions on the symplectic manifolds and Lagrangians. As in \cite{AbouzaidSeidel}, we assume $X$ is Liouville and the Lagrangians are cylindrical at infinity. The almost complex structures $J$ used in defining the Floer cochain complex are assumed to be cylindrical, meaning it is compatible with the structure of $X$ at infinity.

In this setting, we define the Floer cochain complex of two exact Lagrangians $K,L$ as follows. We first assume that $K$ and $L$ meet transversely: in particular, all the intersection points of $K$ and $L$ are contained in a compact subset of $X$. Then the generators of the cochain complex are given by
\begin{equation}
    CF^*(K,L;\Z)=\bigoplus_{x\in K\cap L} \langle\mathfrak{o}_x\rangle,
\end{equation}
where $\langle\mathfrak{o}_x\rangle$ denotes the free $\Z$-module of rank 1 generated by the two possible orientations of the orientation line associated to the intersection point $x$, subject to the relation that their sum is zero. The degrees of the generators of $\mathfrak{o}_x$ are defined to be the Maslov class of the loop given by the path connecting the graded lifts of $T_xK$ and $T_xL$ in the universal cover of the Lagrangian Grassmannian, composed with the canonical short path connecting the lifts of $T_xL$ and $T_xK$. To define a differential $\partial$, we consider the following moduli space of $J$-holomorphic maps
\begin{align}
\begin{split}
    \hat{\mathcal{M}}(y;x)=\{u:\R\times[0,1]\to X\mid ~&\frac{\partial u}{\partial s}+J_{t}\frac{\partial u}{\partial t}=0,u(s,0)\in K, u(s,1)\in L,\\
    &\lim_{s\to-\infty}u(s,t)=y, \lim_{s\to\infty}u(s,t)=x\},
\end{split}
\end{align}
for $x\neq y$ and define $\mathcal{M}(y;x)=\hat{\mathcal{M}}(y;x)/\R$ to be the quotient by the $\R$-action acting by translation. For a Baire set of $\omega$-compatible almost complex structures $\{J_t\}$, this moduli space $\mathcal{M}(y;x)$ is transversely cut out, and has the expected dimension $\deg(y)-\deg(x)-1$. If we restrict to the case when $\deg(y)-\deg(x)=1$, then $\mathcal{M}(y;x)$ is a compact manifold of dimension $0$, i.e. a finite set of points. For each rigid strip $[u]$, the determinant line of the Fredholm operator associated to the Cauchy-Riemann equation determines a map between the orientation lines $\psi_u:\mathfrak{o}_x\to\mathfrak{o}_y$. We now define the Floer differential as the sum
\begin{equation}
    \partial \mathfrak{o}_x = \sum_{\lvert y\rvert=\lvert x\rvert+1,~ u\in\mathcal{M}(y;x)} \psi_u(\mathfrak{o}_x). 
\end{equation}

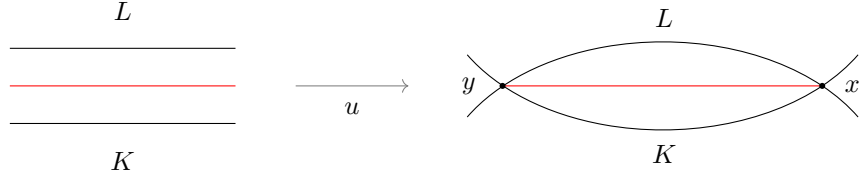
\begin{figure}[h]
\centering
\begin{tikzpicture}

\draw (-4.5,-0.5) -- (-1.5,-0.5);
\draw (-4.5,0.5) -- (-1.5,0.5);
\draw [draw=red] (-4.5,0) -- (-1.5,0);

\draw[->, draw=gray] (-0.7, 0) -- (0.8, 0);

\draw (6.3,0) arc [start angle=45, end angle=150, x radius=3, y radius=2];
\draw (6.3,0) arc [start angle=45, end angle=30, x radius=3, y radius=2];
\draw (6.3,0) arc [start angle=-45, end angle=-150, x radius=3, y radius=2];
\draw (6.3,0) arc [start angle=-45, end angle=-30, x radius=3, y radius=2];
\draw [draw=red] (2.05,0) -- (6.3,0);

\draw [draw=none, fill=black] (6.3,0) circle (0.04 and 0.04);
\draw [draw=none, fill=black] (2.05,0) circle (0.04 and 0.04);

\draw (-3,-1) node {$K$};
\draw (-3,1) node {$L$};
\draw (0.05,-0.3) node {$u$};

\draw (4.2,0.9) node {$L$};
\draw (4.2,-0.9) node {$K$};

\draw (6.7,0) node {$x$};
\draw (1.6,0) node {$y$};

\end{tikzpicture}
\caption{The domain and image of a map $u$ in the moduli space $\mathcal{M}(y;x)$. The red line denotes the homotopy class rel endpoints of the curve $\gamma_t(s)=u(s,t)$ in $X$, which is independent of $t$.}
\label{fig:floer-differential}
\end{figure}

We similarly define the higher $A_\infty$-structure maps $\mu^k$ for $k\geq 2$. To construct the relevant moduli spaces, we first consider $\overline{\mathcal{R}}_{k,1}$, the compactified Deligne-Mumford moduli space of stable disks with $k+1$ boundary marked points, and its universal curve $\mathcal{S}_{k,1}$. Now let $L_0,\cdots,L_k$ be mutually transverse exact Lagrangians. For each subset $I$ of $\{0,1,\cdots,k\}$ of size $l$, we pick positive and negative strip-like ends
\begin{align}
    &\epsilon_I^-:(-\infty,0]\times[0,1]\times\overline{\mathcal{R}}_{l,1}\to\overline{\mathcal{S}}_{l,1}~,\\
    &\epsilon_{I,j}^+:[0,\infty)\times[0,1]\times\overline{\mathcal{R}}_{l,1}\to\overline{\mathcal{S}}_{l,1},~j=1,\cdots,l
\end{align}
which define neighborhoods around the marked point sections of the universal curve. We also equip a family of $\omega$-compatible almost complex structures that are cylindrical at infinity
\begin{equation}
    J_I:\mathcal{S}_{l,1}\to\mathcal{J}(X),
\end{equation}
and require both the strip-like ends and almost complex structures to be consistent with the boundary stratification of $\mathcal{S}_{k,1}$. The construction of these consistent choice of strip-like ends with compatible almost complex structures can be found in \cite[Section 9]{Seidel08}. Now given intersection points
\begin{equation}
    y\in L_0\cap L_k, x_0\in L_0\cap L_1,\cdots,x_{k-1}\in L_{k-1}\cap L_k
\end{equation}
we define the moduli space
\begin{equation}
    \mathcal{M}_k(y;x_0,\cdots,x_{k-1})
\end{equation}
to be the moduli space of stable J-holomorphic maps $u:S_k\to X$ with Lagrangian boundary conditions $L_0,\cdots,L_k$ which map the boundary marked points to $y,x_0,\cdots,x_{k-1}$. Here, $S_k\in\mathcal{R}_{k,1}$ is the unit disc with its standard complex structure, and $k$ positive marked points at the boundary and $1$ negative marked point at $-1$. To compactify this moduli space, we allow stable maps $u$ as well: then by $J$ being cylindrical and the energy of $u$ being equal to the differences of the actions (which we will soon define), a monotonicity argument as in \cite[Section 7]{AbouzaidSeidel} ensures that the image of $u$ stays in a compact region. Therefore by Gromov compactness, the moduli space $\overline{\mathcal{M}_k}(y;x_0,\cdots,x_{k-1})$ is compact, and by a dimension argument we may define a map
\begin{equation}
    \mu^k:CF^*(L_{k-1},L_k)\otimes\cdots\otimes CF^*(L_0,L_1)\to CF^*(L_0,L_k).
\end{equation}

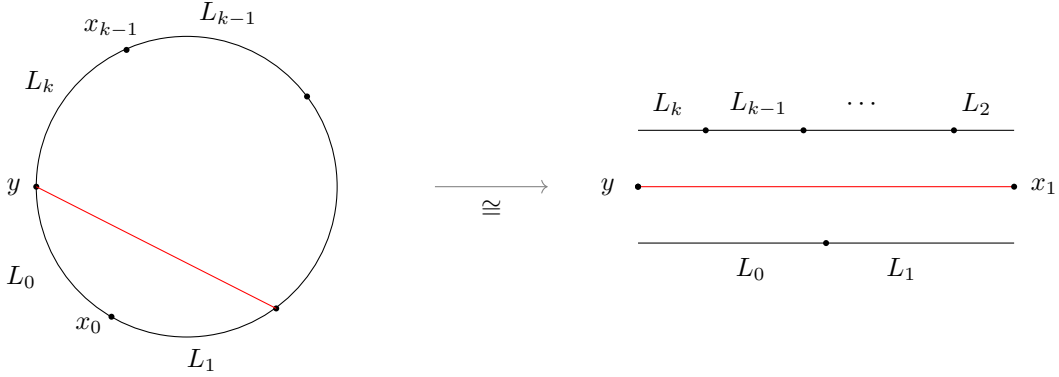
\begin{figure}[h]
\centering
\begin{tikzpicture}

\draw (2,-0.75) -- (7,-0.75);
\draw (2,0.75) -- (7,0.75);
\draw [draw=red] (2,0) -- (7,0);

\draw (-4,0) circle (2 and 2);
\draw[->, draw=gray] (-0.7, 0) -- (0.8, 0);

\draw [draw=none, fill=black] (-6,0) circle (0.04 and 0.04);
\draw [draw=none, fill=black] (-5.0,-1.73) circle (0.04 and 0.04);
\draw [draw=none, fill=black] (-4.8,1.82) circle (0.04 and 0.04);
\draw [draw=none, fill=black] (-2.81,-1.62) circle (0.04 and 0.04);
\draw [draw=none, fill=black] (-2.4,1.2) circle (0.04 and 0.04);
\draw [draw=none, fill=black] (7,0) circle (0.04 and 0.04);
\draw [draw=none, fill=black] (2,0) circle (0.04 and 0.04);

\draw [draw=red] (-6,0) -- (-2.81,-1.62);

\draw (-6.2,-1.2) node {$L_0$};
\draw (-3.8,-2.3) node {$L_1$};
\draw (-5.95,1.4) node {$L_k$};
\draw (-3.45,2.3) node {$L_{k-1}$};
\draw (0.05,-0.3) node {$\cong$};

\draw (-5.3,-1.85) node {$x_0$};
\draw (-5.0,2.1) node {$x_{k-1}$};
\draw (-6.3,0) node {$y$};

\draw (2.4,1.1) node {$L_k$};
\draw (3.6,1.1) node {$L_{k-1}$};
\draw (5.0,1.1) node {$\cdots$};
\draw (6.5,1.1) node {$L_2$};
\draw (3.5,-1.1) node {$L_0$};
\draw (5.5,-1.1) node {$L_1$};

\draw [draw=none, fill=black] (2.9,0.75) circle (0.04 and 0.04);
\draw [draw=none, fill=black] (4.2,0.75) circle (0.04 and 0.04);
\draw [draw=none, fill=black] (6.2,0.75) circle (0.04 and 0.04);
\draw [draw=none, fill=black] (4.5,-0.75) circle (0.04 and 0.04);
\draw [draw=none, fill=black] (2,0) circle (0.04 and 0.04);

\draw (7.4,0) node {$x_1$};
\draw (1.6,0) node {$y$};

\end{tikzpicture}
\caption{The domain of a stable map $u$ in the moduli space $\mathcal{M}_{k,1}(y;x_0,\cdots,x_{k-1})$. The right picture is a biholomorphic image, where the two ends of the strips are compactified at the marked points. The point of the right figure is to show that the homotopy class rel endpoints of the image of the red line in $X$ is again independent of $t$, when we puncture the right image at the points corresponding to $y$ and $x_1$ and consider it as a strip $\R\times[0,1]$.}
\label{fig:higher-moduli-space}
\end{figure}

Because we are in the setting where the symplectic manifold $X$ and Lagrangians $\{L_i\}$ are exact, we may define the \emph{action} of a transverse intersection point.

\begin{defn} \label{def:action}
    We define the \emph{action} associated to an intersection point $x$ of two exact Lagrangians $(K,f_K)$, $(L,f_L)$ to be
    \begin{equation}
        \mathcal{A}(x)=f_L(x)-f_K(x).
    \end{equation}
\end{defn}
The Floer differential increases action in our conventions. Until this point, we have defined $CF^*(K,L)$ for pairs of transversely meeting Lagrangians.

In general, one has to allow Hamiltonian perturbations to define the Floer cochain complexes and $\mu^k$ operations when the Lagrangians do not meet transversely. For the precise details in choosing the Floer datum, we refer to \cite[Chapter 2]{Seidel08}, and we may now define the Fukaya category.

\begin{defn}
    For a Liouville manifold $X$, we define the \emph{compact Fukaya category} $\mathcal{F}(X)$ to be the $A_\infty$-category whose objects are closed exact Lagrangian branes, with morphisms defined as
    \begin{equation}
        \hom(K,L)=CF^*(K,L).
    \end{equation}
    The maps $\mu^k$ defined above equip $\mathcal{F}(X)$ with the structure of a cohomologically unital $A_\infty$-category.
\end{defn}

There is one additional condition required in the construction of the Fukaya category. Recall that for a pair of transversely intersecting Lagrangians $K$ and $L$, it is possible to choose Floer data with vanishing Hamiltonian perturbation. When extending to consistent universal choices of perturbation data, as described in \cite[Lemma 9.5]{Seidel08}, we require that these choices agree with the Floer data on transverse pairs: specifically, the assigned Hamiltonian perturbation must remain zero for such pairs.

For twisted complexes of Lagrangians $\mathcal{K},\mathcal{L}\in Tw\mathcal{F}(X)$, we define
\begin{equation}
    CF^*(\mathcal{K},\mathcal{L})=\hom_{Tw\mathcal F(X)}(\mathcal{K},\mathcal
    L).
\end{equation}

In the Fukaya category $\mathcal{F}(X)$, compactly supported Hamiltonian isotopies of Lagrangians preserve the isomorphism class of the Lagrangian. The first step in the proof is the observation due to Floer (\cite[Theorem 2]{Floer89}) that for a graph of a $C^2$-small Morse function $h$ and a specific choice of $J_h$, the only $J_h$-holomorphic strips with boundary conditions on the zero section and the graph are gradient flowlines of $h$.

\begin{proposition} \label{prop:Floer-to-Morse-C2}
    Let $L$ be a closed exact Lagrangian equipped with a Riemannian metric $g$. Then there exists a constant $\epsilon>0$ such that the following holds: for every Morse function $h:L\to\R$ whose derivatives up to second order are bounded by $\epsilon$, there exists an almost complex structure $J_h$ in a neighborhood $U$ of $L$ containing the graph $L'=\Gamma(dh)$ such that any $J_h$-holomorphic strip with Lagrangian boundary conditions
    \begin{equation}
        u:\R\times[0,1]\to U, ~u(s,0)\in L, ~u(s,1)\in L'
    \end{equation}
    is a reparametrization of a gradient flowline of $h$. 
\end{proposition}

Now for any graph $\Gamma(df)$ contained in a small Weinstein neighborhood of $T^*L$, we use the above lemma to show that $\Gamma(\epsilon df)$ is categorically isomorphic to the zero section for some small $\epsilon>0$, and repeat this process to eventually show that the original graph $\Gamma(df)$ is categorically isomorphic to the zero section as well.

\begin{corollary} \label{cor:graph}
    Let $L_0$ be a closed exact Lagrangian equipped with a Riemannian metric $g$, and let $L_1$ be a graph of a Morse-Smale function $h$ on $L_0$ that is contained in a neighborhood $D_{c(\epsilon)}^*L_0$, where $c(\epsilon)$ is a constant depending on the value $\epsilon$ determined from Proposition \ref{prop:Floer-to-Morse-C2}. Then the two Lagrangians $L_0$ and $L_1$ are isomorphic objects in the Fukaya category.
\end{corollary}
\begin{proof}
    We first show that $L_0$ and $L_1$ are isomorphic objects for the case when the Morse function $h$ satisfies the $C^2$-boundedness conditions of Proposition \ref{prop:Floer-to-Morse-C2}. We have an isomorphism of cochain complexes
\begin{equation}
    CF^*(L_0,L_1;J_h)\cong CM^*(L_0;h,g),\quad CF^*(L_1,L_0;J_h)\cong CM^*(L_0;-h,g)
\end{equation}
    for the almost complex structure $J_h$ specified in Proposition \ref{prop:Floer-to-Morse-C2}. Now we define degree zero elements
\begin{equation}
    \alpha\in CM^*(L_0;h), ~\beta\in CM^*(L_0;-h)
\end{equation}    
    to each be the sum of the local minima, and local maxima of the Morse function $h$. Then $\alpha$ and $\beta$ can be shown to be cocycles, and satisfy
\begin{equation}
    \mu^2(\beta,\alpha)=e_{L_0}.
\end{equation}    
    This is because if we pick another generic Morse function $h'$ on $L_0$, then for each local minimum $x$ of $h'$, there is exactly one gradient flowline from a minimum of $h$ to a maximum of $h$ that passes through $x$. Therefore we may conclude that the equivalence elements $\alpha$ and $\beta$ define a categorical isomorphism between $L_0$ and $L_1$, when the chosen almost complex structure is $J_h$. For a general almost complex structure $J$, a continuation map argument together with the action filtration yields chain isomorphisms
\begin{equation} \label{eqn:Ham-isomorphism}
    CF^*(L_0,L_1;J)\cong CF^*(L_0,L_1;J_h),\quad CF^*(L_1,L_0;J)\cong CF^*(L_1,L_0;J_h),
\end{equation}
    allowing us to define the images of $\alpha$ and $\beta$ under these isomorphisms to be $\alpha'\in CF^*(L_0,L_1;J)$, $\beta'\in CF^*(L_1,L_0;J)$. Since (\ref{eqn:Ham-isomorphism}) is a chain isomorphism, $\alpha',\beta'$ are degree 0 cocycles. By a similar continuation map argument one can show that $[\mu^2(\beta',\alpha')]\in HF^*(L_0,L_0)$ is the cohomological identity, and so is $[\mu^2(\alpha',\beta')]\in HF^*(L_1,L_1)$. Therefore, the existence of isomorphism elements for $CF^*(L_0,L_1;J)$ does not depend on the choice of almost complex structure.

    For the general case, there exists some constant $\delta>0$ such that the function $\delta h$ satisfies the $C^2$-boundedness conditions of Proposition \ref{prop:Floer-to-Morse-C2} for any $h$ as above. Then by repeatedly applying the above argument to $\delta h$, $2\delta h$, $\cdots$, we obtain a sequence of equivalences of Lagrangians which compose to show that $L_0$ and $L_1$ are equivalent.
\end{proof}

The general case follows by decomposing the Hamiltonian isotopy into smaller pieces.

\begin{proposition} \label{prop:Ham-invariance}
    Let $L$ be a closed exact Lagrangian brane, and let $\phi_t$ be a Hamiltonian isotopy in time $t\in[0,1]$. Then the two Lagrangians $L$ and $\phi_1(L)$ define isomorphic objects in the Fukaya category, when $\phi_1(L)$ is equipped with the brane structure induced from $L$ and $\phi_t$.
\end{proposition}
\begin{proof}
    We may perturb the Hamiltonian isotopy $\phi_t$ rel endpoints such that for each time $0\leq t\leq 1$, the image of $L$ under the Hamiltonian isotopy $\phi_t$ is transverse to $L$ except for a finite set of $t$, and each $L_t$ and $L_s$ are also pairwise transverse. The original argument goes back to Floer \cite{Floer}, but the exact statement above can be found in \cite[Lemma 3.15]{AbouzaidKragh}, where it is also shown that such a property holds for a generic $C^1$-small perturbation. Then it is enough for us to show that $L_{t-\varepsilon}$ and $L_{t+\varepsilon}$ are equivalent for each $t,\varepsilon$, where we may assume that $L_{t\pm\varepsilon}$ are transverse to each other.

    Thus we may reduce to a local Weinstein neighborhood of each $L_t$, and reduce to the case where $L_1$ is a graph in $D^*_{c(\epsilon)}L_0$ of a Morse function. Now Corollary \ref{cor:graph} applies, and by composing all the equivalences we conclude that $L_0$ and $L_1$ are equivalent.
\end{proof}

Given a Morse-Smale pair $(h,g)$, we obtain a CW structure on $L$ where the cells are given by the unstable manifolds of the gradient flow, and the attaching maps are likewise determined by the gradient flowlines. Since any two CW structures on $L$ induce chain homotopic cellular cochain complexes, we are free to choose $h$ to be a $C^2$-small Morse function. For such a choice and an almost complex structure $J_h$ determined by $h$, there is a canonical identification between holomorphic strips and gradient flowlines as in Proposition \ref{prop:Floer-to-Morse-C2}. Since continuation map methods give chain homotopies between Floer cochain complexes defined by different almost complex structures, we arrive at the following:

\begin{proposition}
    For any closed exact Lagrangian $L$ and a Morse function $h:L\to\R$, there exists a chain homotopy equivalence
    \begin{equation}
        CF^*(L,L)=CM^*(L;h,g)\cong C^*_{cell}(L)
    \end{equation}
    for some Morse-Smale metric $g$ on $L$.
\end{proposition}

In Subsection \ref{ssec:bimodule}, we will upgrade this to a \emph{simple homotopy equivalence}, after upgrading the cochain complexes on both sides to $\Z\pi_1(X)$-coefficients.

\subsection{Lefschetz fibrations} \label{ssec:Lef-fib}

In this subsection, we review the symplectic geometry of Lefschetz fibrations following \cite{Seidel18}. For additional background and foundational results, see also \cite[Chapter 1]{Seidel03} and \cite[Section 16]{Seidel08}.

First, we define the fiber of the exact Lefschetz fibration to be a Liouville manifold $(M,\omega_M=d\theta_M)$ together with a cylindrical almost complex structure $J_M$. For the base, we use the open upper half-plane $(\mathbb{H},\omega_\mathbb{H},j)$ together with its standard symplectic form and complex structure. The total space of our exact Lefschetz fibration will be an exact symplectic manifold $(E,\omega_E=d\theta_E)$ together with a projection $\pi:E\to \mathbb{H}$. 

We define the above fibration $\pi:E\to\mathbb{H}$ to be \emph{symplectically trivial} in a neighborhood $U_x\subset E$ of a point $x\in E$ if the following conditions hold:
\begin{enumerate}
    \item $d\pi$ is surjective at the point $x$,
    \item The subspaces $TE^v\lvert_x$, $TE^h\lvert_x$ are symplectic subspaces of $(TE\lvert_x,\omega_E)$,
    \item $d\pi^*\omega_{\mathbb{H}}=\omega_E\lvert_{TE^h}$.
\end{enumerate}

In this case, such a map $\pi$ induces a decomposition of the tangent bundle
\begin{equation}
    TE=TE^v\oplus TE^h,
\end{equation}
where the two subbundles $TE^v$ and $TE^h$ are defined as
\begin{equation}
    TE^v=\ker(d\pi),TE^h=(TE^v)^{\omega_E},
\end{equation}
where $TE^h$ is the symplectic orthogonal complement of $TE^v$ with respect to the symplectic form $\omega_E$. The horizontal subspaces $TE^h$ define an Ehresmann connection, and so one can define \emph{symplectic parallel transport} with respect to this connection. This symplectic parallel transport induces exact symplectic isomorphisms between the fibers.

The definition of an \emph{exact Lefschetz fibration} mainly consists of two parts: we require the fibration to be symplectically trivial in the neighborhood of the horizontal and vertical boundaries of $E$, and the critical points to be complex nondegenerate. To state the first part, let $D\subset\mathbb{H}$ be a disc that contains all the critical values $z$ of the map $\pi$. Our assumptions on $E$ are the following: 
\begin{enumerate}
    \item In a neighborhood of $\bigcup \partial_\infty E_x$, the fibration $\pi$ is symplectically trivial and $TE^h\subset T(\partial E)$.
    \item On the complement $\mathbb{H}\setminus D$, the fibration $\pi$ is symplectically locally trivial.
\end{enumerate}
We also pick a basepoint $*\in\mathbb{H}\setminus D$ together with an isomorphism $(M,\omega_M=d\theta_M)\cong (E\lvert_*,\omega_E\lvert_s)$. By symplectic parallel transport, this extends to an identification
\begin{equation}
    (\pi^{-1}(U),\omega_E) \cong (U\times M,\omega_\mathbb{H}+\omega_M)
\end{equation}
over any contractible open subset $U$ of $\mathbb{H}\setminus D$.

Now we wish to compactify the base $\mathbb{H}$ to $D^2$, recovering the more familiar picture of a Lefschetz fibration. Ignoring symplectic forms, we construct a smooth map 
\begin{equation}
    \overline{\pi}: \overline{E}\to \overline{\mathbb{H}}\cong D^2,
\end{equation}
and pick a positive symplectic form $\beta$ on $D^2$. Then we define a symplectic form on $E$ as
\begin{equation}
    \omega_{\overline{E}}=\omega_E+\pi^*(\rho(\beta-\omega_\mathbb{H}))
\end{equation}
for some cutoff function $\rho$ that takes value $1$ near the boundary $\partial D^2$ and takes value $0$ in $D$, and check that this uniquely extends to an exact symplectic form on $\overline{E}$. Now define
\begin{equation*}
    \partial^vE=\overline{\pi}^{-1}(\partial D^2).
\end{equation*}
For the vertical boundary, we note that the conditions required for $\partial E$ define a symplectomorphism over some neighborhood $U$ of $\partial E$ and some neighborhood $V$ of $\partial M$:
\begin{equation}
    (U,\omega_E)\cong (\mathbb{H}\times V, \omega_\mathbb{H}+\omega_M).
\end{equation}

Now we require the ``Lefschetz'' conditions for $\pi: E\to \mathbb{H}$. We assume that $\pi$ only has finitely many critical values, and over each critical value there exists exactly one critical point. In each neighborhood of a critical point, we require that there exists some complex structure $J$ that is compatible with the symplectic form $\omega_{\overline{E}}$ defined earlier, such that $\pi$ is $(J,j_\mathbb{H})$-holomorphic, and each critical point is complex nondegenerate.

\begin{defn}
    The above data of $\pi:E\to \mathbb{H}$, together with all choices, is defined to be an \emph{exact Lefschetz fibration}.
\end{defn}

We now introduce a model for the Fukaya category of a Lefschetz fibration introduced in \cite{Seidel17}, \cite{Seidel18}. We first recall some hyperbolic geometry from \cite[Section 2]{Seidel18}. Define
\begin{equation}
    G=\R\rtimes\R^{>0}
\end{equation}
to be the group of affine transformations of the real line, and let $\mathfrak{g}$ be its Lie algebra. 

Every element $a$ in $\mathfrak{g}$ induces a vector field $X_a$ on the upper half plane $\mathbb{H}$ by the action of $G$. One can check that this vector field is Hamiltonian, and define
\begin{equation}
    H_a=\theta_\mathbb{H}(X_a)
\end{equation}
to be the Hamiltonian that induces $X_a$. Now given a connection
\begin{equation}
    A\in\Omega^1([0,1],\mathfrak{g}),
\end{equation}
there is an induced parallel transport map $\phi_A$ from $0$ to $1$, which can be regarded as an element of $G$.

\begin{defn}
    For any two real numbers $\lambda_0$, $\lambda_1$, we define the set
    \begin{equation}
        \mathcal{A}([0,1],\lambda_0,\lambda_1)
    \end{equation}
    to be the set of connections $A\in\Omega^1([0,1],\mathfrak{g})$ such that
    \begin{equation}
        \phi_A^{-1}(\lambda_1)<\lambda_0.
    \end{equation}
\end{defn}

For any pair of real numbers $(\lambda_0,\lambda_1)$, this set is shown to be weakly contractible in \cite[Section 2]{Seidel18}.

Now we return to our geometric setup. Let $\pi:E\to \mathbb{H}$ be an exact Lefschetz fibration. First, we will define the set of almost complex structures and Hamiltonians that we will use. Define $\mathcal{J}(E)$ to be the $\omega_E$-compatible almost complex structures $J$ such that $d\pi$ is $J$-holomorphic on $\mathbb{H}\setminus D$, and in a neighborhood of the horizontal boundary $J$ is equal to the product $j_\mathbb{H}\times J_M$. For each element $\gamma\in\mathfrak{g}$, define
\begin{equation}
    \mathcal{H}_\gamma(E)\subset C^\infty(E,\R)
\end{equation}
to be the Hamiltonians which agree with the pullback of $H_\gamma$ outside a compact subset of $E$.

Now we may begin to define our Fukaya category of a Lefschetz fibration $\mathcal{F}(\pi)$. First, we will specify the objects of this Fukaya category. Let $L$ be a connected exact Lagrangian submanifold of $E$, with a brane structure.
\begin{defn}
    The Lagrangian $L$ is \emph{admissible} if outside of a compact set, the image $\pi(L)$ is contained in a vertical line
    \begin{equation*}
        \{z\mid \mathrm{re}(z)=\lambda_L, \im(z)\ll 1\}
    \end{equation*}
    for some real number $\lambda_L$. We define $L$ to be a \emph{Lefschetz thimble} if $\pi(L)=\gamma$ for some vanishing path $\gamma$. Here a path $\gamma:[0,\infty)\to\mathbb{H}$ is a \emph{vanishing path} when $\gamma(0)$ is a critical value of $\pi$, and there are no other critical values in the image of $\gamma$.
\end{defn}

A picture of an admissible Lagrangian can be found in Figure \ref{fig:Lefschetz-half-plane}.

\begin{figure}[h]
\centering
\begin{tikzpicture}

\draw[dotted] (-6,-4) -- (-1,-4);
\draw (-4,-3) node {$\times$};
\draw (-3,-2.5) node {$\times$};
\draw (-2.4,-3) node {$\times$};
\draw (-5,-2) node {$\times$};

\draw [fill=lightgray!50] plot [smooth cycle] coordinates {(2.4,-2.5) (3.5,-1.8) (4.8,-3) (3.2,-3.3)};
\draw (3.8,-3.3) -- (3.8,-4);
\draw (2,-2.05) -- (2,-4);

\draw[dotted] (1,-4) -- (6,-4);
\draw (3,-3) node {$\times$};
\draw (4,-2.5) node {$\times$};
\draw (4.6,-3) node {$\times$};
\draw (2,-2) node {$\times$};

\draw[->, color=lightgray] (1,-3.9) -- (6,-3.9);

\draw (4.8,-2.3) node {$L$};
\draw (3.8,-4.3) node {$\lambda_L$};
\draw (1.7,-3.0) node {$B$};

\end{tikzpicture}
\caption{The left picture indicates the upper half-plane model for the base of the Lefschetz fibration, where the dotted line below is the boundary of $\mathbb{H}$. The points marked with a $\times$ are critical values of the projection $\pi$. The right picture shows an example of an admissible Lagrangian $L$ together with its limit value $\lambda_L$, and a Lefschetz thimble $B$. The gray arrow denotes the direction of the wrapping when we compute the Floer cochain complexes.}
\label{fig:Lefschetz-half-plane}
\end{figure}
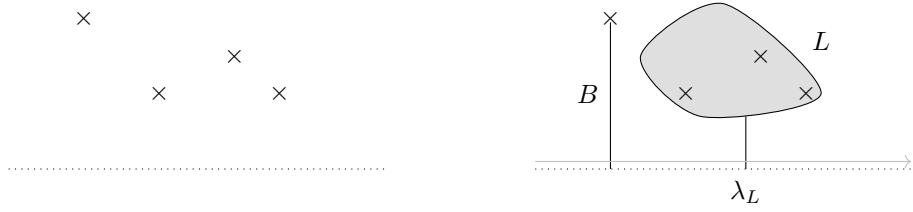

For each pair of admissible Lagrangians $(L_0,L_1)$, we fix the datum
\begin{equation}
    A_{L_0,L_1}\in\mathcal{A}([0,1],\lambda_{L_0},\lambda_{L_1}),~J_{01,t}\in\mathcal{J}(E),~H_{01,t}\in\mathcal{H}_{A_{L_0,L_1}}(E)
\end{equation}
of a connection $A_{L_0,L_1}$, and a generic $J_{01,t},H_{01,t}$ such that the Lagrangians $\phi(L_0)$ and $L_1$ are transverse for the time-1 map $\phi$ of the Hamiltonian $H_{01,t}$.

To define the $A_\infty$-structure maps, we first pick a disc with $d+1$ boundary punctures
\begin{equation}
    S_{d+1}=D^2\setminus\{p_0,\cdots,p_d\}
\end{equation}
and denote by $\partial_iS$ the boundary component of $S_{d+1}$ connecting the punctures $p_i$ and $p_{i+1}$. For each $s\in\partial_j S_{d+1}$, there is an associated Lagrangian boundary condition $L_j$ which is assumed to be admissible.

To this disc, we associate strip-like ends
\begin{align}
\begin{split}
    &\epsilon_0:(-\infty,0]\times[0,1]\to S_{d+1}\\
    &\epsilon_j:[0,\infty)\times[0,1]\to S_{d+1}\quad(j=1,\cdots,d)
\end{split}    
\end{align}
such that
\begin{equation}
    \lim_{s\to\pm\infty}\epsilon_j(s,t)=p_j,~\epsilon_j^{-1}(\partial S_{d+1})=\{(s,t)\mid t=0,1\}.
\end{equation}
Now remember that for each pair of admissible Lagrangians $(L_{j},L_{j+1})$, we fixed a connection
\begin{equation}
    A_j\in\mathcal{A}([0,1],\lambda_{j},\lambda_{j+1}).
\end{equation}
We first define a function
\begin{equation}
    \lambda\in C^\infty(\partial S_{d+1},\R)
\end{equation}
which on each boundary component $\partial_j S_{d+1}$ has the value $\lambda_{L_j}$ associated to the Lagrangian boundary condition $L_j$. Then \cite[Corollary 2.4]{Seidel18} ensures that we can pick a flat connection
\begin{equation}
    A\in\Omega^1(S_{d+1},\mathfrak{g})
\end{equation}
whose parallel transport along any boundary component $\partial_j S_{d+1}$ preserves $\lambda$, and pulls back to $A_j$ at each strip-like end $\epsilon_j$. Once we have chosen such $(A,\lambda)$, pick domain-dependent almost complex structures $(J_s)$ and a perturbation term $K$
\begin{equation} \label{eqn:perturbation-datum}
    J_s\in\mathcal{J}(E), K\in\Omega^1(S_{d+1},C^\infty(E,\R))
\end{equation}
to satisfy the compatibility conditions with the $(H_{j,t},J_{j,t})$ chosen before:
\begin{enumerate}
    \item The convergence $J_{\epsilon_j(s,t)}\to J_{j,t}$ is exponential in $s$ at each strip-like end $\epsilon_j$, and $J_{\epsilon_j(s,t)}$ agrees with $J_{j,t}$ outside a compact subset of $E$,
    \item $\epsilon_j^*K=H_{j,t}dt$,
    \item For each $\xi\in TS_{d+1}$, $K(\xi)\in\mathcal{H}_{A(\xi)}(E)$.
\end{enumerate}
After choosing all this datum to be compatible with the stratification of the Deligne-Mumford moduli space of stable disks with $d+1$ marked points, we may define the moduli space
\begin{align}
\begin{split}
    \mathcal{M}(L_0,\cdots,L_d)=\{u:S_{d+1}\to E\mid ~&(Du-X_K)^{0,1}=0,\\
    &~u(s)\in L_j, s\in\partial_j S_{d+1}\}.
\end{split}    
\end{align}
Then by the compactness results outlined in \cite[Section 4.3]{Seidel18}, the maps $u$ in the moduli space above cannot escape to $\partial E$ and infinity in the base and fiber directions. Therefore we have compactness, and our domain-dependent almost complex structures and perturbation terms ensures transversality. Now we may define the morphisms
\begin{equation}
    \hom_{\mathcal{F}(\pi)}(L_0,L_1)
\end{equation}
as a cochain complex, and the structure maps
\begin{equation}
    \mu^d:\hom(L_{d-1},L_d)\otimes\hom(L_{d-2},L_{d-1})\otimes\cdots\otimes\hom(L_0,L_1)\to\hom(L_0,L_d)
\end{equation}
to satisfy the $A_\infty$-relations. In a nutshell, to compute $\hom_{\mathcal{F}(\pi)}(L_0,L_1)$, we ``wrap'' $L_0$ by a Hamiltonian isotopy $\phi_H$ until $\lambda_{\phi(L_0)}>\lambda_{L_1}$, then compute their Floer cohomology
\begin{equation}
    \hom_{\mathcal{F}(\pi)}(L_0,L_1)=CF^*(\phi(L_0),L_1).
\end{equation}

\subsection{Wrapped Floer cohomology}

In this subsection, we recall the computation of the wrapped Floer cohomology of a cotangent fiber (Proposition \ref{prop:AS}), and the generation statement for the wrapped Fukaya category of a Weinstein 1-handle attachment of cotangent bundles (Corollary \ref{cor:handle-attaching-generators}).

The computation of the wrapped Floer cohomology of a cotangent fiber is due to Abbondandolo-Schwarz \cite{AbbondandoloSchwarz} and Abouzaid \cite{Abouzaid12a}:

\begin{proposition} \label{prop:AS}
    Let $Q$ be a closed smooth Spin manifold, and let $T^*_qQ$ be a cotangent fiber. Then we have an isomorphism
    \begin{equation}
        HW^*(T^*_qQ,T^*_qQ)\cong H_{-*}(\Omega_q Q).
    \end{equation}
    In particular, the wrapped Floer cohomology of a cotangent fiber is supported in non-positive degrees.
\end{proposition}

\cite[Theorem 1.1]{Abouzaid12a} proves an $A_\infty$-equivalence between the cochain complexes $CW^*(T^*_qQ,T^*_qQ)$ and $C_{-*}(\Omega_q Q)$. We also recall from \cite[Theorem 1.1]{Abouzaid11} that a cotangent fiber generates the wrapped Fukaya category of a cotangent bundle $T^*Q$, where $Q$ is a closed smooth Spin manifold $Q$. Every oriented 3-manifold is parallelizable, and hence is Spin. In particular, this applies to all three-dimensional lens spaces $L(p,q)$.

\begin{proposition}[{\cite[Theorem 1.1]{Abouzaid11}}]
    Let $Q$ be an oriented Spin closed smooth manifold. Then the wrapped Fukaya category $\mathcal{W}(T^*Q)$ of the cotangent bundle $T^*Q$ is generated by a cotangent fiber.
\end{proposition}

Another statement we use is that subcritical Weinstein handle attaching does not change the wrapped Fukaya category up to quasi-equivalence.

\begin{proposition}[{\cite[Corollary 1.29]{GPS}}]\label{prop:subcrit-handle}
    Suppose that $X$ is a Liouville manifold, and let $X'$ be a Liouville manifold constructed from $X$ by attaching a subcritical Weinstein handle. Then the wrapped Fukaya category of $X'$ is quasi-equivalent to the wrapped Fukaya category of $X$.
\end{proposition}

Combining these two statements, we obtain the following statement about the generators of the wrapped Fukaya category of two cotangent bundles joined by a Weinstein 1-handle attachment.

\begin{corollary} \label{cor:handle-attaching-generators}
    Let $X$ be a Liouville manifold obtained from attaching two cotangent bundles $T^*Q_1$ and $T^*Q_2$ by a Weinstein 1-handle, where $Q_1,Q_2$ are smooth closed Spin manifolds. Then the wrapped Fukaya category $\mathcal{W}(X)$ of $X$ is generated by the cotangent fibers $T^*_{q_1}Q_1$ and $T^*_{q_2}Q_2$.
\end{corollary}

\subsection{Floer cohomology with local systems} \label{ssec:locsys}

In this subsection, we recall the definition of Floer cohomology for Lagrangians equipped with local systems. We refer to \cite[Section 2]{Abouzaid12b} for more details.

Our setup is a Liouville domain $X$, and all Lagrangians are assumed to be equipped with brane structures. We will define the Floer cohomology of two Lagrangians $(K,\mathcal{E}_K)$, $(L,\mathcal{E}_L)$ equipped with local systems of finite rank free $R$-modules for a ring $R$. The generators in cohomological degree $i$ are defined as
\begin{equation}
    CF^i((K,\mathcal{E}_K),(L,\mathcal{E}_L))\coloneqq \bigoplus_{x\in K\cap L, \lvert x\rvert=i} \hom({\mathcal{E}_K}\vert_x,\mathcal{E}_L\vert_x)\otimes\langle\mathfrak{o}_x\rangle,
\end{equation}
where $\lvert x\rvert$ is the degree of $x$ in the usual Lagrangian Floer cochain complex $CF^*(K,L)$. The differential is defined as follows. For each holomorphic strip $u:\R\times[0,1]\to X$ that contributes to the Floer differential from $x$ to $y$, we define the parallel transport map $\gamma_u^0:\mathcal{E}_K\vert_y\to\mathcal{E}_K\vert_x$ to be the parallel transport along the path $u(s,0)$, and $\gamma_u^1:\mathcal{E}_L\vert_x\to\mathcal{E}_L\vert_y$ to be the parallel transport along the path $u(-s,1)$. Then we may define the differential $\mu^1$ to send an element $\phi\in\hom(\mathcal{E}_K\vert_x,\mathcal{E}_L\vert_x)$ to
\begin{equation}
    \mu^1(\phi)=\sum_u\gamma_u^1\circ\phi\circ\gamma_u^0\otimes\psi_u,
\end{equation}
where $\psi_u:\mathfrak{o}_x\to\mathfrak{o}_y$ is the induced map on orientation lines.

In the case where the two objects are the same underlying closed exact Lagrangian brane equipped with two different local systems, we have the following proposition from \cite[Lemma 2.16]{Abouzaid12b}:
\begin{proposition} \label{prop:same-lag-loc-sys}
    $HF^*((L,\mathcal{E}_0),(L,\mathcal{E}_1))\cong H^*(L,\hom_R(\mathcal{E}_0,\mathcal{E}_1))$.
\end{proposition}

We now define the Fukaya category $\mathcal{S}(X)$ of Lagrangians equipped with local systems. The objects of $\mathcal{S}(X)$ are pairs $(L,\mathcal{E})$, where $L$ is an exact closed Maslov zero Lagrangian submanifold equipped with a brane structure as in \cite[Section 12]{Seidel08}, and $\mathcal{E}$ is a local system over $L$. To define the $A_\infty$-structure maps $\mu^k$, we consider a disc $(S,j)$ with $k+1$ punctures on the boundary, together with its standard almost complex structure. As in the usual Fukaya category, we consider maps $u:(S,j)\to (M,J)$ satisfying the perturbed Cauchy-Riemann equation with Lagrangian boundary conditions for $(L_0,\cdots,L_k)$. Each boundary component in the image of this disc will determine a path in $L_i$, and we define the parallel transport along this map to be $\varphi_i$. We define $\mu^k$ as the sum over such pseudoholomorphic maps $u$ which are rigid:
\begin{equation}
    \mu^k(\alpha_k,\cdots,\alpha_1)=\sum_{u:S\to X} \alpha_k\circ \varphi_{k-1}\circ\alpha_{k-1}\cdots\circ\varphi_1\circ\alpha_1\otimes\psi_u.
\end{equation}
One can check that this is an $A_\infty$-category, and that the full subcategory of objects whose associated local system is trivial of rank $1$ recovers the compact Fukaya category $\mathcal{F}(X)$.

\section{Simple categorical notions for the Fukaya category} \label{sec:simple-equivalences}

In this section, we study how the notion of simple homotopy can further refine $A_\infty$-categorical notions. The key definition we use is an $A_\infty$-bimodule $\underline{CF^*(K,L)}$ with coefficients in the group ring $\Z\pi_1(X)$ introduced in \cite{AbouzaidKragh}. In the following subsection, we first extend the above definition to twisted complexes twisted by cochain complexes over $\Z$, as introduced in Section \ref{ssec:category-theory}.

We remark that Whitehead torsion and Reidemeister torsion are only defined for (co)chain complexes that are finitely generated. Thus, we restrict our attention to Fukaya categories that are chain-level proper: the precise versions of the Fukaya categories we study in this section will be the compact Fukaya category $\mathcal{F}(X)$, and the Fukaya category of a Lefschetz fibration $\mathcal{F}(\pi)$ as defined in \cite{Seidel18}, both over $\Z$-coefficients.

\subsection{The $A_\infty$-bimodule $\underline{CF^*(K,L)}$} \label{ssec:bimodule}

Our setup is an exact symplectic manifold $(X,\omega=d\lambda)$, with exact Lagrangians $(L,f_L)$ equipped with brane structures. We fix a choice of a universal cover $p:\Tilde{X}\to X$. 

We will first define the $A_\infty$-bimodule $\underline{\mathcal{B}}(K,L)=\underline{CF^*(K,L)}$ for transverse Lagrangians $K$, $L$, then show that the simple homotopy type of the underlying cochain complex is invariant under Hamiltonian isotopies of $L$ that are transverse to $K$, and also for continuation maps induced from homotopies of the almost complex structure. Thus the simple homotopy type of the underlying cochain complex of $\underline{CF^*(K,L)}$ will be independent of the choice of Floer datum chosen for defining $CF^*(K,L)$ in the Fukaya category.

Suppose that we have two transverse Lagrangians $K$, $L$, and choose a regular almost complex structure $\{J_t\}$ on $X$. Let $u:\R\times[0,1]\to X$ be an element in $\hat{\mathcal{M}}(y;x)$, meaning that
\begin{equation}
    \frac{\partial u}{\partial s}+J_t\frac{\partial u}{\partial t}=0, \lim_{s\to-\infty}u(s,t)=y,\lim_{s\to\infty}u(s,t)=x.
\end{equation}
Since the domain of $u$ is contractible, once we fix a lift $\Tilde{x}$ of $x$, there exists a unique lift $\Tilde{u}:\R\times[0,1]\to \Tilde{X}$ of $u$ such that
\begin{equation}
    \lim_{s\to\infty}\Tilde{u}(s,t)=\Tilde{x}.
\end{equation}
Then the negative asymptotic of $\Tilde{u}$ will also be a lift of $y$, which we will call $\Tilde{y}(u,\Tilde{x})$.

\begin{defn} \label{defn:bimodule-without-lifts}
    Let $K,L\subset X$ be transversely intersecting exact Lagrangians, and let $J_t$ be a regular almost complex structure. We define the cochain complex $\underline{CF^*(K,L;J_t)}$ as follows. The generators are pairs $(\Tilde{x},\mathfrak{o}_x)$, where $\Tilde{x}$ can be any lift of an intersection point $x\in K\cap L$, and $\mathfrak{o}_x$ is the associated orientation line of $x$:
    \begin{equation}
        \underline{CF^*(K,L;J_t)}=\bigoplus_{x\in K\cap L} \Z\langle(\Tilde{x},\mathfrak{o}_x)\rangle.
    \end{equation}
    The differential is defined by summing over all lifts of rigid pseudoholomorphic strips $u\in\mathcal{M}(y;x)$:
    \begin{equation}
        \partial (\Tilde{x},\mathfrak{o}_x)=\sum_{u} \psi_u(\mathfrak{o}_x)(\Tilde{y}(u,\Tilde{x}),\mathfrak{o}_y),
    \end{equation}
    where $\psi_u:\mathfrak{o}_x\to\mathfrak{o}_y$ is the map induced on orientation lines by the pseudoholomorphic strip $u$, and $\Tilde{y}(u,\Tilde{x})$ denotes the lift of $y$ determined by $u$ and $\Tilde{x}$.
\end{defn}


The deck transformations of the universal cover $p:\Tilde{X}\to X$ induce an action of $\pi_1(X)$ on $\underline{CF^*(K,L;J_t)}$, thus endowing each degree component with the structure of a $\Z\pi_1(X)$-module. By choosing a preferred lift $\Tilde{x}$ for each intersection point $x\in K\cap L$, we obtain a distinguished basis, which identifies $\underline{CF^*(K,L;J_t)}$ as a based cochain complex over $\Z\pi_1(X)$. This identification enables us to define the simple homotopy type of $(\underline{CF^*(K,L;J_t)},\{\Tilde{x}\})$. Since choosing a different choice of basis $\{\Tilde{x}'\}$ corresponds to a simple base change, the simple homotopy type is independent of such choices. Thus the Whitehead torsion of $\underline{CF^*(K,L;J_t)}$ is well-defined without the choice of lifts, when it is acyclic.

We now provide a more concrete explanation of how a choice of basis $\{\Tilde{x}\}$ identifies $\underline{CF^*(K,L;J_t)}$ with a based cochain complex over $\Z\pi_1(X)$. While the previous definition is better suited for understanding the $A_\infty$-bimodule structure of $\underline{CF^*(K,L)}$ later on (as it solely relies on counting lifted pseudoholomorphic discs), this more explicit definition will be particularly useful for demonstrating that the resulting complex has trivial Whitehead torsion.

We begin by explaining how the choice of lifts $\{\Tilde{x}\}$ equips $\underline{CF^*(K,L;J_t)}$ with the structure of a based cochain complex over $\Z\pi_1(X)$. For any other lift $\Tilde{x}'$ of a point $x\in K\cap L$, there exists a unique element $g\in\pi_1(X)$ such that 
\begin{equation}
    \Tilde{x}'=g\cdot\Tilde{x},
\end{equation}
where $g$ acts by the corresponding deck transformation of $\Tilde{X}$. We will label the other lifts of $x$ as $g\cdot \Tilde{x}$. 

To describe the differential, consider a rigid holomorphic strip $u\in\hat{\mathcal{M}}(y;x)$, and let $\Tilde{u}:\R\times[0,1]\to \Tilde{X}$ be a lift of $u$ such that
\begin{equation}
    \lim_{s\to\infty}\Tilde{u}(s,t)=\Tilde{x}.
\end{equation}
Then the negative asymptotic of $\Tilde{u}$ will also be a lift of $y$, which we may write as
\begin{equation}
    \lim_{s\to-\infty}\Tilde{u}(s,t)=g(u)\cdot\Tilde{y}
\end{equation}
for a uniquely determined $g(u)\in\pi_1(X)$. This defines the group element $g(u)$ associated to any rigid holomorphic curve $u\in\Hat{\mathcal{M}}(y;x)$, as illustrated in Figure \ref{fig:lifting-holomorphic-curve}.

\begin{figure}[h]
\centering
\begin{tikzpicture}

\draw (5.5,0) arc [start angle=45, end angle=135, x radius=3, y radius=2];
\draw (5.5,0) arc [start angle=-45, end angle=-135, x radius=3, y radius=2];
\draw [draw=red] (5.5,0) arc [start angle=-50, end angle=-130, x radius=3.3, y radius=1.8];

\draw [draw=red] (5.5,3) arc [start angle=45, end angle=180, x radius=2.6, y radius=0.7];
\draw [draw=none, fill=white] (1.8,3) circle (0.2 and 0.1);
\draw [draw=red] (5.5,4) arc [start angle=45, end angle=315, x radius=2.6, y radius=0.7];

\draw [draw=none, fill=black] (5.5,4) circle (0.04 and 0.04);
\draw [draw=none, fill=black] (1.06,3.5) circle (0.04 and 0.04);
\draw [draw=none, fill=black] (1.06,2.5) circle (0.04 and 0.04);

\draw (3.5,-1) node {$K$};
\draw (3.5,1) node {$L$};
\draw (5.7,-0) node {$x$};
\draw (1.1,0) node {$y$};
\draw (5.7,4) node {$\Tilde{x}$};
\draw (0.7,3.5) node {$\Tilde{y}$};
\draw (1,2.2) node {$g(u)\cdot \Tilde{y}$};

\draw (7,-0) node {$X$};
\draw (7,3.5) node {$\Tilde{X}$};
\draw[->, draw=lightgray] (7, 3.1) -- (7, 0.4);

\end{tikzpicture}
\caption{The lift of the image of the pseudoholomorphic curve $u$ to the universal cover. The red curve is the image of $\gamma(s)=u(s,1/2)$. $\gamma$ lifts to a curve $\Tilde{\gamma}$ such that $\lim_{s\to\infty}\Tilde{\gamma}(s)=\Tilde{x}$, unique up to homotopy rel endpoints. The element $g(u)\in\pi_1(X)$ associated to $u$ is defined as the unique element such that $\lim_{s\to-\infty}\Tilde{\gamma}(s)=g(u)\cdot\Tilde{y}$, with respect to the chosen lift $\Tilde{y}$.}
\label{fig:lifting-holomorphic-curve}
\end{figure}
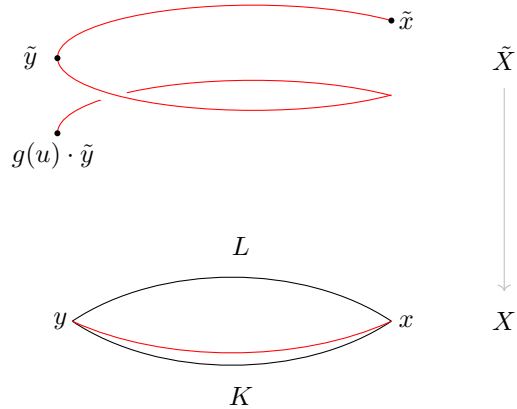

\begin{defn} \label{defn-bimodule-with-lifts}
    Let $K,L\subset X$ be exact Lagrangians intersecting transversely. Fix a choice of preferred lifts $\{\Tilde{x}\}$ of each intersection point $x\in K\cap L$ to the universal cover. We define the based cochain complex
    \begin{equation}
        \underline{CF^*(K,L;\{\Tilde{x}\},J_t)}
    \end{equation}
    as a cochain complex with generators given by the pairs $(\Tilde{x},\mathfrak{o}_x)$, where $\Tilde{x}$ is the chosen lift of $x$, and $\mathfrak{o}_x$ is the associated orientation line of $x$:
    \begin{equation}
        \underline{CF^*(K,L;\{\Tilde{x}\},J_t)}=\bigoplus_{x\in K\cap L} \Z\pi_1(X)\langle(\Tilde{x},\mathfrak{o}_x)\rangle.
    \end{equation}
    The differential is defined by summing over all rigid pseudoholomorphic strips $u\in\mathcal{M}(y;x)$ weighted by $g(u)$, the element of $\pi_1(X)$ associated to each $u$:
    \begin{equation}
        \partial (\Tilde{x},\mathfrak{o}_x)=\sum_{u} g(u) \psi_u(\mathfrak{o}_x)(\Tilde{y},\mathfrak{o}_y),
    \end{equation}
    where $\psi_u:\mathfrak{o}_x\to\mathfrak{o}_y$ is the map on orientation lines induced by $u$.
\end{defn}

Again, the Whitehead torsion of $\underline{CF^*(K,L;\{\Tilde{x}\},J_t)}$ is independent of the chosen lifts $\{\Tilde{x}\}$.

For the sake of completeness, we now explain why the differential of $\underline{CF^*(K,L;\{\Tilde{x}\},J_t)}$, as defined in Definition \ref{defn-bimodule-with-lifts}, squares to zero. In this setting, one must keep track of the group elements $g(u)$ associated to rigid pseudoholomorphic strips $u\in\mathcal{M}(y;x)$, especially when such a strip breaks into a pair $(u_1,u_2)\in\mathcal{M}(y;z)\times\mathcal{M}(z;x)$. As depicted in Figure \ref{fig:breaking-holomorphic-curve}, when a sequence of rigid pseudoholomorphic strips $\{u_k\}$ with fixed $g(u_k)\equiv g(u)$ breaks into $u_1\#u_2$ in the limit, we obtain a relation
\begin{equation}
    g(u)=g(u_2)\cdot g(u_1),
\end{equation}
which shows that $\partial^2=0$ for $\underline{CF^*(K,L;\{\Tilde{x}\},J_t)}$ as required.

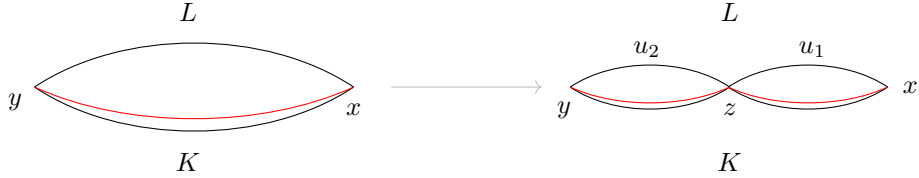
\begin{figure}[h]
\centering
\begin{tikzpicture}

\draw (10,0) arc [start angle=45, end angle=135, x radius=1.5, y radius=1];
\draw (10,0) arc [start angle=-45, end angle=-135, x radius=1.5, y radius=1];
\draw [draw=red] (10,0) arc [start angle=-50, end angle=-130, x radius=1.65, y radius=0.9];

\draw (12.1,0) arc [start angle=45, end angle=135, x radius=1.5, y radius=1];
\draw (12.1,0) arc [start angle=-45, end angle=-135, x radius=1.5, y radius=1];
\draw [draw=red] (12.1,0) arc [start angle=-50, end angle=-130, x radius=1.65, y radius=0.9];
\draw (5,0) arc [start angle=45, end angle=135, x radius=3, y radius=2];
\draw (5,0) arc [start angle=-45, end angle=-135, x radius=3, y radius=2];
\draw [draw=red] (5,0) arc [start angle=-50, end angle=-130, x radius=3.3, y radius=1.8];

\draw (2.8,-1) node {$K$};
\draw (2.8,1) node {$L$};
\draw (5,-0.3) node {$x$};
\draw (0.5,-0.2) node {$y$};

\draw (10,-1) node {$K$};
\draw (10,1) node {$L$};
\draw (12.4,0) node {$x$};
\draw (7.8,-0.3) node {$y$};
\draw (10,-0.3) node {$z$};

\draw (11.1,0.5) node {$u_1$};
\draw (8.9,0.5) node {$u_2$};

\draw[->, draw=lightgray] (5.5, 0) -- (7.5, 0);

\end{tikzpicture}
\caption{A configuration of holomorphic strips that appear in the analysis of $\partial^2=0$ in $\underline{CF^*(K,L;\{\Tilde{x}\})}$. The holomorphic strip $u$ breaks into $u_1\#u_2$, and the red lines represent the elements $g(u)$ and $g(u_1),g(u_2)$ associated to the lifts of the images of $u$ and $u_1,u_2$ in $\Tilde{X}$.}
\label{fig:breaking-holomorphic-curve}
\end{figure}

To compare with the formalism of \cite{AbouzaidKragh}, we may replace the choice of lifts $\Tilde{x}$ with the choice of paths in $X$ from each $x\in K\cap L$ to a fixed basepoint $*\in X$, together with a preferred lift $\Tilde{*}$ to the universal cover. These choices determine a lift $\Tilde{x}$ uniquely, and hence there is a 1:1 correspondence between such choices of paths and choice of lifts to the universal cover. 


As before, each generator $\Tilde{x}$ carries an action
\begin{equation}
    \mathcal{A}(\Tilde{x})=f_L(x)-f_K(x),
\end{equation}
and the differential increases action. Using the compatibility of Whitehead torsion with a filtration (Proposition \ref{prop:Whitehead-torsion-filtration}), we will show that the simple homotopy type of the based cochain complex $\underline{CF^*(K,L;\{\Tilde{x}\},J_t)}$ is independent of the almost complex structure.

\begin{proposition} \label{prop:J-indep}
        The simple homotopy type of the based $\Z\pi_1(X)$-cochain complex $\underline{CF^*(K,L,\{\Tilde{x}\},J_t)}$ is independent of the choice of compatible almost complex structure $J_t$.
\end{proposition}
\begin{proof}
    Suppose we are given two regular almost complex structures $J_t,J_t'$. We can choose a 1-parameter family of regular almost complex structures $J_\lambda$ interpolating between $J_t$ and $J_t'$, and consider the moduli space of solutions to the continuation map equation
    \begin{equation}
        \frac{\partial u}{\partial s}+J_\lambda(s,t)\frac{\partial u}{\partial t}=0
    \end{equation}
    to construct a chain homotopy equivalence
    \begin{equation}
        \psi:CF^*(K,L;J_t)\to CF^*(K,L;J_t').
    \end{equation}
    This lifts to a chain homotopy equivalence on the based complexes:
    \begin{equation}
        \Psi:\underline{CF^*(K,L,\{\Tilde{x}\};J_t})\to \underline{CF^*(K,L,\{\Tilde{x}\};J_t'}).
    \end{equation}
    Our goal is to show that $\Psi$ is a simple homotopy equivalence. To do so, consider the mapping cone
    \begin{equation}
        cone(\Psi)=\underline{CF^*(K,L,\{\Tilde{x}\};J_t})[1]\oplus \underline{CF^*(K,L,\{\Tilde{x}\};J_t'}).
    \end{equation}
    The generators of the above cochain complex consist of two copies of each generator $(\Tilde{x},\mathfrak{o}_x)$ in adjacent degrees. Because the contributions of the nonconstant solutions to the continuation map $\Psi$ strictly increase the action, the differential of $cone(\Psi)$ is upper triangular with respect to the action filtration. In particular, we can choose a finite length filtration such that the differentials in each associated graded piece arise only from constant solutions. Each graded piece is a direct sum of complexes of the form
    \begin{equation}
        [(\Tilde{x},\mathfrak{o}_x)\to(\Tilde{x},\mathfrak{o}_x)],
    \end{equation}
    with the map being $\pm\id$. Therefore each graded piece is acyclic with trivial Whitehead torsion, so by Proposition \ref{prop:Whitehead-torsion-filtration}, it follows that the total complex $cone(\Psi)$ is acyclic with trivial Whitehead torsion. Hence, $\Psi$ is a simple homotopy equivalence.
\end{proof}

The cochain complex $\underline{CF^*(K,L)}$ carries an $A_\infty$-bimodule structure over the compact Fukaya category $\mathcal{F}(X)$. Using Definition \ref{defn:bimodule-without-lifts}, this structure is straightforwardly defined: since the domains of the maps $u\in\mathcal{M}_k(y;x_0,\cdots,x_{k-1})$ are contractible, each map lifts uniquely to the universal cover once a lift $\Tilde{y}$ of the output asymptotic is chosen. Counting these lifted holomorphic discs defines the $A_\infty$-operations on $\underline{CF^*(K,L)}$, and the $A_\infty$-bimodule relations follow from the uniqueness of lifts together with the standard analysis of the codimension 1 boundary strata of the relevant moduli spaces.

Again for completeness, we will explain how the $A_\infty$-bimodule structure for $\underline{CF^*(K,L;\{\Tilde{x}\})}$ as a based cochain complex. For simplicity, we only describe the right $A_\infty$-module structure: the same argument generalizes to the whole $A_\infty$-bimodule structure. We count rigid holomorphic maps $u\in\mathcal{M}(m';m,x_{k-1},\cdots,x_1)$, as shown in the middle of Figure \ref{fig:breaking}. 

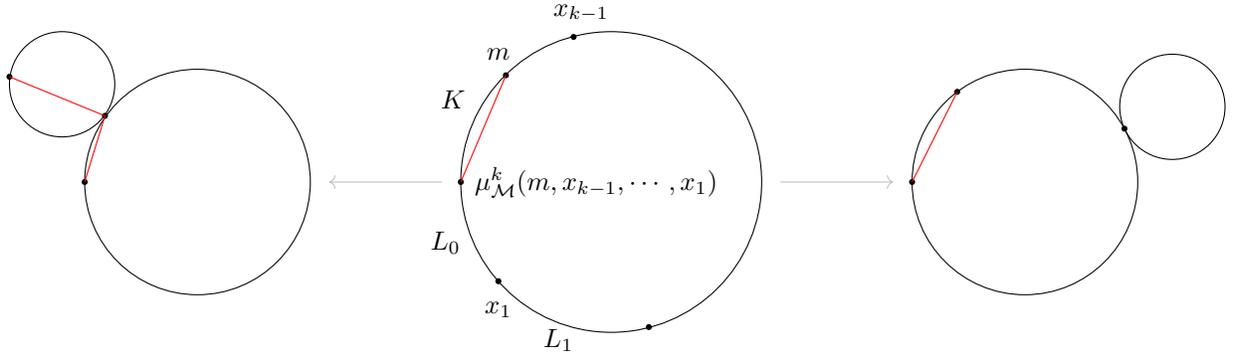
\begin{figure}[h]
\centering
\begin{tikzpicture}

\draw [draw=black] (0,0) circle (2 and 2);
\draw [draw=none, fill=black] (-2,0) circle (0.04 and 0.04);
\draw [draw=none, fill=black] (-1.5,-1.32) circle (0.04 and 0.04);
\draw [draw=none, fill=black] (-1.4,1.42) circle (0.04 and 0.04);
\draw [draw=none, fill=black] (-0.5,1.93) circle (0.04 and 0.04);
\draw [draw=none, fill=black] (0.5,-1.93) circle (0.04 and 0.04);
\draw [draw=red] (-2,-0) -- (-1.4,1.42);

\draw [draw=black] (5.5,0) circle (1.5 and 1.5);
\draw [draw=black] (7.46,1) circle (0.7 and 0.7);
\draw [draw=none, fill=black] (4,0) circle (0.04 and 0.04);
\draw [draw=none, fill=black] (6.823,0.71) circle (0.04 and 0.04);
\draw [draw=none, fill=black] (4.6,1.2) circle (0.04 and 0.04);
\draw [draw=red] (4,-0) -- (4.6,1.2);

\draw [draw=black] (-5.5,0) circle (1.5 and 1.5);
\draw [draw=black] (-7.3,1.3) circle (0.7 and 0.7);
\draw [draw=none, fill=black] (-7,0) circle (0.04 and 0.04);
\draw [draw=none, fill=black] (-6.73,0.88) circle (0.04 and 0.04);
\draw [draw=none, fill=black] (-8,1.4) circle (0.04 and 0.04);
\draw [draw=red] (-7,0) -- (-6.73,0.88);
\draw [draw=red] (-6.73,0.88) -- (-8,1.4);

\draw[->, draw=lightgray] (2.25, 0) -- (3.75, 0);
\draw[->, draw=lightgray] (-2.25, 0) -- (-3.75, 0);

\draw (-2.1,1.1) node {$K$};
\draw (-2.2,-0.8) node {$L_0$};
\draw (-0.7,-2.1) node {$L_1$};
\draw (-1.5,1.7) node {$m$};
\draw (-1.5,-1.7) node {$x_1$};
\draw (-0.4,2.25) node {$x_{k-1}$};
\draw (-0.2,-0) node {$\mu^k_\mathcal{M}(m,x_{k-1},\cdots,x_1)$};

\end{tikzpicture}
\caption{The moduli space of holomorphic disks that contributes to the right $A_\infty$-module structure equations of $\underline{CF^*(\quad,K)}$. The middle picture draws the case when there are $k$ inputs, and the red line depicts the homotopy class that will determine the $\Z\pi_1(X)$ term $g(u)$, after lifting to the universal cover. The left and right pictures depict two possible breakings of this holomorphic disc which determines $A_\infty$-module relations: one can check that the homotopy class of the red curve rel boundary is preserved in this process.}
\label{fig:breaking}
\end{figure}

To be precise, we define maps
\begin{equation}
    \mu^k:\underline{CF^*(L_{k-1},K)}\otimes CF^*(L_{k-2},L_{k-1})\otimes\cdots\otimes CF^*(L_0,L_1)\to\underline{CF^*(L_0,K)}
\end{equation}
which maps the element $(m,x_{k-1},\cdots,x_1)$ to the sum over all rigid holomorphic maps $u\in\mathcal{M}(m';m,x_{k-1},\cdots,x_1)$:
\begin{equation}
    \mu^k(m,x_{k-1},\cdots,x_1)=\sum_u \psi(\mathfrak{o}_m)g(u) (\Tilde{m}',\mathfrak{o}_{m'}),
\end{equation}
where $g(u)$ is defined by the lift of $u$ determined as follows: we consider the chosen lift $\Tilde{m}$ of $m\in K\cap L_{k-1}$, and pick the lift $\Tilde{u}$ such that the asymptotics of the strip-like boundary associated to $K\cap L_{k-1}$ converges to $\Tilde{m}$. By comparing the asymptotics of the unique negative strip-like end of $\Tilde{u}$ to the chosen lift $\Tilde{m}'$, we can associate an element $g(u)\in\pi_1(X)$ to each $u\in\mathcal{M}(m';m,x_{k-1},\cdots,x_1)$. By analyzing the codimension 1 boundary strata of the moduli space $\mathcal{M}(m';m,x_{k-1},\cdots,x_1)$, one can show that the $A_\infty$-module relations hold.

We now show the invariance of the simple homotopy type of $\underline{CF^*(K,L)}$ under compactly supported Hamiltonian isotopies of $K$ and $L$. Recall from Proposition \ref{prop:Ham-invariance} that if $L_0$ is an exact Lagrangian and $L_1$ is its image under a compactly supported Hamiltonian isotopy, there is a chain homotopy equivalence
\begin{equation}
    \mu^2(\alpha,~~):CF^*(K,L_0)\to CF^*(K,L_1)
\end{equation}
where $\alpha\in CF^*(L_0,L_1)$ is the isomorphism element determined by the Hamiltonian isotopy. 

In the following proposition, we will show that this chain homotopy lifts to a simple homotopy equivalence between the two associated based cochain complexes $\underline{CF^*(K,L_0)}$ and $\underline{CF^*(K,L_1)}$. Before stating the proposition, we explain how the bases are chosen. The lifts of the generators $K\cap L_1$ are determined by the lifts of $K\cap L_0$. Specifically, the Hamiltonian isotopy $\phi_H$ mapping $L_0$ to $L_1$ lifts to the universal cover $\Tilde{\phi}_H:\Tilde{X}\to\Tilde{X}$: we define the preferred lifts of $K\cap L_1$ to be the images under $\Tilde{\phi}_H$ of the preferred lifts of $K\cap L_0$.

\begin{proposition} \label{prop:Ham-isotopy-simple}
    Let $K$, $L$ be exact Lagrangians, and let $\phi_H$ be a compactly supported Hamiltonian isotopy which induces an isomorphism element $\alpha\in CF^0(L,\phi_H(L))$. Then the chain map
    \begin{equation}
       \mu^2(\alpha,~~): \underline{CF^*(K,L)}\to\underline{CF^*(K,\phi_H(L))}
    \end{equation}
    induced from the $A_\infty$-bimodule structure is a simple homotopy equivalence.
\end{proposition}
\begin{proof}
    Since compositions of simple homotopy equivalences are again simple homotopy equivalences, it suffices to prove the claim for sufficiently small Hamiltonian isotopies. By the argument in Proposition \ref{prop:Ham-invariance}, we may perturb the Hamiltonian isotopy relative to its endpoints such that $L_t=\phi_t(L)$ is transverse to $K$ for all but finitely many values of $t$, and each $L_t$ and $L_s$ are also pairwise transverse for $t\neq s$. Thus, we may reduce to the case where $L_{t-\epsilon}$, $L_t$ and $L_{t+\epsilon}$ are pairwise transverse, and both $L_{t\pm\epsilon}$ are transverse to $K$. The value of $\epsilon$ will be chosen small enough to satisfy the action estimates required in the argument below.

    For each intersection $x\in K\cap L_t$, choose disjoint open neighborhoods $U_x$ such that all intersection points $K\cap L_{t-\epsilon}$ and $K\cap L_{t+\epsilon}$ lie in the disjoint union $\sqcup U_x$. Now for $y,y'\in K\cap L_{t-\epsilon}$, $z,z'\in K\cap L_{t+\epsilon}$, $p\in L_{t-\epsilon}\cap L_{t+\epsilon}$, we consider the moduli space $\mathcal{M}(z;y,p)$. We assume the intersection points $z$ and $y$ to be contained in some neighborhood $U_x$ and $U_{x'}$ for some $x,x'\in K\cap L$.

    We claim that for sufficiently small $\epsilon$, there exists some constant $\delta>0$ such that
    \begin{enumerate}
        \item If $z$ and $y$ do not belong in the same neighborhood $U_x$, then any $u\in\mathcal{M}(z;y,p)$ has energy at least $\delta$.
        \item If $z,y$ both belong in the same neighborhood $U_x$, then the element $g(u)$ associated to any $u\in\mathcal{M}(z;y,p)$ with energy less than $\delta$ is the identity element $1\in\pi_1(X)$. 
        \item If $z,z'$ both belong in the same neighborhood $U_x$, then the element $g(u)$ associated to any $u\in\mathcal{M}(z;z')$ with energy less than $\delta$ is the identity element $1\in\pi_1(X)$. 
    \end{enumerate}

    This follows from the monotonicity lemma (Lemma \ref{lem:monotonicity-AK}). To apply this lemma, we fix our time-dependent almost complex structure $\{J_t\}$ for the moduli space $\mathcal{M}(z;y,p)$ to be constant on each open neighborhood $U_x$. For small enough $\epsilon$, both $L_{t-\epsilon}$ and $L_{t+\epsilon}$ belong in a $\delta$-neighborhood of $L_t$, and so the conditions of the monotonicity lemma are satisfied. For the second and third statements, we consider the lift of $u$ to the universal cover, and apply Lemma \ref{lem:monotonicity-AK}.

    We first consider the case when birth-death bifurcations do not happen at time $t$. Filter the enhanced mapping cone 
    \begin{equation}
        cone(\mu^2(\alpha,~~))=\underline{CF^*(K,L_{t-\epsilon})}[1]\oplus\underline{CF^*(K,L_{t+\epsilon})}
    \end{equation}
    by action, so that all generators corresponding to the intersection points in the same neighborhood $U_x$ belong in the same graded piece. Thus, the graded pieces consist of direct sums of the form
    \begin{equation}\label{eqn:Ham-isotopy-graded-piece}
    \bigoplus_{y\in U_x}\Z\pi_1(X)\langle\mathfrak{o}_y\rangle[1]\bigoplus_{z\in U_x}\Z\pi_1(X)\langle\mathfrak{o}_z\rangle,
    \end{equation}
    possibly for different $x\in K\cap L_t$ with small action difference. Furthermore, by the above choice of $\delta$, any holomorphic curve $u$ contributing to $\mu^2(\alpha,~~)$ restricted to this graded piece has $g(u)=1$. Thus, the graded piece (\ref{eqn:Ham-isotopy-graded-piece}) is isomorphic to direct sums of
    \begin{equation}\label{eqn:Ham-graded-piece}
         (\bigoplus_{y\in U_x}\Z\langle\mathfrak{o}_y\rangle[1]\bigoplus_{z\in U_x}\Z\langle\mathfrak{o}_z\rangle)\otimes_\Z \Z\pi_1(X),
    \end{equation}
    possibly for distinct $x\in K\cap L_t$ with small action difference. Now the left factor of the tensor product (\ref{eqn:Ham-graded-piece}) is a graded piece of the mapping cone $cone(\mu^2(\alpha,~~))$, which is acyclic. Therefore by Lemma \ref{lem:zero-torsion}, each graded piece (\ref{eqn:Ham-graded-piece}) of $cone(\mu^2(\alpha,~~))$ is acyclic with trivial Whitehead torsion, and thus the total mapping cone $cone(\mu^2(\alpha,~~))$ is acyclic with trivial Whitehead torsion.

    Now consider the case when a unique birth-death bifurcation happens for $K\cap L_t$ at time $t$. Assume that for some open neighborhood $U_x$ of an intersection point $x\in K\cap L_t$, $U_x\cap K\cap L_{t-\epsilon}$ is empty, and $U_x\cap K\cap L_{t+\epsilon}$ consists of two points $z,z'$ in adjacent degrees with action difference smaller than $\epsilon$ chosen above. Then by the monotonicity lemma again, the graded piece of the action filtration of $\underline{CF^*(K,L_{t-\epsilon})}[1]\oplus\underline{CF^*(K,L_{t+\epsilon})}$ has a direct summand
    \begin{equation}
    (\Z\langle\mathfrak{o}_z\rangle\oplus\Z\langle\mathfrak{o}_{z'}\rangle)\otimes_\Z \Z\pi_1(X),
    \end{equation}
    which is acyclic. Therefore by Lemma \ref{lem:zero-torsion}, this piece has trivial Whitehead torsion. The complement of this direct summand can be identified as the graded pieces of the enhanced mapping cone of $\mu^2(\alpha,~)$ as before, so it is acyclic with trivial Whitehead torsion. Thus we conclude that again the homotopy equivalence is simple.
\end{proof}

With this invariance of simple homotopy type, we may define $\underline{CF^*(K,L)}$ for general pairs of Lagrangians.

\begin{defn} \label{def:bimodule-general}
    Let $K$, $L$ be exact Lagrangians. Given a Floer datum $(H,J)$ such that $CF^*(K,L)$ is defined, we define $\underline{CF^*(K,L)}$ as
    \begin{equation}
    \underline{CF^*(K,L;H,J)}\coloneqq\underline{CF^*(K,\phi_H(L);J}).
\end{equation}
\end{defn}
By our previous arguments, the simple homotopy type of $\underline{CF^*(K,L)}$ is independent of the Floer datum chosen for the pair $(K,L)$.

We now specialize to the case $K=L$. Since we proved that compactly supported Hamiltonian isotopies induce simple homotopy equivalences between $\underline{CF^*(K,\phi_H(K))}$ for different Hamiltonian isotopies $\phi_H$, we may use any such isotopy to perturb $K$ to be transverse to itself. The definition above then becomes
\begin{equation}
    \underline{CF^*(K,K)}=\underline{CF^*(K,\phi_H(K))},
\end{equation}
and the resulting simple homotopy type is independent of the choice of Hamiltonian $H$. In particular, we may choose $H$ to be $C^2$-small. Then by Floer's argument (Proposition \ref{prop:Floer-to-Morse-C2}), we can select an almost complex structure $J$ such that $J$-holomorphic strips with boundary on $K$ and $\phi_H(K)$ correspond to gradient flowlines of $H$. This gives a concrete model for $\underline{CF^*(K,K)}$ in terms of Morse theory.

\begin{defn}
    Let $K$ be a closed exact Lagrangian submanifold of $X$ equipped with a Morse function $h:K\to\R$ and a Morse-Smale metric $g$ on $K$. The \emph{enhanced Morse cochain complex}
    \begin{equation}
    \underline{CF^*(K,K)}=\underline{CM^*(K,h,g)}
    \end{equation}
    is defined as follows. Choose a preferred lift $\Tilde{x}$ in the universal cover $\Tilde{X}$ for each critical point $x$ of the Morse function $h$. For each gradient flowline $\gamma$ from $x$ to $y$, define an element $g(\gamma)$ of $\pi_1(X)$ such that the lift $\Tilde{\gamma}$ starts at $\Tilde{x}$ and ends at $g(\gamma)\cdot\Tilde{y}$. We also associate to each critical point $x$ an orientation line $\mathfrak{o}_x$ as before.
    
    With this data, the cochain complex is defined as
    \begin{equation}
        \underline{CM^*(K,h,g,\{\Tilde{x}\})}=\bigoplus_{x\in crit(h)} \Z\pi_1(X)\langle (\Tilde{x},\mathfrak{o}_x)\rangle,
    \end{equation}
    with the differential
    \begin{equation}
        \partial (\Tilde{x},\mathfrak{o}_x) = \sum_{\gamma\in\mathcal{M}(x,y)} g(\gamma) \psi_\gamma (\Tilde{y},\mathfrak{o}_y),
    \end{equation}
    where $\psi_\gamma:\mathfrak{o}_x\to\mathfrak{o}_y$ is the map on orientation lines induced by $\gamma$.
\end{defn}

We first remark that the enhanced Morse complex is defined over $\Z\pi_1(X)$-coefficients, rather than $\Z\pi_1(K)$: this is because we seek a simple homotopy equivalence 
\begin{equation}
    \underline{CF^*(K,K)}\simeq\underline{CM^*(K)}.
\end{equation}
As before, one may also define the enhanced Morse complex without explicitly choosing lifts; the simple homotopy type is independent of this choice. Moreover, continuation maps between complexes defined from different Morse-Smale pairs again induce simple homotopy equivalences. Therefore, we may omit the choice of Morse-Smale pair and write $\underline{CM^*(K)}$.

Since the Morse-Smale pair $(h,g)$ defines a cellular decomposition of $K$, the enhanced Morse complex $\underline{CM^*(K,h,g)}$ may be used to define the cellular cochain complex of $K$ with chosen lifts to $\Tilde{X}$. We abuse notation, and write the simple homotopy equivalence as
\begin{equation} \label{eqn:morse-CW}
    \underline{CM^*(K,h,g)}\simeq\underline{C^*_{cell}(K)}.
\end{equation}
We emphasize again that the right-hand side is also a cochain complex over $\Z\pi_1(X)$. In the case where the inclusion $K\xhookrightarrow{}X$ induces an injection on fundamental groups, this cellular cochain complex can also be used to compute the Reidemeister torsion of $K$.

We summarize the above discussion in the following proposition.
\begin{proposition} \label{prop:Morse=CW}
    For any closed exact Lagrangian $K$, there is a simple homotopy equivalence
    \begin{equation}
        \underline{CF^*(K,K)}\simeq\underline{C^*_{cell}(K)}
    \end{equation}
    as based cochain complexes over $\Z\pi_1(X)$.
\end{proposition}

Now that we have established the properties of $\underline{CF^*(K,L)}$ for single Lagrangians, we extend the definition to twisted complexes $\mathcal{K}=\bigoplus_\alpha C^*_\alpha\otimes K_\alpha$, $\mathcal{L}=\bigoplus D^*_\beta\otimes L_\beta$ in $Tw_{Ch}\mathcal{F}(X)$. Before we define the most general case, we look at the toy case when all $C^*_\alpha$ and $D^*_\beta$ are just freely generated rank $1$ $\Z$-modules. Suppose that
\begin{equation}
    \mathcal{K}=K_1[d_1]\oplus\cdots\oplus K_m[d_m], ~\mathcal{L}=L_1[e_1]\oplus\cdots\oplus L_n[e_n]
\end{equation}
are twisted complexes of exact Lagrangians, equipped with differentials $\delta^\mathcal{K}_{ab}$, $\delta^\mathcal{L}_{ab}$. Then we may define the value of the $A_\infty$-bimodule over $Tw_{Ch}\mathcal{F}(X)$ associated to $\mathcal{K},\mathcal{L}$ as
\begin{equation}
    \underline{CF^*(\mathcal{K},\mathcal{L})}=\bigoplus_{x\in K_a\cap L_b} \Z\langle(\Tilde{x},\mathfrak{o}_x)\rangle
\end{equation}
for all choices of lifts $\Tilde{x}$ for each $x\in K_a\cap L_b$. To define the differential, we first recall that for an element $x\in CF^*(K_{i_0},L_{j_0})$, we define the differential of $\hom_{Tw}(K_{i_0},L_{j_0})$ by
\begin{equation}
    \mu^1_{Tw}(x)=\sum_{d\geq1} \mu^d(\delta^\mathcal{L},\cdots,\delta^\mathcal{L},x,\delta^\mathcal{K},\cdots,\delta^\mathcal{K}).
\end{equation}
Geometrically, this can be represented as the count of rigid holomorphic maps whose domain is $S_d$, the disc with $d+1$ boundary punctures. We consider holomorphic maps with Lagrangian boundary conditions
\begin{equation}
    u:(S_d,\partial S_d)\to(X;K_{i_k},\cdots,K_{i_0},L_{j_0},\cdots,L_{j_l}),
\end{equation}
and equip $S_d$ with choices of strip-like ends $\epsilon_k$ near each boundary puncture, such that there are $d$ positive strip-like ends, and 1 negative strip-like end equipped at the point $-1\in S_d$. At the strip-like end that corresponds to the intersection point of $K_{i_0}\cap L_{j_0}$, $u$ is assumed to converge to the intersection point $x$. On the other strip-like ends, $u$ is assumed to converge to a point that corresponds to elements in $CF^0(K_{i_a},K_{i_{a+1}})$ or $CF^0(L_{j_b},L_{j_{b+1}})$ that contribute to the differential of the twisted complex $\delta^\mathcal{K}_{a,a+1}$ and $\delta^\mathcal{L}_{b,b+1}$. See Figure \ref{fig:twisted-cpx} for a picture. Then we define the differential of the cochain complex as a sum over all possible lifts $\Tilde{u}$ of rigid holomorphic maps $u$:

\begin{equation}
    \underline{\mu^1_{Tw}(x)}=\sum_{d\geq1}\sum_{u:S_d\to M} g(u) \mu_{\Tilde{u}}^d(\delta^\mathcal{L},\delta^\mathcal{L},\cdots,\delta^\mathcal{L},x,\delta^\mathcal{K},\cdots,\delta^\mathcal{K}),
\end{equation}
where $\mu^d_u$ is the contribution of the map $\Tilde{u}$ to the count $\mu^d$ in $\Tilde{X}$.

\begin{figure}[h]
\centering
\begin{tikzpicture}

\draw (5.5,0) arc [start angle=45, end angle=135, x radius=6, y radius=4];
\draw (5.5,0) arc [start angle=-45, end angle=-135, x radius=6, y radius=4];
\draw [draw=red] (5.5,0) arc [start angle=-50, end angle=-130, x radius=6.6, y radius=3.6];
\draw [draw=none, fill=black] (-1.4,0.75) circle (0.04 and 0.04);
\draw [draw=none, fill=black] (3.95,0.75) circle (0.04 and 0.04);
\draw [draw=none, fill=black] (-1.15,-0.84) circle (0.04 and 0.04);
\draw [draw=none, fill=black] (4.15,-0.68) circle (0.04 and 0.04);

\draw (1.3,-1.8) node {$\mathcal{K}$};
\draw (1.3,1.8) node {$\mathcal{L}$};

\draw (5.2,-0.67) node {$K_{i_0}$};
\draw (-2.5,-0.75) node {$K_{i_k}$};
\draw (5.1,0.67) node {$L_{j_0}$};
\draw (-2.4,0.75) node {$L_{j_l}$};
\draw (3.95,1.2) node {$\delta^{\mathcal{L}}_{j_0j_1}$};
\draw (-1.4,1.4) node {$\delta^{\mathcal{L}}_{j_{l-1}j_l}$};
\draw (4.2,-1.2) node {$\delta^{\mathcal{K}}_{i_0i_1}$};
\draw (-1.1,-1.3) node {$\delta^{\mathcal{K}}_{i_0i_1}$};

\draw (5.7,0) node {$x$};
\draw (-3.2,0) node {$y$};

\end{tikzpicture}
\caption{A holomorphic curve with Lagrangian boundary conditions that contributes to the differential of the twisted complex $\underline{CF^*(\mathcal{K},\mathcal{L})}$. Here, $x\in K_{i_0}\cap L_{j_0}$, and $y\in K_{i_k}\cap L_{j_l}$. The Lagrangian boundary conditions for the holomorphic curve are $K_{i_0}\cup K_{i_1}\cup\cdots\cup K_{i_k}$ and vice versa for $L$, which must appear exactly in this order. Allowed intersection points of $K_{i_a}\cap K_{i_{a+1}}$ are the ones that contribute to the differential of the twisted complex $\delta^\mathcal{K}_{a,a+1}$.}
\label{fig:twisted-cpx}
\end{figure}
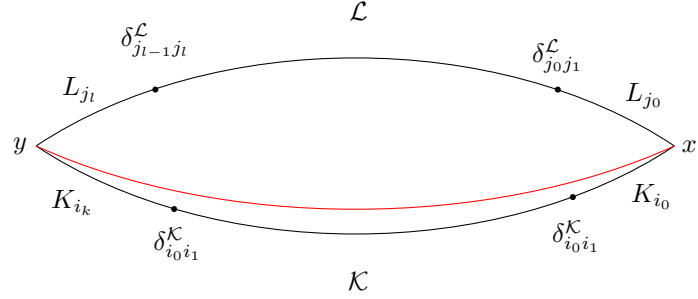

One can verify that this construction defines a differential, so $\underline{CF^*(\mathcal{K},\mathcal{L})}$ forms a cochain complex. As before, a choice of preferred lifts $\{\Tilde{x}\}$ equips $\underline{CF^*(\mathcal{K},\mathcal{L)}}$ with the structure of a based cochain complex over $\Z\pi_1(X)$, and its simple homotopy type is independent of the choice of lifts. Moreover, $\underline{CF^*(\mathcal{K},\mathcal{L})}$ inherits an $A_\infty$-bimodule structure over $Tw\mathcal{F}(X)$, defined by counting lifted holomorphic discs with appropriate boundary conditions.

The proof that the simple homotopy type of $\underline{CF^*(\mathcal{K},\mathcal{L},\{\Tilde{x}\})}$ is independent of the almost complex structure follows from the same argument as in Proposition \ref{prop:J-indep}. In particular, we again use that continuation maps count pseudoholomorphic curves with positive energy. 

We are now ready to move onto the general case. 

\begin{defn}
    Let $\mathcal{K}=\bigoplus_\alpha C^*_\alpha\otimes K_\alpha$ and $\mathcal{L}=\bigoplus D^*_\beta\otimes L_\beta$ be two objects in $Tw_{Ch}\mathcal{F}(X)$. We define the value of the $A_\infty$-bimodule $\underline{CF^*(\mathcal{K},\mathcal{L}})$ as
    \begin{equation}
         \underline{CF^*(\mathcal{K},\mathcal{L})}=\bigoplus_{x\in K_\alpha\cap L_\beta} \Z \langle\phi_{\alpha\beta}\otimes(\Tilde{x},\mathfrak{o}_x)\rangle,
    \end{equation}
    where $\Tilde{x}$ ranges over all possible lifts of each intersection point $x\in K_\alpha\cap L_\beta$, and $\{\phi_{\alpha\beta}\}$ is a chosen $\Z$-basis for $\hom_\Z(C^*_\alpha,D^*_\beta)$. The differential is given by
    \begin{equation}
        \underline{\mu^1_{Tw_{Ch}}(\phi\otimes (\Tilde{x},\mathfrak{o}_x))}=\partial_{Ch}\phi\otimes(\Tilde{x},\mathfrak{o}_x)+(-1)^{\deg\phi-1}\phi\otimes\underline{\mu^1_{Tw}(\Tilde{x})},
    \end{equation}
    where $\underline{\mu^1_{Tw}}$ is the lifted Floer differential previously defined, and $\partial_{Ch}(\phi)=\partial_D\circ\phi-(-1)^{\deg\phi}\phi\circ\partial_C$ is the differential in the DG category of cochain complexes.
\end{defn}

By fixing a preferred lift $\Tilde{x}$ of each intersection point $x\in K_\alpha\cap L_\beta$, we obtain a distinguished basis $\{\phi_{\alpha\beta}\otimes \Tilde{x}\}$, which identifies $\underline{CF^*(\mathcal{K},\mathcal{L})}$ as a based cochain complex over $\Z\pi_1(X)$. We now show that its simple homotopy type does not depend on these choices.

\begin{lemma} \label{lem:basis-invariance-bimodule}
    The simple homotopy type of the based cochain complex 
    \begin{equation}
        \bigl(\underline{CF^*(\mathcal{K},\mathcal{L})},\{\phi_{\alpha\beta}\otimes \Tilde{x}\}\bigr)
    \end{equation}
    is independent of the choice of bases $\{\phi_{\alpha\beta}\}$ for each $\hom_\Z(C^*_\alpha,D^*_\beta)$ and of the choice of lifts $\Tilde{x}$ for $x\in K_\alpha\cap L_\beta$.
\end{lemma}

\begin{proof}
    We separate the two kinds of choices. First, suppose a different lift $g\cdot \Tilde{y}$ is chosen for some $y\in K_\alpha\cap L_\beta$ and $g\in\pi_1(X)$. Then for each $\alpha,\beta$, the basis element $\phi_{\alpha\beta}\otimes \Tilde{y}$ is replaced by $g\cdot(\phi_{\alpha\beta}\otimes \Tilde{y})$. In other words, a single change of lift from $\tilde{y}$ to $g\cdot\tilde{y}$ induces simultaneously $\sum_{\alpha,\beta}\mathrm{rk}(\hom_\Z(C^*_\alpha,D^*_\beta))$ many elementary basis changes to the induced basis $\{\phi_{\alpha\beta}\otimes \Tilde{x}\}$. Thus the two based complexes given by $\underline{CF^*(\mathcal{K},\mathcal{L})}$ with the two different bases are related by simple operations, and their simple homotopy types agree.

    Next, suppose we choose a different basis $\{\phi'_{\alpha\beta}\}$ for $\hom_\Z(C^*_\alpha,D^*_\beta)$. The two bases $\{\phi_{\alpha\beta}\}$ and $\{\phi'_{\alpha\beta}\}$ are related by a $\Z$-valued change-of-basis matrix. This matrix defines a sequence of $\Z$-valued row and column operations relating the two bases, which in turn induce elementary simple operations on the corresponding $\Z\pi_1(X)$-bases $\{\phi_{\alpha\beta}\otimes \Tilde{x}\}$ and $\{\phi'_{\alpha\beta}\otimes \Tilde{x}\}$ for $\underline{CF^*(\mathcal{K},\mathcal{L})}$. Hence the resulting based complexes have the same simple homotopy type.
\end{proof}


We now define the $A_\infty$-bimodule structure on $\underline{CF^*(\mathcal{K},\mathcal{L}})$. Recall that the $\mu^k$ operations in $Tw_{Ch}\mathcal{F}(X)$ are inherited from the additive enlargement $\Sigma_{Ch}\mathcal{F}$, then deformed by the Maurer-Cartan elements associated to each object. To simplify notation, we will explain how each $\mu^s_{\Sigma\mathcal{F}}$-term in the $A_\infty$-bimodule operations for $\underline{CF^*(\mathcal{K},\mathcal{L)}}$ is defined.

We begin with the case $s=1$, contributing to the operation $\mu^{0\mid1\mid0}$ from $\underline{CF^*(\mathcal{K},\mathcal{L}})$ to itself. This operation acts on a generator $\phi\otimes \Tilde{x}\in\hom(C^*_\alpha\otimes K_\alpha,D^*_\beta\otimes L_\beta)$ by
\begin{equation}
    \mu^1_{\Sigma\mathcal{F}}(\phi\otimes\Tilde{x})=(\partial_{Ch}\phi)\otimes\Tilde{x}+(-1)^{\deg{\phi}-1}\phi\otimes\underline{\mu^1(\Tilde{x})},
\end{equation}
where $\partial_{Ch}$ is the differential on $\hom_\Z(C^*_\alpha,D^*_\beta)$ induced from the two cochain complexes, and $\underline{\mu^1_{Tw}}$ is the $\mu^{0\mid1\mid0}$ operation defined for $\underline{CF^*(K_\alpha,L_\beta)}$ as an $A_\infty$-bimodule over $\mathcal{F}(X)$.

To define the higher order operations arising from $\mu^s_{\Sigma\mathcal{F}}$ for $s\geq 2$, we fix objects $\mathcal{K}_0,\cdots,\mathcal{K}_k$ and $\mathcal{L}_0,\cdots,\mathcal{L}_l$ in $Tw_{Ch}\mathcal{F}(X)$. We write
\begin{equation}
    \mathcal{K}_i=\bigoplus_{\alpha}C^*_{i,\alpha}\otimes K_{i,\alpha},~ \mathcal{L}_j=\bigoplus_\beta D^*_{j,\beta}\otimes L_{j,\beta}.
\end{equation}
Given generators
\begin{align}
    &\psi_i\otimes y_i\in\hom_{Tw}(C^*_{i,\alpha_i}\otimes K_{i,\alpha_i},C^*_{i-1,\alpha_{i-1}}\otimes K_{i-1,\alpha_{i-1}}),\\
    &\phi\otimes x\in\hom_{Tw}(C^*_{0,\alpha_0}\otimes K_{0,\alpha_0},D^*_{0,\beta_0}\otimes L_{0,\beta_0}),\\ &\nu_j\otimes z_j\in\hom_{Tw}(D^*_{j-1,\beta_{j-1}}\otimes L_{j-1,\beta_{j-1}},D^*_{j,\beta_j}\otimes L_{j,\beta_j}),
\end{align}
the undeformed $A_\infty$-bimodule operation $\mu^{k\mid1\mid l}_{\Sigma\mathcal{F}}$ is given by
\begin{align}
    \underline{\mu^{k\mid1\mid l}_{\Sigma\mathcal{F}}(\psi_k\otimes y_k,\cdots,\psi_1\otimes y_1,\phi\otimes \Tilde{x},\nu_1\otimes z_1,\cdots,\nu_l\otimes z_l)}\\
    =(-1)^{\bowtie} \psi_k\circ\cdots\circ\psi_1\circ\phi\circ\nu_1\cdots\circ\nu_l \otimes\underline{\mu^{k\mid 1\mid l}(y_k,\cdots,y_1,\Tilde{x},z_1,\cdots,z_l)}.
\end{align}
where again $\underline{\mu^{k\mid1\mid l}}$ is the $A_\infty$-bimodule operation from $\underline{CF^*(K_{0,\alpha_0},L_{0,\beta_0})}$ and $\bowtie$ is as in Equation \ref{eqn-signconvention}.

In general, the full $A_\infty$-bimodule structure over $Tw_{Ch}\mathcal{F}(X)$ involves a sum over all insertions of the differentials $\delta$. For brevity, we omit the full expression here. As before, a similar action filtration argument shows that the simple homotopy type of $\underline{CF^*(\mathcal{K},\mathcal{L})}$ is independent of the choice of almost complex structure. 

The issue with Hamiltonian invariance is a little more subtle. Instead of proving a general statement, we settle for the following weaker proposition, whose proof we defer to the following subsection.

\begin{proposition} \label{prop:algebraic-twist-Ham-isotopy}
    Let $V$ be a closed exact Lagrangian in $X$, and let $L$ be any exact Lagrangian. Then for any compactly supported Hamiltonian isotopy $\phi$, there is a simple isomorphism of objects in $Tw_{Ch}\mathcal{F}(X)$:
    \begin{equation}
        T_VL\simeq T_V\phi(L).
    \end{equation}
\end{proposition}

\subsection{Simple categorical notions}

Using the construction of the $A_\infty$-bimodule $\underline{CF^*(\mathcal{K},\mathcal{L})}$ from the previous subsection, we now define ``simple'' analogues of the $A_\infty$-categorical notions introduced in Section \ref{ssec:category-theory}. The central result of this subsection is Proposition \ref{prop:automaticsimplicity}, which states that any categorical isomorphism in a Fukaya category which admits simply-connected simple generators induces a simple homotopy equivalence. We note that some well-studied examples of generators for Fukaya categories, such as Lefschetz thimbles \cite{Seidel08} or cocores to critical Weinstein handles \cite{GPS} are simply connected. 

In this subsection, the Fukaya category under consideration may refer to either the compact Fukaya category $\mathcal{F}(X)$, or to the Fukaya category of a Lefschetz fibration $\mathcal{F}(\pi)$, as introduced in Section \ref{ssec:Lef-fib}. To treat both cases uniformly, we denote the underlying category by $\mathcal{F}$. In either case, the $A_\infty$-category $\mathcal{F}$ is proper and c-unital.

Recall that an acyclic based cochain complex $(C^*,\partial_*,\{c_i\})$ is defined to be \emph{simply acyclic} if it has trivial Whitehead torsion. More generally, we may define the Whitehead torsion of a simple homotopy type whenever it is represented by a based cochain complex that is acyclic. Given that the simple homotopy type of $\underline{CF^*(K,L)}$ is well-defined, we may similarly define simply acyclic objects in Fukaya categories.

\begin{defn} \label{def:simply-acyclic-obj}
    An object $K$, or more generally a twisted complex $\mathcal{K}$ in $Tw_{Ch}\mathcal{F}$ is \emph{left simply acyclic} if $\underline{CF^*(\mathcal{K},L)}$ is simply acyclic for every object $L$. Similarly, $\mathcal{K}$ is \emph{right simply acyclic} if $\underline{CF^*(L,\mathcal{K})}$ is simply acyclic for any object $L$. We say that $\mathcal{K}$ is \emph{simply acyclic} if it is both left and right simply acyclic.
\end{defn}

We likewise define a (left/right) $A_\infty$-module $\mathcal{M}$ to be a \emph{simply acyclic $A_\infty$-module} if it assigns to each object $L$ a simple homotopy type $\mathcal{M}(L)$ over $\Z\pi_1(X)$ that is simply acyclic. We may extend left/right Yoneda modules to $\Z\pi_1(X)$-valued simple homotopy types as follows:

\begin{defn} \label{def:enhanced-Yoneda-module}
    For any object $K$ or twisted complex $\mathcal{K}$ in the Fukaya category, its \emph{enhanced right Yoneda module} $\underline{\mathcal{Y}_\mathcal{K}^r}$ assigns to each Lagrangian $L$
    \begin{equation}
        \underline{\mathcal{Y}_\mathcal{K}^r}(L)=\underline{CF^*(L,\mathcal{K})},
    \end{equation}
    which is a $\Z$-graded cochain complex over $\Z$ with a well-defined simple homotopy type as a $\Z\pi_1(X)$-coefficient cochain complex. We similarly define the \emph{enhanced left Yoneda module} $\underline{\mathcal{Y}_\mathcal{K}^l}$.
\end{defn}

We will also write $\underline{CF^*(\quad,\mathcal{K})}$ instead of $\underline{\mathcal{Y}_\mathcal{K}^r}$ and $\underline{CF^*(\mathcal{K},\quad)}$ instead of $\underline{\mathcal{Y}_\mathcal{K}^l}$ for clarity. In this language, we may reformulate left simply acyclic objects to be the twisted complexes $\mathcal{K}$ whose enhanced left Yoneda module $\underline{CF^*(\mathcal{K},\quad)}$ is a simply acyclic $A_\infty$-module, and vice versa for right simply acyclic objects.

Now suppose that we have an isomorphism element $\alpha\in CF^0(\mathcal{K},\mathcal{L})$ between two twisted complexes $\mathcal{K}$, $\mathcal{L}$ in $Tw_{Ch}\mathcal{F}(X)$. Then $\alpha$ induces $A_\infty$-module homomorphisms between the enhanced Yoneda modules
\begin{align}
    &l^1(\alpha):\underline{\mathcal{Y}^r_{\mathcal{K}}}\to\underline{\mathcal{Y}^r_\mathcal{L}},\\
    &r^1(\alpha):\underline{\mathcal{Y}^l_\mathcal{L}}\to\underline{\mathcal{Y}^l_\mathcal{K}}.
\end{align}

Recall that the module homomorphism $l^1(\alpha)$ is defined as the collection of maps
\begin{align}
    \underline{CF^*(L_k,\mathcal{K})}\otimes CF^*(L_{k-1},L_k)\otimes\cdots&\otimes CF^*(L_0,L_1)\to\underline{CF^*(L_0,\mathcal{L})},\\
    (y,x_k,\cdots,x_1)&\mapsto \mu^{k+2}(\alpha,y,x_k,\cdots,x_1),
\end{align}
and $r^1(\alpha)$ is similarly defined by counting lifted holomorphic discs contributing to the $A_\infty$-relations. On the level of cochain complexes, we are only interested in the $\mu^2$ term
\begin{equation}
    \mu^2_{Tw_{Ch}}(\alpha,\quad): \underline{\mathcal{Y}^r_\mathcal{K}}\to\underline{\mathcal{Y}^r_\mathcal{L}},
\end{equation}
which is a chain homotopy equivalence. We take its mapping cone
\begin{equation}
    \underline{cone(\mu^2_{Tw_{Ch}}(\alpha,\quad))}=\underline{\mathcal{Y}^r_\mathcal{K}}[1]\oplus\underline{\mathcal{Y}^r_\mathcal{L}}
\end{equation}
and make the following definition:

\begin{defn} \label{def:simple-equivalence}
    An isomorphism element $\alpha$ between two objects $K,L$, or more generally between two twisted complexes $\mathcal{K}$, $\mathcal{L}$ in $Tw_{Ch}\mathcal{F}$ is a \emph{left simple isomorphism} if $\underline{cone(\mu^2_{Tw_{Ch}}(\alpha,\quad))}$ is a simply acyclic $A_\infty$-module. Similarly, $\alpha$ is a \emph{right simple isomorphism} if $\underline{cone(\mu^2_{Tw_{Ch}}(~~,\alpha))}$ is a simply acyclic $A_\infty$-module. We call $\alpha$ a \emph{simple isomorphism} if it is both left and right simple.
\end{defn}

Recall that all isomorphism elements between two objects $\mathcal{K},\mathcal{L}$ in $Tw_{Ch}\mathcal{F}$ lie in the same cohomology class, up to multiplication by a unit of the coefficient ring, provided that either $H^0\hom_{Tw_{Ch}}(\mathcal{K},\mathcal{K})$ or $H^0\hom_{Tw_{Ch}}(\mathcal{L},\mathcal{L})$ is free of rank 1.

\begin{proposition} \label{prop:simple-equivalence-indep-cohomology-class}
    Let $\alpha\in CF^*(\mathcal{K},\mathcal{L})$ be a simple isomorphism element. Then for any $\nu\in CF^*(\mathcal{K},\mathcal{L})$, the element $\alpha+\mu^1_{Tw_{Ch}}(\nu)$ is a simple isomorphism element.
\end{proposition}
\begin{proof}
    By the $A_\infty$-relations, $\mu^1_{Tw_{Ch}}(\nu)$ is a chain homotopy between the homotopy equivalences
\begin{equation}
    \mu^2_{Tw_{Ch}}(\alpha,~~),\mu^2_{Tw_{Ch}}(\alpha+\mu^1_{Tw}\nu,~~):\underline{CF^*(X,\mathcal{K})}\to \underline{CF^*(X,\mathcal{L})}
\end{equation}
    for any object $X$. Since chain homotopic homotopy equivalences between based cochain complexes have the same Whitehead torsion (Proposition \ref{prop:homotopy-invariance-torsion}), it follows that $\alpha+\mu^1_{Tw_{Ch}}\nu$ is a simple isomorphism.
\end{proof}

\begin{proposition} \label{prop:composition-of-simple-equivalences}
    Let $\alpha\in CF^*(\mathcal{K},\mathcal{L})$, $\beta\in CF^*(\mathcal{L},\mathcal{N})$ be simple isomorphism elements. Then the composition $\mu^2_{Tw_{Ch}}(\beta,\alpha)\in CF^*(\mathcal{K},\mathcal{N})$ is also a simple isomorphism element.
\end{proposition}
\begin{proof}
    Since Whitehead torsion is multiplicative under composition (Proposition \ref{prop:Whitehead-torsion-additive}), the homotopy equivalence
    \begin{equation}
        \mu^2_{Tw_{Ch}}(\beta,\mu^2_{Tw_{Ch}}(\alpha,~~)):\underline{CF^*(X,\mathcal{K})}\to \underline{CF^*(X,\mathcal{N})}
    \end{equation}
    has trivial Whitehead torsion for any object $X$. Since $\mu^2_{Tw_{Ch}}(\mu^2_{Tw_{Ch}}(\beta,\alpha),~)$ and $\mu^2_{Tw_{Ch}}(\beta,\mu^2_{Tw_{Ch}}(\alpha,~))$ are chain homotopic, their Whitehead torsions agree, and therefore $\mu^2_{Tw_{Ch}}(\beta,\alpha)$ is a simple isomorphism.
\end{proof}

\begin{proposition} \label{prop:composition-simple}
    Let $\alpha\in CF^*(\mathcal{K},\mathcal{L})$, $\beta\in CF^*(\mathcal{L},\mathcal{N})$ be simple isomorphism elements, and assume that either $H^0\hom^*_{Tw_{Ch}}(\mathcal{K},\mathcal{K})$ or $H^0\hom^*_{Tw_{Ch}}(\mathcal{N},\mathcal{N})$ is free of rank 1. Then any isomorphism element $\gamma\in CF^*(\mathcal{K},\mathcal{N})$ is a simple isomorphism.
\end{proposition}
\begin{proof}
    By Proposition \ref{prop:composition-of-simple-equivalences}, the product $\mu^2_{Tw_{Ch}}(\beta,\alpha)\in CF^*(\mathcal{K},\mathcal{N})$ is a simple isomorphism. Under our assumptions, any two isomorphism elements between $\mathcal{K}$ and $\mathcal{N}$ belong in the same cohomology class up to a unit. Therefore, there exists an element $\nu\in CF^*(\mathcal{K},\mathcal{N})$ and a unit $c\in \Z^\times$, which must be $\pm1$, such that
    \begin{equation}
        \beta\cdot\alpha=c\gamma+\mu^1_{Tw_{Ch}}(\nu).
    \end{equation}
    By Proposition \ref{prop:simple-equivalence-indep-cohomology-class}, $\pm\gamma$ is a simple isomorphism.
\end{proof}

\begin{defn} \label{def:simple-generators}
    A collection of objects $\{L_i\}$ that generate the Fukaya category $\mathcal{F}$, or in general the category of twisted complexes $Tw_{Ch}\mathcal{F}$ are \emph{simple generators} if for each object $K$, there exists a twisted complex $\mathcal{K}$ made out of objects in $\{L_i\}$ such that there is a simple isomorphism
    \begin{equation}
        \mathcal{K}\cong K.
    \end{equation}
\end{defn}

Recall that to check whether a twisted complex $\mathcal{K}$ is an acyclic object, it is enough to check that $\hom(\mathcal{K},L)$ is acyclic for each Lagrangian $L$, rather than for all twisted complexes $\mathcal{L}$. The same is true for simple categorical notions.

\begin{proposition} \label{prop:simple-notions-check-Lagrangian}
    Let $\mathcal{K}$ be a twisted complex in $Tw_{Ch}\mathcal{F}$ such that for any Lagrangian $L\in\mathcal
    F$, $\underline{\hom_{Tw_{Ch}}(\mathcal{K},L)}$ is simply acyclic. Then for any twisted complex $\mathcal{L}$ in $Tw_{Ch}\mathcal{F}$, $\underline{\hom_{Tw_{Ch}}(\mathcal{K},\mathcal{L})}$ is also simply acyclic.
\end{proposition}
\begin{proof}
    Write $\mathcal{K}=\bigoplus_\alpha C^*_\alpha\otimes K_\alpha$. The twisted complex $\mathcal{L}$ comes with a natural filtration whose graded pieces are of the form $D^*_\beta\otimes L_\beta$, where $D^*_\beta$ is a cochain complex and $L_\beta\in\mathcal{F}$. By Proposition \ref{prop:Whitehead-torsion-filtration}, it suffices to show that each graded piece
    \begin{equation}
        \underline{\hom_{Tw_{Ch}}(\mathcal{K},D^*_\beta\otimes L_\beta)}=\bigoplus_\alpha \hom_\Z(C^*_\alpha,D^*_\beta)\otimes\underline{CF^*(K_\alpha,L_\beta)}
    \end{equation}
    is simply acyclic. We define a filtration on $D^*_\beta$ by degree:
    \begin{equation}
        F^iD^k_\beta=\begin{cases}
            D^k_\beta &(k\leq i) \\
            0 &\mathrm{(else)}.
        \end{cases}
    \end{equation}
    This induces a filtration on $ \underline{\hom_{Tw_{Ch}}(\mathcal{K},D^*_\beta\otimes L_\beta)}$, whose graded pieces are
    \begin{equation}
    \bigoplus_\alpha\hom_\Z(C^*_\alpha,\Z^{r_k})\otimes_{\Z}\underline{CF^*(K_\alpha,L_\beta)}\cong (\bigoplus_\alpha\hom_\Z(C^*_\alpha,\Z)\otimes_{\Z}\underline{CF^*(K_\alpha,L_\beta)})^{r_k},
    \end{equation}
    where $r_k$ is the rank of the free $\Z$-module $D^k_\beta$. By the assumption that $\mathcal{K}$ is left simply acyclic against any single Lagrangian, the complex
    \begin{equation}
        \bigoplus_\alpha\hom_\Z(C^*_\alpha,\Z)\otimes_{\Z}\underline{CF^*(K_\alpha,L_\beta)}
    \end{equation}
    is simply acyclic. Hence, each graded piece of $ \underline{\hom_{Tw_{Ch}}(\mathcal{K},D^*_\beta\otimes L_\beta)}$ is simply acyclic, and so by Proposition \ref{prop:Whitehead-torsion-filtration} again the total complex is simply acyclic. Therefore we conclude that $\mathcal{K}$ is a left simply acyclic object in $Tw_{Ch}\mathcal{F}$.
\end{proof}

The heart of these definitions is captured in the following ``automatic simplicity'' lemma.

First, suppose that $L$ is a simply connected Lagrangian submanifold. Then its lift to the universal cover $\Tilde{X}$ consists of disjoint copies of $L$, so any pseudoholomorphic curve $u:\Sigma\to X$ with a boundary component on $L$ lifts to a curve $\Tilde{u}:\Sigma\to\Tilde{X}$ whose boundary contained in a single sheet of $p^{-1}(L)$. Choosing lifts of the intersection points to lie in a single sheet, we obtain an isomorphism of based cochain complexes
\begin{equation}
    \underline{CF^*(K,L)}\cong CF^*(K,L)\otimes_\Z \Z\pi_1(X)
\end{equation}
for any other Lagrangian $K$. A similar argument applies for any twisted complex $\mathcal{K}\in Tw_{Ch}\mathcal{F}$, yielding
\begin{equation}
    \underline{CF^*(\mathcal{K},L)}\cong CF^*(\mathcal{K},L)\otimes_\Z\Z\pi_1(X).
\end{equation}
By Lemma \ref{lem:zero-torsion}, these based cochain complexes have trivial Whitehead torsion when acyclic.

For a general Lagrangian $L$, if the Fukaya category $\mathcal{F}(X)$ admits simply connected simple generators, we can find a twisted complex $\mathcal{L}$ simply isomorphic to $L$ built from such generators. The filtration on $\underline{CF^*(\mathcal{K},\mathcal{L})}$ induced from the filtration of the twisted complex $\mathcal{L}$ has graded pieces $\underline{CF^*(\mathcal{K},C^*_\beta\otimes L_\beta)}$, each of which has trivial Whitehead torsion when acyclic from above. Therefore by the filtration lemma (Proposition \ref{prop:Whitehead-torsion-filtration}), the total complex is acyclic with trivial Whitehead torsion.

We now state the automatic simplicity lemma, and provide the details of the above argument.

\begin{proposition} \label{prop:automaticsimplicity}
    Suppose that the Fukaya category $\mathcal{F}$ admits a collection of objects $\{L_i\}$ such that each $L_i$ is simply connected, and the collection $\{L_i\}$ simply generates $\mathcal{F}$. Then every isomorphism between two objects, or generally two twisted complexes in $Tw_{Ch}\mathcal{F}$ is automatically a simple isomorphism.
\end{proposition}

\begin{proof}
    Recall that an isomorphism element $\gamma$ is a simple isomorphism if its mapping cone is both left and right simply acyclic. Thus, it suffices to show that every acyclic object in $Tw_{Ch}\mathcal{F}$ is simply acyclic. So pick any acyclic object $\mathcal{K}=\bigoplus_\alpha C^*_\alpha\otimes K_\alpha$ in $Tw_{Ch}\mathcal{F}$: we will show that it is left simply acyclic, the right simple acyclicity is analogous.
    
    We first check this for the generators $L_i$. Since each $L_i$ is simply connected, the preimage $p^{-1}(L_i)$ in the universal cover $p:\Tilde{X}\to X$ is a disjoint union of copies of $L_i$. Choosing lifts of the intersection points $K\cap L_i$ to lie in a single sheet, we have an isomorphism of based cochain complexes
    \begin{equation}\label{eqn: autsimp-checkgenerators}
        \underline{CF^*(\mathcal{K},L_i)}\cong \hom_{Tw_{Ch}}^*(\mathcal{K},L_i)\otimes_\Z\Z\pi_1(X).
    \end{equation}
    The right complex has a basis $\{\phi_\alpha\otimes x_\alpha\}$,where $\{\phi_\alpha\}$ is a $\Z$-basis of $\hom_\Z(C^*_\alpha,\Z)$ and $x_\alpha\in K_\alpha\cap L_i$. Since $\mathcal{K}$ is an acyclic object in $Tw_{Ch}\mathcal{F}$, the right-hand side of (\ref{eqn: autsimp-checkgenerators}) is acyclic, and by Lemma \ref{lem:zero-torsion} it has trivial Whitehead torsion. Therefore $\underline{CF^*(\mathcal{K},L_i)}$ is acyclic with trivial Whitehead torsion, and by Lemma \ref{lem:zero-torsion-general} $\underline{CF^*(\mathcal{K},D^*\otimes L_i)}$ is acyclic with trivial Whitehead torsion for any finitely generated cochain complex $D^*$ over $\Z$.

    Now consider a general Lagrangian $L$. Since the collection $\{L_i\}$ simply generates $Tw_{Ch}\mathcal{F}$, there exists a twisted complex $\mathcal{L}\in Tw_{Ch}\mathcal{F}$ made out of the objects $L_i$ such that there is a simple isomorphism
    \begin{equation} \label{eqn:auto-simpleiso}
        L\cong \mathcal{L} = \bigoplus_\beta D^*_\beta\otimes L_\beta
    \end{equation}
    in $Tw_{Ch}\mathcal{F}$. This induces a simple homotopy equivalence
    \begin{equation}
        \underline{CF^*(\mathcal{K},L)}\simeq\underline{CF^*(\mathcal{K},\mathcal{L})}
    \end{equation}
    given by the module action of the isomorphism element of the simple isomorphism (\ref{eqn:auto-simpleiso}). Therefore, it is enough to show that the cochain complex $\underline{CF^*(\mathcal{K},\mathcal{L})}$ has trivial Whitehead torsion.

    As a twisted complex, $\mathcal{L}$ admits a filtration whose graded pieces are $D^*_\beta\otimes L_\beta$. This induces a filtration on $\underline{CF^*(\mathcal{K},\mathcal{L})}$ with associated graded pieces $\underline{CF^*(\mathcal{K},D^*_\beta\otimes L_\beta)}$, where $L_\beta\in\{L_i\}$ is a simply connected generator. From the earlier argument, these graded pieces are acyclic with trivial Whitehead torsion. Therefore by Proposition \ref{prop:Whitehead-torsion-filtration}, the total complex $\underline{CF^*(\mathcal{K},\mathcal{L})}$ is acyclic with trivial Whitehead torsion.
    
    Thus for any Lagrangian $L$, $\underline{CF^*(\mathcal{K},L)}$ is acyclic with trivial Whitehead torsion, and therefore we conclude that $\mathcal{K}$ is left simply acyclic. Now for any isomorphism element $\gamma\in\hom_{Tw_{Ch}}^0(\mathcal{K}_1,\mathcal{K}_2)$ the mapping cone $cone(\gamma)$ is an acyclic object. Therefore our result shows that the mapping cone of $\gamma$ is simply acyclic, and thus any isomorphism element $\gamma$ is a simple isomorphism element.
\end{proof}
Automatic simplicity (Proposition \ref{prop:automaticsimplicity}) can also be used to show that generators of an $A_\infty$-category $\mathcal{A}$ simply generate if $\mathcal{A}$ can be embedded as a full subcategory of a Fukaya category with simply connected simple generators.

\begin{proposition} \label{prop:subcat-simply-gen}
    Suppose that $\mathcal{F}$ is a Fukaya category that satisfies the conditions of Proposition \ref{prop:automaticsimplicity}, i.e. has simply connected simple generators. Then for any full subcategory $\mathcal{A}$ of $\mathcal{F}$ with generators $\{K_j\}$, the objects $\{K_j\}$ simply generate $\mathcal{A}$.
\end{proposition}
\begin{proof}
    By Proposition \ref{prop:automaticsimplicity}, every isomorphism in $\mathcal{F}$ is automatically a simple isomorphism. Thus every isomorphism in the subcategory $\mathcal{A}$ is also a simple isomorphism, and it follows that the generators $\{K_j\}$ of $\mathcal{A}$ are automatically simple generators.
\end{proof}

We also show that isomorphic twisted complexes have the same Reidemeister torsion.

\begin{proposition} \label{prop:R-torsions-agree}
    Let $\mathcal{K}$, $\mathcal{L}$ be two left and right simply isomorphic twisted complexes in $Tw_{Ch}\mathcal{F}$ with simple isomorphism elements $\alpha\in CF^0(\mathcal{K},\mathcal{L})$, $\beta\in CF^0(\mathcal{L},\mathcal{K})$. Suppose that there exists a ring homomorphism $\rho:\Z\pi_1(X)\to\mathbb{F}$ to a field $\mathbb{F}$ such that 
    \begin{equation}
        \underline{CF^*(\mathcal{K},\mathcal{K})}\otimes_\rho \mathbb{F}
    \end{equation}
    is acyclic. Then $\underline{CF^*(\mathcal{L},\mathcal{L})}\otimes_\rho \mathbb{F}$ is also acyclic, and their Reidemeister torsions agree:
    \begin{equation}
        \Delta_\rho(\underline{CF^*(\mathcal{K},\mathcal{K})}\otimes_\rho \mathbb{F}) = \Delta_\rho(\underline{CF^*(\mathcal{L},\mathcal{L})}\otimes_\rho \mathbb{F}).
    \end{equation}
\end{proposition}
\begin{proof}
    Since $\alpha$ and $\beta$ are simple isomorphisms, we have simple homotopy equivalences 
    \begin{equation}
        \underline{CF^*(\mathcal{K},\mathcal{K})}\simeq \underline{CF^*(\mathcal{K},\mathcal{L})}\simeq \underline{CF^*(\mathcal{L},\mathcal{L})}
    \end{equation}
    induced by the module actions of the simple isomorphism elements $\mu^2_{Tw_{Ch}}(\alpha,~~~)$ and $\mu^2_{Tw_{Ch}}(~~~,\beta)$. From Proposition \ref{prop:R-torsion-simple-invariant}, it follows that the Reidemeister torsions agree.
\end{proof}

To conclude the subsection, we prove that the algebraic twists of simply isomorphic twisted complexes are also simply isomorphic. This generalizes Proposition \ref{prop:algebraic-twist-Ham-isotopy}, and will be used to prove that the Dehn twist exact triangle is simple. Recall that the \emph{algebraic twist} $T_V$ of a twisted complex $\mathcal{L}$ is defined as
\begin{equation}
    T_V\mathcal{L}=CF^*(V,\mathcal{L})\otimes V[1]\oplus\mathcal{L},
\end{equation}
where the morphism is given by the evaluation map.

\begin{proposition} \label{prop: algebraic-twist-simple-twisted-cpx}
    Let $\mathcal{L}$, $\mathcal{L}'$ be two twisted complexes in $Tw_{Ch}\mathcal{F}$ that are simply isomorphic. Then for any object $V$, the objects $T_V\mathcal{L}$, $T_V\mathcal{L}'$ are simply isomorphic.
\end{proposition}
\begin{proof}
    We prove the statement in the case where $\mathcal{L}=L$, $\mathcal{L}'=L'$ are single Lagrangians $L$, $L'$. The argument for general twisted complexes is analogous.

    First, we will show that $T_VL$ and $T_VL'$ are isomorphic by constructing an explicit isomorphism element. The relevant maps are summarized in the following commutative diagram:
    \begin{equation} \label{eqn:comm-diagram}
    \begin{tikzcd}
        CF^*(V,L)[1]\otimes V \arrow[r, "ev'"] \arrow[dr, "\gamma"] \arrow[d, "\phi\otimes \epsilon"'] &L \arrow[d, "\alpha"]\\
        CF^*(V,L')[1]\otimes V \arrow [r, "ev'"'] &L'
    \end{tikzcd}.    
    \end{equation}
    Here $\alpha\in CF^0(L,L')$ is a simple isomorphism element, $\phi=(-1)^{\deg}\mu^2(\alpha,~):CF^*(V,L)\to CF^*(V,L')$, $ev'$ is the shifted evaluation map with signs depending on the degree of the input, and $\epsilon\in CF^0(V,V)$ is a chain representative for the cohomological unit. 
    
    We aim to define a degree 0 morphism $\gamma$ such that
    \begin{equation}\label{eqn-gamma}
        \mu^1_{Tw_{Ch}}(\gamma)+\mu^2_{Tw_{Ch}}(\alpha,ev')+\mu^2_{Tw_{Ch}}(ev',\phi\otimes\epsilon)=0.
    \end{equation}

    Rather than working directly in the morphism space of $Tw_{Ch}\mathcal{F}$, we use the identification (\ref{eqn-evaluation_map}) to identify the above maps with elements in $\hom_\Z(CF^*(V,L)[1],CF^*(V,L'))$.

    First, by Lemma \ref{lem:mu^2-with-ev}, the component $\mu^2_{Tw_{Ch}}(\alpha,ev')$ is identified with the map
    \begin{equation}
        (-1)^{\deg-1}\mu^2(\alpha,~)[1]:CF^*(V,L)[1]\to CF^*(V,L'),
    \end{equation}
    and similarly the component $\mu^2_{Tw_{Ch}}(ev',\phi\otimes\epsilon)$ is identified with the map
    \begin{equation}
        (-1)^{\deg}\mu^2(\alpha,\mu^2(~,\epsilon)):CF^*(V,L)[1]\to CF^*(V,L').
    \end{equation}

    Now we define $\gamma$ in terms of a homotopy. Since $\epsilon$ represents the identity element in cohomology, the map $\mu^2(~,\epsilon):CF^*(V,L')\to CF^*(V,L')$ is chain homotopic to the identity. Thus there exists a degree -1 map $h:CF^*(V,L)\to CF^*(V,L')$ such that
    \begin{equation}\label{eqn-cancel}
        \partial h+h\partial+\mu^2(\mu^2(\alpha,~),\epsilon)=\mu^2(\alpha,~).
    \end{equation}
    Define the map $\gamma$ in the commutative diagram (\ref{eqn:comm-diagram}) by
    \begin{equation}
        (-1)^{\deg} h+(-1)^{\deg}\mu^3(\alpha,~,\epsilon).
    \end{equation}
    Then combining (\ref{eqn-cancel}) and the $A_\infty$-relation
    \begin{equation}
        -\mu^3(\alpha,\mu^1(~),\epsilon)+\mu^1(\mu^3(\alpha,~,\epsilon))+\mu^2(\alpha,\mu^2(~,\epsilon))-\mu^2(\mu^2(\alpha,~),\epsilon)=0,
    \end{equation}
    we verify that (\ref{eqn-gamma}) holds for this choice of $\gamma$.
    
    Therefore, $(\phi\otimes\epsilon,\gamma,\alpha)$ defines a degree zero cocycle in $\hom^*_{Tw_{Ch}}(T_VL,T_VL')$. Since each component map $\phi\otimes\epsilon$ and $\alpha$ is an isomorphism element, this cocycle defines an isomorphism element.

    Now that we have identified a particular isomorphism element between $T_VL$ and $T_VL'$, we claim that it is in fact a simple isomorphism element. By Proposition \ref{prop:simple-notions-check-Lagrangian}, it is enough to check for each Lagrangian $K$ that the homotopy equivalence
    \begin{equation}
        \underline{\hom^*_{Tw_{Ch}}(T_VL,K)}\simeq\underline{\hom^*_{Tw_{Ch}}(T_VL',K)}
    \end{equation}
    induced by the module action of the isomorphism element is simple.

    By Proposition \ref{prop:Whitehead-torsion-filtration}, it suffices to check that the induced maps on each summand of the twisted complex are simple homotopy equivalences: i.e.
    \begin{align}
        \mu^2_{Tw_{Ch}}(\mu^2(\alpha,~)\otimes \epsilon,~~)&:\underline{\hom^*_{Tw_{Ch}}(K,CF^*(V,L)\otimes V)}\to \underline{\hom^*_{Tw_{Ch}}(K,CF^*(V,L')\otimes V)},\\
        \mu^2(\alpha,~)&:\underline{CF^*(K,L)}\to\underline{CF^*(K,L')}
    \end{align}
    are each simple homotopy equivalences.

    Since $\alpha$ is a simple isomorphism element, the second induced map
    \begin{equation}
        \mu^2(\alpha,~):\underline{CF^*(K,L)}\to\underline{CF^*(K,L')}
    \end{equation}
    is a simple homotopy equivalence, meaning its mapping cone is simply acyclic.

    Now we analyze the remaining component, which is the mapping cone of the composition of the following two homotopy equivalences:
    \begin{align}
        \mu^2(\alpha,~)\otimes\id&:CF^*(V,L)\otimes\underline{CF^*(K,V)}\to CF^*(V,L')\otimes\underline{CF^*(K,V)},\\
        \id\otimes\mu^2(~,\epsilon)&:CF^*(V,L')\otimes\underline{CF^*(K,V)}\to CF^*(K,L')\otimes\underline{CF^*(K,V)}.
    \end{align}
    Therefore, it suffices to show that both are simple. The enhanced mapping cone of the first map is
    \begin{equation}
        cone(\mu^2(\alpha,~))\otimes\underline{CF^*(K,V)}~,
    \end{equation}
    which is simply acyclic by Lemma \ref{lem:zero-torsion-general}, since the tensor product of any complex over $\Z\pi_1$ with an acyclic complex is always simply acyclic.

    For the second map, its enhanced mapping cone is
    \begin{equation} \label{eqn-TVL}
        CF^*(V,L')\otimes\underline{cone(\mu^2(~,\epsilon))},
    \end{equation}
    and since $\epsilon$ is a simple isomorphism element, $\underline{cone(\mu^2(~,\epsilon))}$ is simply acyclic. Therefore by Lemma \ref{lem:zero-torsion-tensor-complex}, (\ref{eqn-TVL}) is simply acyclic. It follows that both maps $\mu^2(\alpha,~)\otimes\id$ and $\id\otimes\mu^2(~,\epsilon)$ are simple homotopy equivalences, and hence their composition is as well. Thus, we conclude that the homotopy equivalence
    \begin{equation}
        \underline{\mu^2_{Tw_{Ch}}(K,T_VL)}\simeq\underline{\mu^2_{Tw_{Ch}}(K,T_VL')}
    \end{equation}
    induced by the module action of the isomorphism element is simple for any object $K$. This implies that $T_VL$ and $T_VL'$ are simply isomorphic.
\end{proof}

\subsection{Fukaya categories of Lefschetz fibrations} \label{ssec:dehn-twist-simple}

In this subsection, we prove that the Fukaya category of a Lefschetz fibration is simply generated by Lefschetz thimbles. Our setup throughout this section is an exact Lefschetz fibration $\pi:E\to\mathbb{H}$, equipped with an identification of the fiber over a specific basepoint $*\in\mathbb{H}$ with the exact symplectic manifold $(M,\omega_M=d\theta_M)$. This subsection is largely divided into two parts: in the first part, we prove that the Dehn twist exact triangle can be formulated as a simple isomorphism between the algebraic and geometric twists. The proof of simple generation will come in the second part.

To motivate the proof of the simplicity of the Dehn twist exact triangle, we first provide a quick review of Seidel's original argument in \cite{Seidel03}. Let $V\subset M$ be a framed exact Lagrangian sphere, and let $K,L\subset M$ be exact Lagrangians intersecting transversely with $V$, and with each other. The exact Lefschetz fibration we consider is constructed as $\pi:E\to\mathbb{H}$, with a single critical point at $i\in\mathbb{H}$, and the fiber over the basepoint $*\in\mathbb{H}$ identified with $M$. We also require that the vanishing cycle in the fiber over $*$ to be identified with the framed exact Lagrangian sphere $V$.

In the framework of \cite{Seidel08}, the spherical twist is realized as an $A_\infty$-module assigning to a Lagrangian $K$
\begin{equation}
    \mathcal{T}_V(L)(K)= CF^*(V,L)\otimes CF^*(K,V)[1]\oplus CF^*(K,L),
\end{equation}
with $A_\infty$-module operations we omit. To show that this $A_\infty$-module is quasi-isomorphic to the Yoneda module $\mathcal{Y}^r_{\tau_V(L)}$, we construct an $A_\infty$-module homomorphism
\begin{equation}
    t:\mathcal{T}_VL\to\mathcal{Y}^r_{\tau_VL}.
\end{equation}
By \cite[Lemma 5.3]{Seidel08}, such a morphism exists if one can find a degree zero cocycle $c\in CF^*(L,\tau_VL)$ and a degree -1 map $k:CF^*(V,L)\to CF^*(V,\tau_VL)$ satisfying
\begin{equation}
    \mu^1(c)=0,~\mu^1(k)+k(\mu^1)+\mu^2(c,~)=0.
\end{equation}
Such a pair $(c,k)$ can be constructed by counting pseudoholomorphic sections of a Lefschetz fibration with moving boundary conditions: we will return to this construction later in this section. Having defined the morphism $t$, we aim to show that its mapping cone is acyclic. Concretely, for any other Lagrangian $K$ transverse to $V$ and $L$, we verify that the cochain complex
\begin{equation} \label{eqn:chain-cpx-Dehn-twist}
     (CF^*(V,L)\otimes CF^*(K,V))[2]\oplus CF^*(K,L)[1]\oplus CF^*(K,\tau_VL)
\end{equation}
is acyclic. For later purposes, we will show that the above complex is acyclic for $\tau_VK$ instead of $K$. Under the identifications $CF^*(\tau_VK,V)\cong CF^*(K,V)$ and $CF^*(\tau_VK,\tau_VL)\cong CF^*(K,L)$, we may show that
\begin{equation}
    (CF^*(V,L)\otimes CF^*(K,V))[2]\oplus CF^*(\tau_VK,L)[1]\oplus CF^*(K,L)
\end{equation}
with the differential coming from the mapping cone of $t$
\begin{equation}\label{eqn:Dehn-trig-diff}
\begin{pmatrix}
    \mu^1 & &\\
    \mu^2 &\mu^1 &\\
    \mu^2(k(~),~)+\mu^3(c,~,~) &\mu^2(c,~) &\mu^1
\end{pmatrix}    
\end{equation}
is acyclic. On the level of generators, we have an identification
\begin{equation}
    K\cap\tau_VL=(K\cap L) \sqcup (K\cap V)\times (L\cap V).
\end{equation}
Moreover, we may define injective maps
\begin{align} \label{eqn:Lef-fib-identification-generators}
    &p: (K\cap V)\times(L\cap V) \to \tau_VK\cap L,\\ 
    &q: K\cap L\to \tau_VK\cap L,
\end{align}
such that $q$ is the inclusion map with respect to the above identification. These give an action on the generators, and thus maps $p:CF^*(V,L)\otimes CF^*(K,V)\to CF^*(\tau_VK,L)$ and $q:CF^*(\tau_VK,L)\to CF^*(K,L)$ with a possible sign for each generator. The proof of \cite{Seidel03} uses a careful analysis of energy estimates to find an action filtration such that the low-energy terms of the differential of the cochain complex (\ref{eqn:chain-cpx-Dehn-twist}) takes the form
\begin{equation}\label{eqn:low-energy-differential}
\begin{pmatrix}
    0 & &\\
    p &0 &\\
    0 &q &0 
\end{pmatrix}.    
\end{equation}
The precise statement is organized into the following lemma:
\begin{lemma} \label{lem:Dehn-twist-action-filter}
    Given $V,K,L$ as above, we may pick a model for the Dehn twist $\tau_V$ to be supported in a small enough Weinstein neighborhood such that there exists some $\epsilon>0$ for which the following holds:
\begin{enumerate}
    \item For $y\in\tau_VK\cap L$ and $x\in K\cap L$, either $y=q(x)$ and $\mathcal{A}(y)=\mathcal{A}(x)$, or $y\neq x$ and $\mathcal{A}(x)-\mathcal{A}(y)\not\in [0,3\epsilon)$.
    \item For $y\in\tau_VK\cap L$ and $(x_0,x_1)\in (K\cap V)\times(L\cap V)$, either $y=p(x_0,x_1)$ and $\mathcal{A}(y)-\mathcal{A}(x_0)-\mathcal{A}(x_1)\in[0,\epsilon)$ or $y\neq p(x_0,x_1)$ and $\mathcal{A}(y)-\mathcal{A}(x_0)-\mathcal{A}(x_1)\not\in[0,3\epsilon)$.
\end{enumerate}
\end{lemma}

This constant $\epsilon>0$ depends only on the local model for the Dehn twist, and the Lagrangian submanifolds $V,K,L$. In particular, once we have chosen an $\epsilon$ for some $V,K,L$, we have the freedom to choose a different model for the Dehn twist $\tau_V'$ that is supported in a smaller neighborhood of $V$ compared to $\tau_V$. Since these two different choices are related by a Hamiltonian isotopy, $\tau_VK$ and $\tau_V'K$ are simply isomorphic Lagrangians. So for our purposes, we may choose a Dehn twist that is supported in any small neighborhood of $V$ we want, and likewise we may take $\epsilon$ to be an arbitrarily small positive real number.

Combining \cite[Proposition 3.4]{Seidel03} and \cite[Lemma 3.8]{Seidel03}, we have the following lemma:

\begin{lemma}\label{lem:J-section-low-energy}
    The pseudoholomorphic sections with energy less than $\epsilon$ that contribute to the $\mu^2$ term in the differential are given by $(x_0,x_1)\mapsto p(x_0,x_1)$. The only pseudoholomorphic sections with energy less than $\epsilon$ that contribute to the $\mu^2(c,~)$ term are the constant sections.
\end{lemma}

The proof ends by finding a homotopy between the composition of the two off-diagonal terms $\mu^2(c,\mu^2(~,~))$ and the zero map, and applying \cite[Lemma 2.32]{Seidel03}.

We now explain how the proof of the isomorphism between $\mathcal{T}_VL$ and $\mathcal{Y}^r_{\tau_VL}$ as $A_\infty$-modules can be adapted to show an isomorphism between $T_VL$ and $\tau_VL$ in $Tw_{Ch}\mathcal{F}(\pi)$.

In \cite[Section (17d)]{Seidel08}, a degree zero cocycle
\begin{equation}
    c\in CF^0(V,\tau_VL)
\end{equation}
is defined by counting holomorphic sections to a Lefschetz fibration with a moving boundary condition, as depicted in Figure \ref{fig:pseudohol-section-cocycle-c}. This is the cocycle $c$ that we have constructed in Equation (\ref{eqn:Dehn-trig-diff}).

\begin{figure}[h]
\centering
\begin{tikzpicture}

\draw  (5,0) arc [start angle=270, end angle=450, x radius=1.8, y radius=1.8];
\draw (2.5,0) -- (5,0);
\draw (2.5,3.6) -- (5,3.6);
\draw [dotted, thick] (5,1.8) -- (6.8,1.8);

\draw[->] (5,1) arc[start angle=-90, end angle=90, radius=0.9];

\draw (4.5,4) node {$\tau_VL$};
\draw (4.5,-0.5) node {$V$};
\draw (6.2,2.4) node {$\tau_V$};

\end{tikzpicture}
\caption{The pseudoholomorphic section counted for the definition of the cocycle $c$.}
\label{fig:pseudohol-section-cocycle-c}
\end{figure}
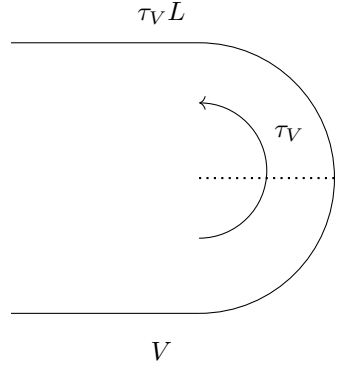

Also, there is a degree -1 map $k:CF^*(V,L)\to CF^*(V,\tau_VL)$ that satisfies for any $x\in CF^*(V,L)$,
\begin{equation}
    \mu^1(k(x))+k(\mu^1(x))+\mu^2(c,x)=0.
\end{equation}
This is constructed by counting a 1-parameter Lefschetz fibration interpolating between the two moving Lagrangian boundary conditions as in Figure \ref{fig:pseudohol-sections-k}. 

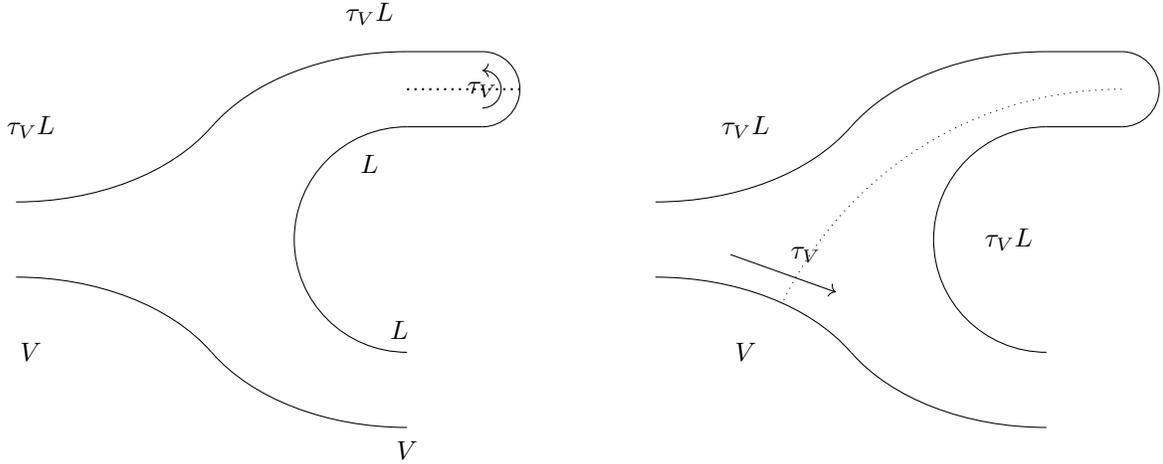
\begin{figure}[h]
\centering
\begin{tikzpicture}

\draw  (3.5,0) arc [start angle=90, end angle=270, x radius=1.5, y radius=1.5];
\draw  (3.5,-4) arc [start angle=-90, end angle=-150, x radius=3, y radius=2];
\draw  (0.9,-3) arc [start angle=30, end angle=90, x radius=3, y radius=2];
\draw  (3.5,1) arc [start angle=90, end angle=150, x radius=3, y radius=2];
\draw  (0.9,0) arc [start angle=-30, end angle=-90, x radius=3, y radius=2];
\draw  (4.5,0) arc [start angle=270, end angle=450, x radius=0.5, y radius=0.5];
\draw (3.5,0) -- (4.5,0);
\draw (3.5,1) -- (4.5,1);
\draw[dotted] (4.5,0.5) arc[start angle=90, end angle=154, radius=5];
\draw[->] (-0.7,-1.7) -- (0.7,-2.2);

\draw (0.3,-1.7) node {$\tau_V$};
\draw (3,-1.5) node {$\tau_VL$};
\draw (-0.5,-3) node {$V$};
\draw (-0.5,0) node {$\tau_VL$};

\draw  (-5,0) arc [start angle=90, end angle=270, x radius=1.5, y radius=1.5];
\draw[->] (-4,0.25) arc[start angle=-90, end angle=90, radius=0.25];
\draw  (-5,-4) arc [start angle=-90, end angle=-150, x radius=3, y radius=2];
\draw  (-7.6,-3) arc [start angle=30, end angle=90, x radius=3, y radius=2];
\draw  (-5,1) arc [start angle=90, end angle=150, x radius=3, y radius=2];
\draw  (-7.6,0) arc [start angle=-30, end angle=-90, x radius=3, y radius=2];
\draw  (-4,0) arc [start angle=270, end angle=450, x radius=0.5, y radius=0.5];
\draw (-5,0) -- (-4,0);
\draw (-5,1) -- (-4,1);
\draw [dotted, thick] (-5,0.5) -- (-3.5,0.5);

\draw (-5.1,-2.7) node {$L$};
\draw (-10,-3) node {$V$};
\draw (-5,-4.3) node {$V$};
\draw (-10,0) node {$\tau_VL$};
\draw (-5.5,1.5) node {$\tau_VL$};
\draw (-5.5,-0.5) node {$L$};
\draw (-4,0.5) node {$\tau_V$};

\end{tikzpicture}
\caption{The 1-parameter family of pseudoholomorphic sections counted to define $k$.}
\label{fig:pseudohol-sections-k}
\end{figure}

Since
\begin{align}
    \hom_{Tw_{Ch}}(CF^*(V,L)\otimes V,\tau_VL)&\cong \hom_\Z(CF^*(V,L),\Z)\otimes CF^*(V,\tau_VL)\\
    &\cong \hom_\Z(CF^*(V,L),CF^*(V,\tau_VL)),
\end{align}
we may regard $k$ as a degree -1 morphism in $ \hom_{Tw_{Ch}}(CF^*(V,L)\otimes V,\tau_VL)$. Thus, we may define a degree zero morphism
\begin{equation}
    (\kappa,c):T_VL\to \tau_VL
\end{equation}
in $Tw_{Ch}\mathcal{F}(\pi)$, where $\kappa=(-1)^{\deg-1}k$ is the map $k$ modified by a sign depending on the degree of the input. To check that this is a cocycle, recall from Section \ref{ssec:category-theory} that for $k\in\hom_\Z(CF^*(V,L),CF^*(V,\tau_VL))$, we have $\mu^1_{Tw_{Ch}}(k)=\mu^1\circ k+k\circ\mu^1$, since $k$ is of odd degree. Therefore the relations we have to check to ensure that $(\kappa,c)$ is a cocycle reduces to
\begin{equation}
    \mu^1(c)=0,~\mu^1(k)+k(\mu^1)+\mu^2(c,~)=0
\end{equation}
which are precisely the relations we required for $c$ and $k$. Now to show that this is an isomorphism, it is enough to check that its mapping cone is an acyclic object in $Tw_{Ch}\mathcal{F}(\pi)$. From Section \ref{ssec:category-theory}, it is enough to check this for single Lagrangians, so take an exact Lagrangian $K$: then the cochain complex
\begin{equation}
    \hom_{Tw_{Ch}}(K,CF^*(V,L)\otimes V[2]\oplus L[1]\oplus \tau_VL)
\end{equation}
is isomorphic to the cochain complex
\begin{equation}
    CF^*(V,L)\otimes CF^*(K,V)[2]\oplus CF^*(K,L)[1]\oplus CF^*(K,\tau_VL),
\end{equation}
and the differentials agree with (\ref{eqn:Dehn-trig-diff}) since $\mu^2_{Tw_{Ch}}(c,ev)=\mu^2(c,~)$. Therefore the same proof as above shows that for a specific choice of $\tau_V$, there exists an isomorphism
\begin{equation}\label{eqn:simple-iso-twch}
    (\kappa,c):T_VL\to \tau_VL
\end{equation}
where the pair $(\kappa,c)$ depends on the choice of $\tau_V$.

From now on, we will prove that the above isomorphism $T_VL\cong\tau_VL$ in $Tw_{Ch}\mathcal{F}(\pi)$ is a simple isomorphism. We first describe the choice of basis. For the intersection points of $\tau_VK\cap L$ that are identified with $K\cap L$, we choose the same lifts. For the intersection points of $\tau_VK\cap L$ that are identified with $(K\cap V)\times(L\cap V)$, we note that all these intersection points lie in a small Weinstein neighborhood $U$ of $V$. Because $U$ is simply connected, we may pick a sheet of the lift $\pi^{-1}(U)$ to the universal cover, and choose the lifts of $K\cap V$, $L\cap V$, and the points in $\tau_VK\cap L$ that lie in $U$ to be in this sheet.

\begin{thm} \label{thm:Dehn-twist-exact-triangle-simple}
    The isomorphism $T_VL\to\tau_VL$ constructed as in (\ref{eqn:simple-iso-twch}) is simple.
\end{thm}
\begin{proof}
The proof will consist of two parts. First, we will show that for any exact Lagrangian $K$, there exists some Dehn twist $\tau_V'$ supported in a small neighborhood of $V$ such that the homotopy equivalence
\begin{equation}
    \mu^2_{Tw_{Ch}}(A',~):\underline{CF^*(K,T_VL)}\to\underline{CF^*(K,\tau_V'L)}
\end{equation}
is simple, where $A'$ is the isomorphism element given by the corresponding element $(k',c')$ for the Dehn twist $\tau_V'$. Assume that this is true, then since compactly supported Hamiltonian isotopies are simple, the induced homotopy equivalence
\begin{equation}
    \underline{CF^*(K,\tau_V'L)}\to\underline{CF^*(K,\tau_VL)}
\end{equation}
is simple. Then by Proposition \ref{prop:composition-simple}, it follows that the homotopy equivalence
\begin{equation}
    \mu^2_{Tw_{Ch}}(A,~):\underline{CF^*(K,T_VL)}\to\underline{CF^*(K,\tau_V'L)}
\end{equation}
is simple, where $A$ is the isomorphism element $(k,c)$ for the Dehn twist $\tau_V$. Therefore, the homotopy equivalence $\mu^2_{Tw}(A,~)$ is a simple homotopy equivalence for any exact Lagrangian $K$, and therefore it is left simple. Right simpleness can be proved similarly, and so we may conclude that the isomorphism $T_VL\to\tau_VL$ is simple.

Thus, we now only need to prove that we can choose some Dehn twist $\tau_V$ for each exact Lagrangian $K$ such that the based cochain complex
\begin{equation} \label{eqn-Dehn-twist-mapping-cone}
 CF^*(V,L)\otimes \underline{CF^*(\tau_VK,V)}[2]\oplus \underline{CF^*(\tau_VK,L)}[1]\oplus \underline{CF^*(K,L)}    
\end{equation}
is simply acyclic. By the multiplicativity of Whitehead torsion and the previous action filtration argument, it is enough to show that the graded pieces, which are mapping cones of the maps
\begin{align}
    &p: CF^*(V,L)\otimes\underline{CF^*(\tau_VK,V)} \to \underline{CF^*(\tau_VK,L)},\\ 
    &q: \underline{CF^*(\tau_VK,L)}\to\underline{CF^*(K,L)},
\end{align}
are simply acyclic.

Now we take a filtration of the cochain complex (\ref{eqn-Dehn-twist-mapping-cone}) such that each graded piece belongs to an action window less than $3\epsilon$. By Lemma \ref{lem:J-section-low-energy}, it follows that for $q$, all the pseudoholomorphic sections contributing to $q$ with energy smaller than $3\epsilon$ have constant image, and therefore contributes trivially to the $\pi_1(X)$-coefficient of the differential of the mapping cone. Therefore the mapping cone of $q$ is acyclic with trivial Whitehead torsion.

For $p$, we need to show that the pseudoholomorphic maps $u$ contributing to the product $\mu^2$ has contribution to the $\pi_1(X)$-term $g(u)=1$. See Figure \ref{fig:Dehn-twist} for an illustration. In view of the above choice of basis, this amounts to showing that the path in the image $\im(u)$ connecting $x_0$ to $p(x_0,x_1)$ as in Figure \ref{fig:Dehn-twist} can be homotoped to a path that lies inside of $U$.

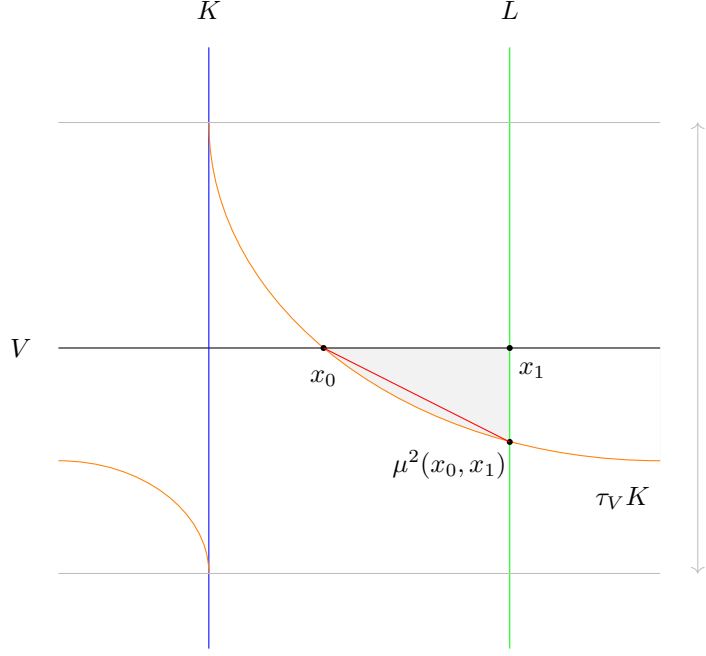
\begin{figure}[h]
\centering
\begin{tikzpicture}

\path [draw=none,fill=lightgray!20] (4,3) -- (-2,3) arc [start angle=180, end angle=270, x radius=6, y radius=4.5] (4,-1.5) -- (4,3);
\path [draw=none,fill=white] (4,3) -- (-2,3) -- (-2,0) -- (4,0) -- (4,3);
\path [draw=none,fill=white] (4,3) -- (2,3) -- (2,-3) -- (4,-3) -- (4,3);

\draw [draw=green] (2,-4) -- (2,4);
\draw (-4,0) -- (4,0);
\draw [draw=blue] (-2,-4) -- (-2,4);

\draw [draw=orange] (-2,3) arc [start angle=180, end angle=270, x radius=6, y radius=4.5];
\draw [draw=orange] (-4,-1.5) arc [start angle=90, end angle=0, x radius=2, y radius=1.5];

\draw [draw=lightgray] (-4,3) -- (4,3);
\draw [draw=lightgray] (-4,-3) -- (4,-3);
\draw[<->, draw=lightgray] (4.5, -3) -- (4.5, 3);

\draw (-4.5,0) node {$V$};
\draw (-2,4.5) node {$K$};
\draw (2,4.5) node {$L$};
\draw (3.5,-2) node {$\tau_VK$};

\draw [draw=none, fill=black] (-0.475,0) circle (0.04 and 0.04);
\draw [draw=none, fill=black] (2,0) circle (0.04 and 0.04);
\draw [draw=none, fill=black] (2,-1.25) circle (0.04 and 0.04);

\draw [draw=red] (-0.475,0) -- (2,-1.25);

\draw (-0.475,-0.4) node {$x_0$};
\draw (2.3,-0.3) node {$x_1$};
\draw (1.2,-1.55) node {$\mu^2(x_0,x_1)$};

\end{tikzpicture}
\caption{The Weinstein neighborhood of the vanishing cycle $V$, together with the Lagrangians $K,L$, and $\tau_V(K)$. The Dehn twist $\tau_V$ is supported in the region between the two gray lines. The shaded area represents the image of the holomorphic disc that contributes to the product $\mu^2:CF^*(V,L)\otimes\underline{CF^*(\tau_VK,V)}\to\underline{CF^*(\tau_VK,L)}$, and the red line depicts the contribution of this curve to the $\pi_1(X)$-coefficient of the complex. 
}
\label{fig:Dehn-twist}
\end{figure}

The idea is to apply the monotonicity lemma to portions of the image of pseudoholomorphic maps lying outside a fixed open neighborhood $U$ of the vanishing cycle $V$. Specifically, we use a version of the monotonicity lemma for pseudoholomorphic maps with switching Lagrangian boundary conditions, as developed in \cite{CEL} and recalled in Lemma \ref{thm:monotonicity-lemma-CEL}. Since we have previously shown that continuation maps induced from homotopies of almost complex structures are simple homotopy equivalences, we are free to pick a convenient regular almost complex structure suited to our analysis. We now describe such an almost complex structure.

For each intersection point $x\in K\cap L$, choose a small open neighborhood $Nbd(x)$, and fix an almost complex structure that is integrable in this neighborhood. Additionally, we also require a holomorphic chart identifying the pair $(T_xK,T_xL)\cong(\R^n,i\R^n)$. We extend this to an almost complex structure $J$ on $M\setminus U$. Now for the regular almost complex structure on $E$ defining the map $p$, we pick $\{J_z\}$ such that for each $z$, $J_z$ restricts to our chosen $J$ on each fiber. Since any pseudoholomorphic map $u$ that contributes to $\mu^2$ must pass through $U$, regularity of $\{J_z\}$ can be achieved by a perturbation supported in the open neighborhood $U$, keeping it fixed outside.

Now let $u$ be a pseudoholomorphic map with Lagrangian boundary conditions $(V,L,\tau_VK)$ contributing to $\mu^2$. If the image of $u$ lies entirely in $U$, then $g(u)=1$, since $U$ is simply connected. Otherwise, suppose that there exists a point $x\in\im(u)\setminus U$. Possibly after shrinking $U$, there exists an open ball $B(x,r)\subset M\setminus U$ with radius $r$ measured with respect to the Riemannian metric $g_J=\omega(~,J~)$. We fix this $r$ to be a bounded constant only depending on the size of $U$. Then for the preimage $\Omega$ of $B(x,r)$ under $u$, $u\vert_\Omega$ is a $J$-holomorphic map with boundary in the immersed Lagrangian $K\cup L$, and we may apply Lemma \ref{thm:monotonicity-lemma-CEL} to obtain a lower bound on the energy of $u$:
\begin{equation}
    E(u)\geq Cr^2
\end{equation}
for some constant $C$ depending on $g_J$, and the geometry of $K$ and $L$ outside of $U$. Since the model of the Dehn twist $\tau_V$ does not affect the geometry of $K$, $L$, and $\tau_VK$ outside of $U$, the constant $C$ is independent of the twist $\tau_V$.

We now choose the support of the Dehn twist small enough such that $3\epsilon<Cr^2$, where $\epsilon$ is the action filtration gap. Then by the above argument, any pseudoholomorphic map $u$ contributing to the map $p$, restricted to this action window, must be entirely contained in $U$, and thus has $g(u)=1$.

It follows that the graded pieces of the mapping cone complex (\ref{eqn-Dehn-twist-mapping-cone}) with respect to the action filtration are identified with the graded pieces of (\ref{eqn:chain-cpx-Dehn-twist}) tensored with $\otimes_\Z \Z\pi_1(X)$, and thus have trivial Whitehead torsion. Therefore by Proposition \ref{prop:Whitehead-torsion-filtration}, the total complex has trivial Whitehead torsion, and thus the equivalence $(k,c):T_VL\to\tau_VL$ is a right simple isomorphism. The left simple case is analogous.
\end{proof}

Having established the simplicity of the Dehn twist exact triangle, we now proceed to prove that the Fukaya category of a Lefschetz fibration is simply generated by its Lefschetz thimbles. For this purpose, we work with the version of the Fukaya category $\mathcal{F}(\pi)$ from Subsection \ref{ssec:Lef-fib}. 

In this Fukaya category, we allow non-compact admissible Lagrangians. However, the results in Subsection \ref{ssec:bimodule} only show that compactly supported Hamiltonian isotopies induce simple isomorphisms. Thus, we must first verify that the simple homotopy type of
\begin{equation}
    \underline{\hom_{\mathcal{F}(\pi)}(K,L)}=\underline{CF^*(\phi_H(K),L)}
\end{equation}
is independent of the choice of Hamiltonian isotopy $\phi_H$ induced from the perturbation datum (\ref{eqn:perturbation-datum}).

Suppose that $K_1=\phi_H(K)$, $K_2=\phi_{H'}(K)$ are two such Hamiltonian perturbations satisfying $\lambda_{K_1}>\lambda_{K_2}>\lambda_L$. We claim that there exists a simple homotopy equivalence
\begin{equation}
    \underline{CF^*(K_1,L)}\simeq\underline{CF^*(K_2,L)}.
\end{equation}
To prove this, observe that there exists a Hamiltonian isotopy $\psi\circ\rho:K_2\to K_1$, where:
\begin{enumerate}
    \item $\rho$ is supported near $\partial_\infty E$ and adjusts the value of $\lambda$ such that $\lambda_{\rho(K_2)}=\lambda_{K_1}$,
    \item $\psi$ is a compactly supported Hamiltonian isotopy from $\rho(K_2)$ to $K_1$.
\end{enumerate}
Since $\lambda_{K_2}>\lambda_L$, we may arrange $\rho$ such that it introduces no new intersections with $L$. Thus, we have an equality of cochain complexes
\begin{equation}
    \underline{CF^*(K_2,L)}=\underline{CF^*(\rho(K_2),L)}.
\end{equation}
Now we apply the compactly supported Hamiltonian isotopy $\psi$ to $\rho(K_2)$: from Proposition \ref{prop:Ham-isotopy-simple}, this induces a simple homotopy equivalence
\begin{equation}
    \underline{CF^*(\rho(K_2),L)}\simeq\underline{CF^*(K_1,L)}.
\end{equation}
Combining the above two equivalences, we conclude the desired simple homotopy equivalence
\begin{equation}
    \underline{CF^*(K_1,L)}\simeq\underline{CF^*(K_2,L)}.
\end{equation}

Unlike the compact Fukaya category, there is no immediate Poincaré duality statement that guarantees that left and right simple categorical notions are equivalent for $\mathcal{F}(\pi)$. So we prove the following proposition, which shows that even in this case, left and right simple notions indeed do agree.

\begin{proposition} \label{prop:left-right-agree}
    Let $\mathcal{K}\in Tw_{Ch}\mathcal{F}(\pi)$ be a left simply acyclic object. Then $\mathcal{K}$ is also right simply acyclic.
\end{proposition}
\begin{proof}
    Let $\mathcal{K}=\bigoplus_\alpha C^*_\alpha\otimes K_\alpha$ be a left simply acyclic object. To verify that $\mathcal{K}$ is right simply acyclic, it suffices to check that for every admissible Lagrangian $L$, the complex
    \begin{equation}
        \underline{\hom_{Tw_{Ch}\mathcal{F}(\pi)}(L,\mathcal{K})}
    \end{equation}
    is simply acyclic. By the definition of the morphisms in $\mathcal{F}(\pi)$, this complex is given by
    \begin{equation}
        \underline{CF^*(\phi_H(L),\mathcal{K})},
    \end{equation}
    where $\phi_H$ is a Hamiltonian isotopy chosen such that $\lambda_{\phi_H(L)}>\lambda_{K_\alpha}$ for any admissible Lagrangian $K_\alpha$ appearing in the twisted complex $\mathcal{K}$. 

    Since $\mathcal{K}$ is left simply acyclic and $\phi_{-H}(\mathcal{K})$ is simply isomorphic to $\mathcal{K}$, we know that
    \begin{equation}
        \underline{CF^*(\phi_{-H}(\mathcal{K}),L)}\cong \underline{CF^*(\mathcal{K},\phi_H(L))}
    \end{equation}
    is simply acyclic. Now by Poincaré duality, it follows that $ \underline{CF^*(\phi_H(L),\mathcal{K})}$ is also simply acyclic, and thus $\mathcal{K}$ is right simply acyclic.
\end{proof}

As a consequence, we conclude that left and right notions of simple acyclicity coincide for $\mathcal{F}(\pi)$. Therefore we may define an object $K\in\mathcal{F}(\pi)$ to be \emph{simply acyclic} if for every object $L\in\mathcal{F}(\pi)$, the complex
\begin{equation}
    \underline{\hom_{\mathcal{F}(\pi)}(K,L)}=\underline{CF^*(\phi_H(K),L)}
\end{equation}
is simply acyclic, for any choice of Hamiltonian perturbation $\phi_H$ used in the definition of the morphisms. Likewise, \emph{simple isomorphisms} are the morphisms whose mapping cones are simply acyclic.

Having defined simple generation for $\mathcal{F}(\pi)$, we now show that the Lefschetz thimbles simply generate $\mathcal{F}(\pi)$.

\begin{proposition} \label{prop:simple-gen-fuk-lef}
    The $A_\infty$-category $\mathcal{F}(\pi)$ is simply generated by the Lefschetz thimbles $B_\gamma$.
\end{proposition}
\begin{proof}

\begin{figure}[h]
\centering
\begin{tikzpicture}

\draw [fill=lightgray!10] plot [smooth cycle] coordinates {(1.3,-3.5) (1.9,-4.1) (3.0,-3.9) (3.1,-3.5)};

\draw [fill=lightgray!50] plot [smooth cycle] coordinates {(1.4,-2.5) (2.5,-1.8) (3.8,-3) (2.7,-3.3)};
\draw (2.7,-3.3) -- (2.7,-5);
\draw (1.9,-4.1) -- (1.9,-5);

\draw (3.4,-3) -- (3.4,-5);

\draw plot [smooth] coordinates {(4.5,-2.51) (4.5,-4.2) (5.0,-4.5) (5,-4.9)};
\draw (5,-4.9) -- (5,-5);

\draw plot [smooth] coordinates {(5.0,-3.01) (5.0,-4.2) (4.5,-4.5) (4.5,-4.9)};
\draw (4.5,-4.9) -- (4.5,-5);

\draw[dotted] (1,-5) -- (6,-5);
\draw (3.4,-3) node {$\times$};
\draw (4.5,-2.5) node {$\times$};
\draw (5.0,-3) node {$\times$};

\draw (1.0,-2.3) node {$L$};
\draw (1.0,-3.6) node {$L'$};
\draw (2.7,-5.3) node {$\lambda_{L}$};
\draw (1.9,-5.3) node {$\lambda_{L'}$};
\draw (5,-5.3) node {$\lambda_{B_1}$};
\draw (3.4,-5.3) node {$\lambda_{B_m}$};
\draw (3.8,-4.0) node {$B_m$};
\draw (5.4,-4.7) node {$B_1$};

\end{tikzpicture}
\caption{}
\label{fig:Lefschetz-fibration-generation}
\end{figure}    
    We adapt the argument from \cite{BaiSeidel} to show that the Lefschetz thimbles simply generate. First, pick a choice of Lefschetz thimbles $B_1,\cdots,B_m$ for each critical point of $\pi$ such that
    \begin{equation}
        \lambda_{B_1}>\cdots>\lambda_{B_m}.
    \end{equation}
    In particular, we require that all intersections between the Lefschetz thimbles occur in the neighborhood of $\partial E$ where the fibration is trivial. Also, pick an admissible Lagrangian $L$ such that $\lambda_{B_m}>\lambda_L$. Our claim is that for any other admissible Lagrangian $L'$ such that $\lambda_L>\lambda_{L'}$,
    \begin{equation}
        \underline{\hom_\mathcal{C}(T_{B_1}\cdots T_{B_m}(L),L')}
    \end{equation}
    is simply acyclic. Here, the morphism is taken in the directed subcategory $\mathcal{C}=\mathcal{F}^\rightarrow_{Tw}(B_1,\cdots,B_m,L,L')$. 

\begin{figure}[h]
\centering
\begin{tikzpicture}

\draw plot [smooth, tension=1] coordinates {(1.9,-4.1) (1.9,-5.0) (1.5,-5.3) (0,-4.5) (0,-2.5) (-2.5,0) (-4.5,0) (-5.3,-1.5) (-5,-1.9) (-4.1,-1.9)}; 
\draw [fill=white, draw=white] plot coordinates{(1.3,-4.1) (1.3,-5) (2.5,-5) (2.5,-4.1)};
\draw [fill=white, draw=white] plot coordinates{(-4.1,-1.3) (-5,-1.3) (-5,-2.5) (-4.1,-2.5)};

\draw plot [smooth, tension=1] coordinates {(5.0,-4.0) (5.0,-5.0) (5.5,-5.3) (5.5,0) (0,-0.5) (-2,0) (-6,3) (-6,5.5) (-5,6) (-4.5,5) (-4.5,3.5)};
\draw [fill=white, draw=white] plot coordinates{(4.5,-1.3) (5.0,-1.3) (5.0,-5) (4.5,-5)};
\draw [fill=white, draw=white] plot coordinates{(-4.1,3.3) (-5.0,3.3) (-5.0,5) (-4.1,5)};

\draw plot [smooth, tension=1] coordinates {(4.5,-4.0) (4.5,-5.0) (4.8,-5.8) (5.5,-5.8) (6,0) (0,0.5) (-4.5,2) (-5.5,3.5) (-5.8,5) (-5.4,5.5) (-5,5) (-5,3.5)};
\draw [fill=white, draw=white] plot coordinates{(4.2,-1.3) (5.1,-1.3) (5.1,-5) (4.2,-5)};
\draw [fill=white, draw=white] plot coordinates{(-4.1,3.3) (-5.1,3.3) (-5.1,5) (-4.1,5)};

\draw plot [smooth, tension=1] coordinates {(2.7,-5) (5,-6.6) (6.6,-6.6) (6.6,0) (5,1) (2,1) (0,3) (1,4) (5,4) (5.5,3.35) (5,2.7)};

\draw [fill=lightgray!10] plot [smooth cycle] coordinates {(1.3,-3.5) (1.9,-4.1) (3.0,-3.9) (3.1,-3.5)};
\draw [fill=lightgray!10] plot [smooth cycle] coordinates {(-1.3,3.5) (-1.9,4.1) (-3.0,3.9) (-3.1,3.5)};
\draw [fill=lightgray!10] plot [smooth cycle] coordinates {(3.5,1.3) (4.1,1.9) (3.9,3.0) (3.5,3.1)};
\draw [fill=lightgray!10] plot [smooth cycle] coordinates {(-3.5,-1.3) (-4.1,-1.9) (-3.9,-3.0) (-3.5,-3.1)};

\draw [fill=lightgray!50] plot [smooth cycle] coordinates {(1.4,-2.5) (2.5,-1.8) (3.8,-3) (2.7,-3.3)};
\draw [fill=lightgray!50] plot [smooth cycle] coordinates {(-1.4,2.5) (-2.5,1.8) (-3.8,3) (-2.7,3.3)};
\draw [fill=lightgray!50] plot [smooth cycle] coordinates {(2.5,1.4) (1.8,2.5) (3,3.8) (3.3,2.7)};
\draw [fill=lightgray!50] plot [smooth cycle] coordinates {(-2.5,-1.4) (-1.8,-2.5) (-3,-3.8) (-3.3,-2.7)};

\draw (2.7,-3.3) -- (2.7,-5);
\draw (-2.7,3.3) -- (-2.7,5);
\draw (3.3,2.7) -- (5,2.7);
\draw (-3.3,-2.7) -- (-5,-2.7);

\draw (1.9,-4.1) -- (1.9,-5);
\draw (-1.9,4.1) -- (-1.9,5);
\draw (4.1,1.9) -- (5,1.9);
\draw (-4.1,-1.9) -- (-5,-1.9);

\draw (3.4,-3) -- (3.4,-5);

\draw plot [smooth] coordinates {(4.5,-2.51) (4.5,-4.2) (5.0,-4.5) (5,-4.9)};
\draw (5,-4.9) -- (5,-5);
\draw (-5,3) -- (-5,5);

\draw plot [smooth] coordinates {(5.0,-3.01) (5.0,-4.2) (4.5,-4.5) (4.5,-4.9)};
\draw (4.5,-4.9) -- (4.5,-5);
\draw (-4.5,2.5) -- (-4.5,5);

\draw[dotted] (1,-5) -- (5.5,-5);
\draw[dotted] (5,1) -- (5,5.5);
\draw[dotted] (-1,5) -- (-5.5,5);
\draw[dotted] (-5,-1) -- (-5,-5.5);

\draw (3.4,-3) node {$\times$};
\draw (4.5,-2.5) node {$\times$};
\draw (5.0,-3) node {$\times$};
\draw (-4.5,2.5) node {$\times$};
\draw (-5.0,3) node {$\times$};

\draw[dotted, gray] (-5,0) -- (5,0);
\draw[dotted, gray] (0,-5) -- (0,5);

\draw (3.0,-5.7) node {$N$};
\draw (0,-5) node {$N'$};
\draw (0.5,0.7) node {$S_2$};
\draw (5.7,-1.6) node {$S_1$};

\end{tikzpicture}
\caption{The $4:1$ branched cover of the Lefschetz fibration. Note that the Lagrangian spheres $S_1$ and $S_2$ only intersect in one of the four quadrants.}
\label{fig:quadruple-cover}
\end{figure}
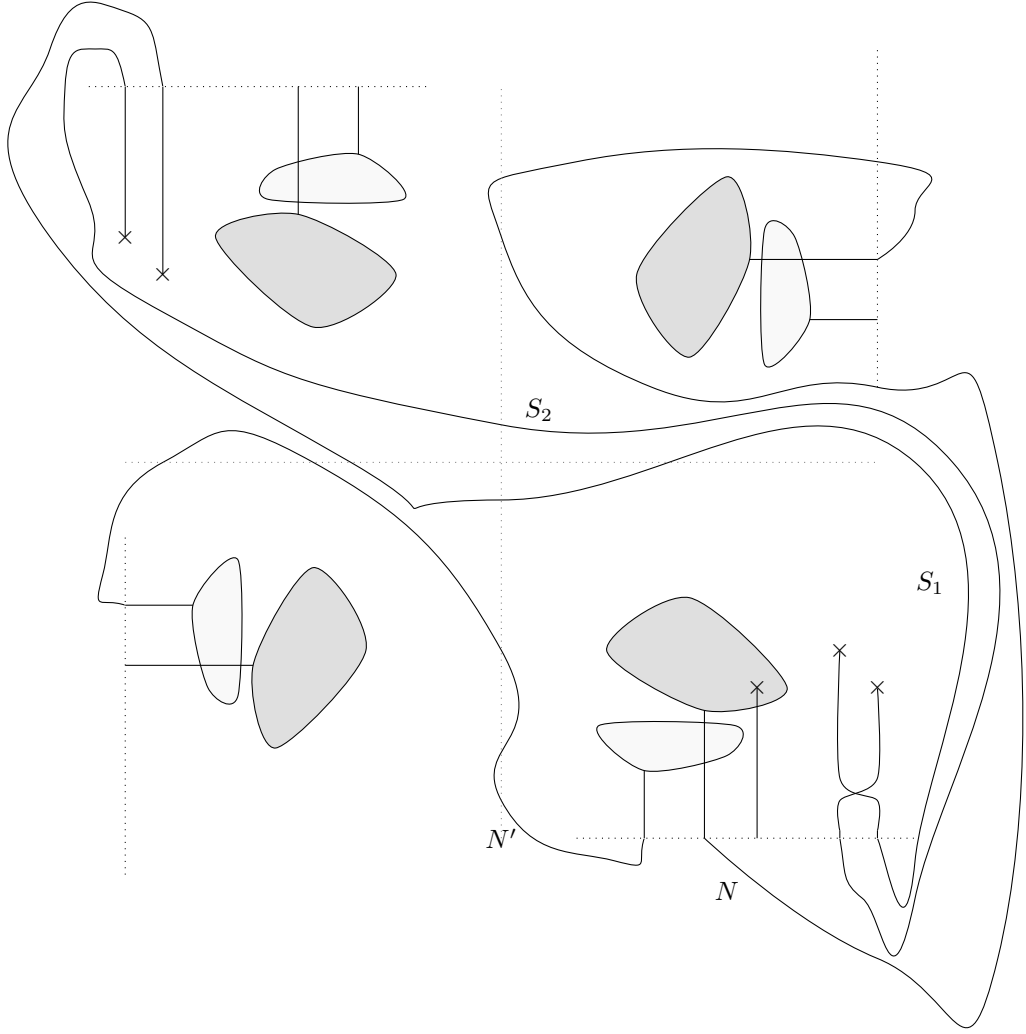

We construct a 4:1 branched cover $E^{(4)}\to E$, branched over the fiber above the basepoint $*\in\mathbb{H}$. For the Lefschetz thimbles, we take double covers such that each thimble $B_i$ lifts to a Lagrangian sphere $S_i$ in $E^{(4)}$. For the Lagrangian $L$, we choose a lift $N$ supported in the first and second branches of the cover, while for $L'$, we choose a lift $N'$ supported in the fourth and first branches. For the thimbles $B_i$, we take lifts supported in the first and third branch, arranged so that two lifts $S_i$ and $S_j$ intersect only in the first branch to ensure that $S_i\cap S_j$ is in $1:1$ correspondence to $B_i\cap B_j$. This can be ensured because we required $B_i$ and $B_j$ to only intersect near the boundary, where the fibration is trivial. These choice of lifts are illustrated in Figure \ref{fig:quadruple-cover}.

Now observe that the compositions of the Dehn twists along the Lagrangian spheres $S_i$ maps $N$ to $\Bar{N}$: a double cover of $L$ which is now supported in the second and third branch. Since $\Bar{N}$ and $N'$ are supported in different branches, it follows that $\Bar{N}\cap N'=\varnothing$, and therefore 
\begin{equation}
    \underline{CF^*(\Bar{N},N')}=0,
\end{equation}
which is simply acyclic. This implies that
\begin{equation}
   \underline{CF^*(\tau_{S_1}\cdots \tau_{S_m}N,N')}
\end{equation}
is simply acyclic as well. Our goal is now to show that the homotopy equivalence
\begin{equation}
   \underline{CF^*(\tau_{S_1}\cdots \tau_{S_m}N,N')}\simeq \underline{CF^*(T_{S_1}\cdots T_{S_m}N,N')}
\end{equation}
is simple. This follows from Proposition \ref{prop: algebraic-twist-simple-twisted-cpx}. First, the simplicity of the Dehn twist exact triangle shows that
\begin{equation}
    \tau_{S_m}N\simeq T_{S_m}N
\end{equation}
is a simple isomorphism. Assume by induction that for some $i$,
\begin{equation}
    \tau_{S_i}\cdots\tau_{S_m}N\simeq T_{S_i}\cdots T_{S_m}N
\end{equation}
is a simple isomorphism. Then because of the simplicity of the Dehn twist exact triangle,
\begin{equation}
    \tau_{S_{i-1}}(\tau_{S_i}\cdots\tau_{S_m}N)\simeq T_{S_{i-1}}(\tau_{S_i}\cdots \tau_{S_m}N)
\end{equation}
is a simple isomorphism. Since algebraic twists preserve simple isomorphisms (Proposition \ref{prop: algebraic-twist-simple-twisted-cpx}),
\begin{equation}
    T_{S_{i-1}}(\tau_{S_i}\cdots\tau_{S_m}N)\simeq T_{S_{i-1}}(T_{S_i}\cdots T_{S_m}N)
\end{equation}
is a simple isomorphism. Using that simple isomorphisms are preserved under composition (Proposition \ref{prop:composition-simple}), it follows inductively that $\underline{CF^*(T_{S_1}\cdots T_{S_m}N,N')}$ is simply acyclic.

As can be seen from Figure \ref{fig:quadruple-cover}, all pseudoholomorphic curves contributing to $CF^*(T_{S_1}\cdots T_{S_m}N,N')$ are entirely supported within the first branch. This implies that
\begin{equation}
    \underline{CF^*(T_{B_1}\cdots T_{B_m}L,L')}\simeq \underline{CF^*(T_{S_1}\cdots T_{S_m}N,N')},
\end{equation}
from which it follows that $T_{B_1}\cdots T_{B_m}L$ is simply acyclic when paired against any admissible Lagrangian $L'$ such that $\lambda_{L'}<\lambda_{L}$. By a similar argument to Proposition \ref{prop:left-right-agree}, this result extends to admissible Lagrangians $L'$ for which $\lambda_{L'}>\lambda_L$ as well. Therefore, any object $L$ is simply isomorphic to a twisted complex in $Tw_{Ch}\mathcal{F}(\pi)$ built out of the thimbles $B_i$, completing the proof of simple generation.
\end{proof}

Now that we have established that the Lefschetz thimbles (which are simply connected) simply generate the Fukaya category of the Lefschetz fibration $\mathcal{F}(\pi)$, we can invoke automatic simplicity to show that any two isomorphic closed Lagrangians in $\mathcal{F}(\pi)$ are simply isomorphic. Combining this with the result from \cite{GirouxPardon}, which shows the existence of a Lefschetz fibration on any Weinstein manifold, we arrive at our main theorem.

\begin{thm} \label{thm:Weinstein-equivalence-simple-reprise}
    Let $X$ be a Weinstein manifold with $c_1(X)=0$, and let $K$ and $L$ be two closed exact Maslov zero Lagrangians in $X$, equipped with brane structures so that they define objects in the compact Fukaya category $\mathcal{F}(X)$. If $K$ and $L$ define isomorphic objects in $\mathcal{F}(X)$, they are automatically simply isomorphic.
\end{thm}
\begin{proof}

    We first provide a brief summary of Giroux--Pardon's argument. By \cite[Theorem 1.10]{GirouxPardon}, there exists a Weinstein manifold $X'$ that is Weinstein homotopic to $X$, together with a Lefschetz fibration $\pi:X'\to\mathbb{H}$. The idea is to first choose a Weinstein homotopy to a Stein manifold $X''$ as in \cite{CieliebakEliashberg}, then choose a Stein homotopy from $X''$ to $X'$, another Stein structure on $X$ that admits a function $\pi:X'\to\C$ which equips $X'$ with the structure of a Lefschetz fibration. These deformations can be integrated to obtain an exact symplectomorphism $\phi:X\to X'$, and we will argue that $\phi$ can be chosen to be the identity map outside of a compact subset of $X$.

    Recall from \cite[Proposition 11.8]{CieliebakEliashberg} that for any Liouville homotopy $(V,\omega_s=d\lambda_s,X_s)_{s\in[0,1]}$ such that the union of the skeleton $\mathrm{Skel}(V,\lambda_s,X_s)$ is compact, one can find a compactly supported exact symplectomorphism $\phi_s:V\to V$ such that $\phi_s^*\lambda_s-\lambda_0$ is compactly supported and exact. Thus, the Weinstein deformation $X$ to $X''$ can be chosen to happen in a compact subset. For the second part, the proof of \cite[Theorem 1.5]{GirouxPardon} uses quantitative transversality results to find a holomorphic function with controlled behavior at infinity to construct the Stein Lefschetz fibration. In our case, where the Weinstein structure is standard at infinity, the Stein deformation can also be chosen to be compactly supported.

    Therefore, we may choose an exact symplectomorphism $\phi:X\to X'$ that is identity outside of a compact set, and a Lefschetz fibration $\pi:X'\to\C$. Since all critical values of $\pi$ lie in a compact subset of $\C$, we will consider this as a Lefschetz fibration $\pi:X'\to \mathbb{H}$ after postcomposing $\pi$ with a Möbius transformation. 

    The Lagrangians $K$ and $L$ are mapped by $\phi$ to closed exact Maslov zero Lagrangians $\phi(K),\phi(L)$ in $X'$, which inherit brane structures and hence define objects of the compact Fukaya category $\mathcal{F}(X')$. Since $\phi$ preserves holomorphic curve counts, $\phi(K)$ and $\phi(L)$ are isomorphic in both $\mathcal{F}(X')$ and $\mathcal{F}(\pi)$. By the simple generation of $\mathcal{F}(\pi)$ by Lefschetz thimbles and automatic simplicity, it follows that $\phi(K)$ and $\phi(L)$ are simply isomorphic. Finally since $\phi$ is an exact symplectomorphism, we conclude that $K$ and $L$ are simply isomorphic objects in $\mathcal{F}(X)$.
\end{proof}

\section{Applications} \label{sec:applications}

In the previous section, we proved that any two isomorphic objects in the compact Fukaya category of a Weinstein manifold $X$ with $c_1(X)=0$ are automatically simply isomorphic. In this section, we use this statement to prove the applications Theorem \ref{thm:auteq}, Theorem \ref{thm:lens-space}, and Theorem \ref{thm:simple-he-Lagrangian} as stated in the introduction.

\subsection{Cotangent bundles} \label{ssec:cotangent-bundle}

To demonstrate the strength of Theorem \ref{thm:Weinstein-equivalence-simple-reprise}, we provide a quick proof that the cotangent bundles of lens spaces are symplectomorphic if and only if the lens spaces themselves are diffeomorphic.

\begin{thm} \label{thm:cotangentbundle-reprise}
    Suppose that we have two lens spaces $L(p,q)$, $L(p,q')$. Then their cotangent bundles $T^*L(p,q)$ and $T^*L(p,q')$ are symplectomorphic if and only if $L(p,q)$ and $L(p,q')$ are diffeomorphic.
\end{thm}
\begin{proof}
Suppose that there exists a symplectomorphism $\phi$ from $T^*L(p,q)$ to $T^*L(p,q')$. Then the image of the zero section $K=\phi(L(p,q))$ is a closed exact Lagrangian submanifold of $T^*L(p,q')$, homeomorphic to $L(p,q)$. Since $H^1(L(p,q);\Z)=0$, any such Lagrangian has vanishing Maslov class. Thus by \cite[Lemma C.1]{Abouzaid12b}, we may equip $K$ with a brane structure such that $K$ is isomorphic to the zero section $N=L(p,q')$ in the Fukaya category $\mathcal{F}(T^*L(p,q'))$ of closed exact Lagrangians in $T^*L(p,q')$. 

Therefore by Theorem \ref{thm:Weinstein-equivalence-simple-reprise}, the isomorphism between $K$ and $N$ must also be simple. In particular, the homotopy equivalences induced by the module actions of the isomorphism elements
\begin{equation}\label{eqn:ssec-5.1}
    \underline{CF^*(K,K)}\simeq\underline{CF^*(K,N)}\simeq\underline{CF^*(N,N)}
\end{equation}
are simple homotopy equivalences. So by Proposition \ref{prop:R-torsion-simple-invariant}, the Reidemeister torsions $\Delta_{\rho'}(\underline{CF^*(K,K)})$ and $\Delta_{\rho}(\underline{CF^*(N,N)})$ must agree for any ring homomorphism $\rho:\Z[\pi_1(T^*L(p,q'))]\to\C$, when defined. We emphasize that $\rho':\pi_1(K)\to\C$ is induced from $\rho$ via the identification $\pi_1(L(p,q))\cong\pi_1(L(p,q'))$, which may not preserve the chosen generators in the identifications with $\Z/p$. Now by the simple homotopy equivalence of the Morse complex and the cellular complex (Proposition \ref{prop:Morse=CW}), we have simple homotopy equivalences
\begin{align}
    &\underline{CF^*(K,K)}\simeq \underline{C^*_{cell}(L(p,q))},\\
    &\underline{CF^*(N,N)}\simeq \underline{C^*_{cell}(L(p,q'))}.
\end{align}
Since the inclusions of $K$ and $N$ into $T^*L(p,q')$ both induce isomorphisms of fundamental groups, the enhanced cellular cochain complexes compute the Reidemeister torsions of $K$ and $N$ with respect to the two ring homomorphisms $\rho'$ and $\rho$. Combined with the simple homotopy equivalence (\ref{eqn:ssec-5.1}) above, we conclude that the $\rho'$-Reidemeister torsion of $K$, which is homeomorphic to $L(p,q)$, and $\rho$-Reidemeister torsion of $N$, which is homeomorphic to $L(p,q')$ must agree, when defined. 

As in the computation of Reidemeister torsion for lens spaces, choose the ring homomorphism $\rho:\Z\pi_1(T^*L(p,q'))\cong\Z[t]/(t^p-1)\to\C$ sending $t$ to $\xi=e^{2\pi i/p}$. To identify $\rho':\Z\pi_1(L(p,q))\to\C$, let $m$ be an integer such that the inclusion $K\xhookrightarrow{}T^*L(p,q')$ sends the generator $t$ of $\pi_1(L(p,q))\cong\Z/p$ to $t^m\in\pi_1(T^*L(p,q'))\cong\Z/p$. Then $\rho':\Z[t]/(t^p-1)\to\C$ maps $t$ to $\xi^m$. Thus, the Reidemeister torsion of $K$ may be computed as
\begin{equation}
    \Delta_{\rho'}(\underline{CF^*(K,K)})=(\xi^{mr}-1)(\xi^m-1)\in \C^\times/\{\pm\xi^k\}_{k\in\Z}
\end{equation}
for some integer $r$ such that $qr\equiv 1~(\mathrm{mod}~p)$. Likewise,
\begin{equation}
    \Delta_{\rho}(\underline{CF^*(N,N)})=(\xi^{r'}-1)(\xi-1)\in\C^\times/\{\pm\xi^k\}_{k\in\Z}
\end{equation}
for $qr'\equiv 1~(\mathrm{mod}~p)$. But now since
\begin{equation}
    (\xi^{mr}-1)(\xi^m-1)=\pm\xi^k(\xi^{r'}-1)(\xi-1)
\end{equation}
only holds for $r'=\pm r^{\pm 1}$ as seen in the classification of lens spaces, it follows that $L(p,q)$ and $L(p,q')$ are diffeomorphic.
\end{proof}

Theorem \ref{thm:cotangentbundle-reprise} holds for higher dimensional lens spaces as well, by the same argument.

\subsection{Weinstein 1-handle connect sums of cotangent bundles} \label{ssec:app-auteq}

In this subsection, we prove Theorem \ref{thm:auteq}. The theorem below is stated in a different form than Theorem \ref{thm:auteq}; we will explain why this version implies the original statement.

We recall that for two Liouville manifolds $X$, $Y$, the symbol $X\natural Y$ stands for the Liouville manifold constructed from the disjoint union $X\amalg Y$ by attaching a Weinstein 1-handle to connect the two disjoint components.

\begin{thm} \label{thm:auteq-reprise}
    Let $X$ be the Weinstein 1-handle connect sum $T^*L(7,1)\natural T^*L(7,2)$. For any exact symplectomorphism $\phi:X\to X$, the induced autoequivalence $\phi_*:\mathcal{W}(X)\to\mathcal{W}(X)$ maps any compact Lagrangian in the first summand with respect to the decomposition $\mathcal{W}(X)\cong\mathcal{W}(T^*L(7,1))\oplus\mathcal{W}(T^*L(7,2))$ to the $\mathcal{W}(T^*L(7,1))$ summand again.
\end{thm}

    We now explain how Theorem \ref{thm:auteq} follows from the above statement. The zero section $L(7,1)$ generates the summand $H_3(T^*L(7,1))\subset H_3(X)$. The above theorem implies that for any exact symplectomorphism $\phi$, the image $\phi(L(7,1))$ has vanishing Floer cohomology with any cotangent fiber of $T^*L(7,2)$. Since the Euler characteristic of Lagrangian Floer cohomology agrees with the algebraic intersection number, it follows that $[\phi(L(7,1))]$ also lies in the $H_3(T^*L(7,1))$ summand of $H_3(X)$. Applying the same argument with a cotangent fiber of $T^*L(7,1)$ shows that the multiplicity must be $\pm1$, completing the proof.

    The key idea of the proof is in Proposition \ref{prop:classification-Lagrangians}, which classifies connected closed exact Maslov zero Spin Lagrangians in $X$. Assuming this proposition, let us take any such closed exact Lagrangian $L$ in $\mathcal{F}(T^*L(7,1))\subset\mathcal{W}(X)$. Then by Proposition \ref{prop:classification-Lagrangians}, its image $\phi(L)$ under the symplectomorphism must be isomorphic to either of the zero sections $L(7,1)$ or $L(7,2)$ in $\mathcal{W}(X)$, up to a possible shift in grading. For the sake of contradiction, suppose that $\phi(L)$ is isomorphic to $Q_2[k]$ for some integer $k$ recording the degree shift.
    
    By Proposition \ref{thm:Weinstein-equivalence-simple-reprise}, this isomorphism will be simple, and thus from Proposition \ref{prop:R-torsions-agree}, for any $\rho:\Z\pi_1(X)\to\C$, the Reidemeister torsions of $\underline{CF^*(\phi(L),\phi(L))}\otimes_{\rho_1}\C$ and $\underline{CF^*(Q_2,Q_2)}\otimes_{\rho_2}\C$ must agree when defined. Here $\rho_1$ and $\rho_2$ are the induced representations on the fundamental groups $\pi_1(\phi(L))$ and $\pi_1(Q_2)$.
    
    However, the computation in Proposition \ref{prop:R-torsion-7*7} shows that we can choose a map $\rho$ for which the two Reidemeister torsions above do not agree, yielding the desired contradiction.

    Before proving our key proposition, we introduce a lemma that classifies all compact objects in the wrapped Fukaya category $\mathcal{W}(T^*S^n)$ whose endomorphism algebras are cohomologically supported in nonnegative degrees and is free in the zeroth degree component, for all $n\geq2$.

\begin{lemma} \label{lem:compact-objects-T^*S^n}
    Let $L$ be an object in the wrapped Fukaya category $\mathcal{W}(T^*S^n)$ with $\Z$-coefficients for $n \geq 2$ such that $HW^0(L,L) \cong \Z^r$ for some $r\geq1$, and suppose that $HW^*(L,L)$ is supported in finitely many nonnegative degrees. Such an object $L$ can exist only if $r=s^2$ for some integer $s$, and in that case, $L$ is isomorphic to $s$ copies of the zero section $(S^n)^{\oplus s}$ in $Tw\mathcal{W}(T^*S^n)$, up to a possible uniform shift in grading.
\end{lemma}

\begin{proof}
    We first show that $HW^*(N,L) \cong \Z^s$ for $N = T^*_x S^n$ a cotangent fiber. Let $k$ be a field, either $\Q$ or $\mathbb{F}_p$ for some prime $p$. By applying the homological perturbation lemma, we may assume that the $A_\infty$-algebra
    \begin{equation}
        \mathcal{A} = CW^*(N,N;k)
    \end{equation}
    and the $A_\infty$-module
    \begin{equation}
        \mathcal{M} = CW^*(N,L;k)
    \end{equation}
    are minimal. Then, by Lemma \ref{lem:classifcation-minimal-modules}, there exists an $A_\infty$-module
    \begin{equation}
        \mathcal{L} = \bigoplus_i k[i] \otimes M^i
    \end{equation}
    isomorphic to the Yoneda module of $L$, where each $M^i$ is uniquely determined by the $A_\infty$-module action of $\mathcal{A}^0\cong k$. Therefore we may identify each $M^i$ with the Yoneda module of the zero section $S^n$. Thus by Lemma \ref{lem:twisted-cpx-endo}, $\mathcal{L}$ is quasi-isomorphic to the direct sum of $t$ copies of the Yoneda module of $S^n$ (up to a uniform degree shift), for some $t$. Then it follows that $HW^0(L,L;k)\cong k^{t^2}$, so we conclude that $r$ must be a square $r=s^2$, and $t=s$. In particular, we obtain
    \begin{equation}
        HW^*(N,L;k) \cong (k[i])^{\oplus s}
    \end{equation}
    for some integer $i$, and for all $k = \Q, \mathbb{F}_p$.

    Since $CW^*(N,L)$ is finitely generated over $\Z$, $HW^*(N,L)$ is a finitely generated abelian group whose base change to every field $k$ is $s$-dimensional. Thus we conclude that
    \begin{equation}
        HW^*(N,L)\cong(\Z[i])^{\oplus s}
    \end{equation}
    for some shift $i$.

    Now we return to $\Z$-coefficients. By the computation of wrapped Floer cohomology of a cotangent fiber (Proposition \ref{prop:AS}), we have that
    \begin{equation}
        HW^*(N,N)\cong H_{-*}(\Omega_x S^n).
    \end{equation}
    By a Serre spectral sequence argument and an explicit identification of the cup product structure (see \cite[Proposition 3.22]{Hatcher} for details), the latter can be computed as
     \begin{equation} \label{eqn:loop space cohomology}
         H_{-*}(\Omega_x S^n)\cong \begin{cases}
             \Gamma_\Z[x], & \text{if $n$ odd}\\
            \Lambda_Z[x]\otimes \Gamma_\Z[y], & \text{if $n$ even}
         \end{cases}
     \end{equation}
    where $x$ has degree $n-1$, and $y$ has degree $2(n-1)$. Here $\Gamma_\Z[x]$ denotes the divided polynomial algebra, and $\Lambda[x]$ denotes the exterior algebra. Therefore the cohomology of the $A_\infty$-algebra
    \begin{equation}
        \mathcal{A}'=CW^*(N,N)
    \end{equation}
    is a free $\Z$-module in each degree, and so we may apply the homological perturbation lemma (Lemma \ref{lem:homological-perturbation}) and assume that $\mathcal{A}'$ is a minimal $A_\infty$-algebra. Since $HW^*(N,L)$ was shown to be a free $\Z$-module in each degree, we may again apply the homological perturbation lemma and assume that $HW^*(N,L)$ is a minimal $A_\infty$-module over $\mathcal{A}'$. By the same argument as before, it follows that $L$ is isomorphic to the direct sum of $s$ copies of the zero section $S^n$, up to a uniform shift in grading.
\end{proof}

We now begin the proof of our key proposition.

\begin{proposition} \label{prop:classification-Lagrangians}
    Let $K$ be a connected closed exact Maslov zero Spin Lagrangian submanifold in $X=T^*L(7,1)\natural T^*L(7,2)$. Then $K$ is isomorphic to one of the two zero sections $Q_1=L(7,1)$ or $Q_2=L(7,2)$ in the compact Fukaya category $\mathcal{F}(X)$, up to a possible shift in grading.
\end{proposition}

\begin{proof}
    By Proposition \ref{prop:subcrit-handle}, the wrapped Fukaya category of $X$ splits into a direct sum
    \begin{equation}
        \mathcal{W}(X)\cong \mathcal{W}(T^*L(7,1))\oplus\mathcal{W}(T^*L(7,2))
    \end{equation}
    up to quasi-equivalence. Thus, $K$ is isomorphic in $\mathcal{W}(X)$ to either an object in $\mathcal{W}(T^*L(7,1))$ or an object in $\mathcal{W}(T^*L(7,2))$, whose endomorphism algebra is cohomologically supported in nonnegative degrees.

    Suppose that $K$ is isomorphic to an object $P$ in $\mathcal{W}(T^*L(7,1))$: then for the cotangent fiber $N_2$ of $T^*L(7,2)$, we have that $HW^*(K,N_2)$ is zero. In this case, we will show that $K$ is actually isomorphic to the zero section $Q_1=L(7,1)$ in $\mathcal{W}(X)$, again up to a possible grading shift. Consider the 7:1 cover $\pi:\Tilde{X_1}\to X$ defined by taking the 7:1 cover $T^*S^3\to T^*L(7,1)$ and attaching 7 copies of $T^*L(7,2)$:
    \begin{equation}
        \tilde{X}_1=T^*S^3\natural T^*L(7,2)\natural\cdots\natural T^*L(7,2).
    \end{equation}
    Let $N_1$ be a cotangent fiber of $T^*L(7,1)$ at a basepoint $x\in L(7,1)$, and let $\tilde{N}_1$ be the cotangent fiber of $T^*S^3\subset\tilde{X_1}$ at some lift $\tilde{x}\in S^3$ of $x$. 
    
    We may consider the preimage $\pi^{-1}(K)$ as an object of $\mathcal{W}(\tilde{X_1})$. By Proposition \ref{prop:subcrit-handle} again, the wrapped Fukaya category of $\tilde{X}_1$ splits into a direct sum
    \begin{equation}
        \mathcal{W}(\tilde{X}_1)\cong \mathcal{W}(T^*S^3)\oplus\mathcal{W}(T^*L(7,2))^7
    \end{equation}
    up to quasi-equivalence, and since the object $\pi^{-1}(K)$ is orthogonal to any of the cotangent fibers generating the $\mathcal{W}(T^*L(7,2))$ components, we conclude that $\pi^{-1}(K)$ is isomorphic to an object $\tilde{P}$ in $\mathcal{W}(T^*S^3)$.

    We claim that $\tilde{P}$ is isomorphic to the zero section $S^3$ in $\mathcal{W}(T^*S^3)$, up to a shift. Since $K$ is connected and $\pi:\tilde{X}\to X$ is a 7:1 cover, the preimage $\pi^{-1}(K)$ is a closed Lagrangian with either one or seven connected components. It follows that $HW^*(\pi^{-1}(K),\pi^{-1}(K))$ is supported in finitely many nonnegative degrees, and $HW^0(\pi^{-1}(K),\pi^{-1}(K))$ is isomorphic to either $\Z$ or $\Z^7$. Since $\tilde{P}$ is isomorphic to $\pi^{-1}(K)$ in $\mathcal{W}(\tilde{X})$, the same holds for $\tilde{P}$. By Lemma \ref{lem:compact-objects-T^*S^n}, $HW^0(\tilde{P},\tilde{P})$ cannot have rank 7 as a free $\Z$-module, so it follows that $\pi^{-1}(K)$ must be connected. Therefore by Lemma \ref{lem:compact-objects-T^*S^n} again, we conclude that $\tilde{P}$ is isomorphic to the zero section $S^3$, up to a shift.
     
    In particular, we conclude that there is an isomorphism
    \begin{equation}
        HW^*(\tilde{N}_1,\pi^{-1}(K))\cong \Z[j],
    \end{equation}
    where the $j$ stands for a possible degree shift.
    
    Now we look at $CW^*(N_1,K)$. Since any holomorphic curve that contributes to the differential of $CW^*(N_1,K)$ lifts uniquely to a holomorphic curve that contributes to the differential of $CW^*(\tilde{N}_1,\pi^{-1}(K))$ and vice versa a holomorphic curve that contributes to the differential of $CW^*(\tilde{N},\pi^{-1}(K))$ projects down, we have that
    \begin{equation}
        HW^*(N_1,K)\cong HW^*(\tilde{N}_1,\pi^{-1}(K))\cong \Z.
    \end{equation}
    So in particular the cohomology of the $A_\infty$-module $CW^*(N_1,K)$ is a free $\Z$-module. As for the cohomology of the fiber $CW^*(N_1,N_1)$, a loop space cohomology computation as in (\ref{eqn:loop space cohomology}) shows that
    \begin{equation} \label{eqn:loop space cohomology lens}
         H_{-*}(\Omega_x L(7,1))\cong \Gamma_\Z[x]\otimes_\Z \Z[\Z/7]
     \end{equation}
    where $x$ has degree $2$, and $\Gamma_\Z[x]$ denotes the divided polynomial algebra.

    Therefore the cohomology of both $CW^*(N_1,N_1)$ and $CW^*(N_1,K)$ consists of free $\Z$-modules in each degree, so we may apply the homological perturbation lemma and assume that these are minimal $A_\infty$-algebras and modules.

    By  Lemma \ref{lem:classifcation-minimal-modules}, the Yoneda module of $K$ is isomorphic to a twisted complex built from $\Z$-modules equipped with an action of $\mathcal{A}^0$. Moreover, since the cohomology of the $A_\infty$-module $CW^*(N_1,K)$ is supported in a single degree, Lemma \ref{lem:degree-filtration} shows that this is determined by the cohomological module $HW^*(N_1,K)$ over $\mathcal{A}^0=CW^0(N,N)\cong\Z\pi_1(L(7,1))$. Therefore, the Yoneda module of $K$ is isomorphic to a Yoneda module of a rank 1 local system over the zero section $L(7,1)$. But since $\pi_1(L(7,1))\cong\Z/7$, all rank 1 $\Z$-local systems on $L(7,1)$ are trivial. Therefore we conclude that $K$ is isomorphic to $Q_1$, again possibly up to a shift in grading.
\end{proof}

We are now ready to prove Theorem \ref{thm:auteq-reprise}.
\begin{proof}
    For the sake of contradiction, suppose that there exists an exact symplectomorphism $\phi:X\to X$, and an object $L\in\mathcal{F}(T^*L(7,1))$ such that $\phi(L)\in\mathcal{W}(T^*L(7,2))$, with respect to the decomposition
    \begin{equation}
        \mathcal{W}(X)\cong\mathcal{W}(T^*L(7,1))\oplus\mathcal{W}(T^*L(7,2)).
    \end{equation}
    By Proposition \ref{prop:classification-Lagrangians}, $L$ is isomorphic to $Q_1[j]$ and $\phi(L)$ is isomorphic to $Q_2[k]$ in $\mathcal{F}(X)$, for some integers $j$, $k$ which record the degree shift.

   Since $X$ is Weinstein, Theorem \ref{thm:Weinstein-equivalence-simple-reprise} implies that $\phi(L)$ and $Q_2[k]$ are simply isomorphic. Thus $\phi(Q_1)[j]$ and $Q_2[k]$ are also simply isomorphic. In particular, by Proposition \ref{prop:R-torsions-agree}, the Reidemeister torsions 
    \begin{equation}\label{eqn:R-torsions-agree}
        \Delta_{\rho_1}(\underline{CF^*(Q_1,Q_1)})=\Delta_{\rho_2}(\underline{CF^*(Q_2,Q_2)})\in\C^\times/\pm\im{\rho}
    \end{equation}
    agree for all ring homomorphisms $\rho:\Z\pi_1(X)\to\C$, when defined. Again, $\rho_1$ and $\rho_2$ are the induced representations on $\pi_1(\phi(Q_1))$ and $\pi_1(Q_2)$ from $\rho$.

    Now we recall the computation in Proposition \ref{prop:R-torsion-7*7}. There we defined a ring homomorphism $\rho$ such that the above equation (\ref{eqn:R-torsions-agree}) does not hold, which is a contradiction. Thus we conclude that no such exact symplectomorphism $\phi$ can exist.
\end{proof}

The same argument also shows the following general theorem:

\begin{thm}
    Any connected closed exact Maslov zero Spin Lagrangian submanifold in a Weinstein 1-handle connect sum of two cotangent bundles of lens spaces must have the simple homotopy type of one of the zero sections. For example, there does not exist a Lagrangian diffeomorphic to $L(17,4)$ inside the Weinstein 1-handle connect sum $T^*L(17,1)\natural T^*L(17,2)$.
\end{thm}

\begin{remark}
    The above theorem implies that the autoequivalence of $\mathcal{W}(X)$ for $X=T^*L(7,1)\natural T^*L(7,2)$ swapping the two cotangent fibers cannot be realized by any underlying exact symplectomorphism of $X$.
\end{remark}

We conclude this subsection by clarifying the nature of the obstruction in Theorem \ref{thm:auteq-reprise}. 

First, recall that an \emph{end structure at infinity} on an open manifold $M$ consists of an open subset $U\subset M$ with compact complement, together with a diffeomorphism $U\cong \partial U\times[0,\infty)$. Cotangent bundles carry canonical end structures at infinity given by the complement of their disc bundles. Since lens spaces are parallelizable, their cosphere bundles are trivial. Thus we may identify the cosphere bundles of $T^*L(7,1)$ and $T^*L(7,2)$ with $L(7,1)\times S^2$ and $L(7,2)\times S^2$, respectively.

Although Milnor \cite{MilnorLens} shows that $T^*L(7,1)$ and $T^*L(7,2)$ are diffeomorphic as open manifolds, this diffeomorphism cannot preserve the canonical ends coming from the disk bundles. Indeed, the corresponding cosphere bundles are $L(7,1)\times S^2$ and $L(7,2)\times S^2$, which are distinguished by Reidemeister torsion. The obstruction for a diffeomorphism between $T^*L(7,1)$ and $T^*L(7,2)$ preserving the prescribed end structures is therefore topological, and is independent of the symplectic geometry of the cotangent bundles.

Milnor's construction reflects this point explicitly: it uses an $h$-cobordism $W$ between $L(7,1)\times S^2$ and $L(7,2)\times S^2$ with nontrivial Whitehead torsion, and builds the open diffeomorphism by stacking $W$ and $-W$ infinitely many times. Thus the diffeomorphism necessarily changes the end structure at infinity.

However, the underlying diffeomorphism of an exact symplectomorphism need not preserve any prescribed end structure. Therefore the obstruction to an exact symplectomorphism between $T^*L(7,1)$ and $T^*L(7,2)$ observed by Abouzaid--Kragh \cite{AbouzaidKragh} is genuinely symplectic, rather than a consequence of smooth or PL topology.

We show that the same phenomenon occurs for Theorem \ref{thm:auteq-reprise}: although exact symplectomorphisms cannot interchange the two summands of $H_3(X)$, there exists a self-diffeomorphism of the open manifold $X=T^*L(7,1)\natural T^*L(7,2)$ that swaps the two generators of $H_3(X)$.

\begin{lemma}\label{lem:swap-diffeo}
    There exists a diffeomorphism $\phi:X\to X$ such that the induced map
    \begin{equation}
        \phi_*:H_3(X)\to H_3(X)
    \end{equation}
    swaps the two summands of the decomposition $H_3(X)\cong H_3(L(7,1))\oplus H_3(L(7,2))$.
\end{lemma}
\begin{proof}
    As a caution, we first point out an incorrect approach. One might attempt to construct the map $\phi$ by a diffeomorphism $\psi:L(7,1)\times S^2\setminus D^5\to L(7,2)\times S^2\setminus D^5$. Since every orientation-preserving diffeomorphism of $S^4$ extends across $D^5$, the existence of such a $\psi$ would imply that $L(7,1)\times S^2\cong L(7,2)\times S^2$. This is impossible, since these manifolds are distinguished by Reidemeister torsion. Thus the desired diffeomorphism cannot be obtained by identifying the standard ends.

    Our construction is laid out in two steps. First, choose a proper ray at infinity $\gamma:[0,\infty)\to T^*L(7,1)$ in the standard end, together with a framed tubular neighborhood of $\gamma$. Milnor's diffeomorphism
    \begin{equation}
        \Phi:T^*L(7,1)\cong T^*L(7,2)
    \end{equation}
    defines the corresponding proper ray at infinity $\Phi(\gamma)$ with a framed neighborhood in $T^*L(7,2)$.

    Define an open manifold $X'$ by attaching a standard open $1$-handle  $D^1\times D^5$ at infinity between the framed neighborhoods of $\gamma\subset T^*L(7,1)$ and $\Phi(\gamma)\subset T^*L(7,2)$. Then $\Phi$ extends to a self-diffeomorphism $\phi':X'\to X'$ defined as follows: on the cotangent bundles, define
    \begin{equation}
        \phi'\vert_{T^*L(7,1)}=\Phi,\quad \phi'\vert_{T^*L(7,2)}=\Phi^{-1}.
    \end{equation}
    These maps exchange the two attaching rays. On the 1-handle, use the reflection of $D^1$ together with the antipodal map of $D^5$ to define an orientation-preserving diffeomorphism. With the induced identification from the framings, these data define an orientation-preserving diffeomorphism $\phi'$ of $X'$ which swaps the two summands of $H_3(X')\cong H_3(T^*L(7,1))\oplus H_3(T^*L(7,2))$.
    
    It remains to show that $X'$ and $X$ are diffeomorphic. By \cite[Theorem 1.1]{CalcutGompf}, in dimension at least $4$, the diffeomorphism type of a 1-handle attachment at infinity is independent of the choice of attaching rays, provided the relevant ends satisfy the Mittag--Leffler condition. In our case this condition holds because the ends are topologically collared (see the discussion after \cite[Definition 4.2]{CalcutGompf}, and also \cite[Proposition 4.3]{CalcutGompf}). Therefore the open manifold $X'$, constructed from a 1-handle attachment at infinity using the rays $\gamma\subset T^*L(7,1)$ and $\Phi(\gamma)\subset T^*L(7,2)$, is diffeomorphic to the original connect sum $X$, obtained by using standard rays in the two ends.

    Transporting $\phi'$ across such a diffeomorphism $X'\cong X$ gives the desired self-diffeomorphism $\phi:X\to X$ which exchanges the two summands of $H_3(X)$.
\end{proof}

Together with Theorem \ref{thm:auteq-reprise}, this shows that Whitehead torsion provides an obstruction to realizing elements of the smooth mapping class group as exact symplectomorphisms.

\subsection{Simple homotopy equivalence of closed Lagrangians} \label{ssec:app-simply-homotopic}

We prove our main application Theorem \ref{thm:simple-he-Lagrangian} in this subsection. For the reader's convenience, we repeat the statement.

\begin{thm} \label{thm:simple-he-Lagrangians-reprise}
    Let $X$ be a Weinstein manifold with $c_1(X)=0$, and let $L$ be a connected closed exact Lagrangian brane such that the inclusion $L\xhookrightarrow{}X$ is a homotopy equivalence. Then, for any Lagrangian brane $K$ isomorphic to $L$ in the compact Fukaya category $\mathcal{F}(X)$ such that $\pi_1(K)\to\pi_1(X)$ is an isomorphism, the inclusion $K\xhookrightarrow{}X$ is also a homotopy equivalence, and the composition with any homotopy inverse of $L\xhookrightarrow{}X$
    \begin{equation}
        K\xhookrightarrow{}X\to L
    \end{equation}
    is a simple homotopy equivalence.
\end{thm}
\begin{proof}
    Before laying out the full details, we begin with a brief outline of the argument. For this proof only, we temporarily switch to homological grading for $\underline{CF_*(K,L)}$ and cellular chain complexes to match the conventions used in \cite{AbouzaidKragh}.
    
    The first step is to note that the Whitehead torsion of the inclusion of a closed exact Lagrangian can be computed using Morse chain complexes. By adapting the argument of \cite[Lemma 5.1]{AbouzaidKragh} to Weinstein manifolds, we show that for any closed exact Lagrangian brane $K$, the chain map
    \begin{equation}
        \iota_*:\underline{C^{cell}_*(K)}\to\underline{C^{cell}_*(X)}
    \end{equation}
    induced by the inclusion $K\xhookrightarrow{}X$ is homotopic to the inclusion 
    \begin{equation}
        j_*:\underline{CM_*(K)}\xhookrightarrow{}\underline{CM_*(X)}
    \end{equation}
    of enhanced Morse complexes. In particular, $\iota_*$ is a homotopy equivalence if and only if $j_*$ is. These Morse complexes are constructed with respect to a suitably chosen Morse function in $X$, which we construct in Lemma \ref{lem:Morse-function}. Then using the simple homotopy equivalence $\underline{CF_*(K,K)}\simeq\underline{CM_*(K)}$, we define a map
    \begin{equation}
        \psi_K:\underline{CF_*(K,K)}\to \underline{CM_*(X)},
    \end{equation}
    and similarly a map $\psi_L:\underline{CF_*(L,L)}\to\underline{CM_*(X)}$. Since the inclusion $L\xhookrightarrow{}X$ is a homotopy equivalence, $\psi_L$ is also a homotopy equivalence.

    In the second step, we show that there is a chain homotopy equivalence
    \begin{equation}
        \psi_L^{-1}\circ\psi_K\simeq\mu^2(~,\beta)\circ\mu^2(\alpha,~):\underline{CF_*(K,K)}\to \underline{CF_*(L,L)}
    \end{equation}
    where $\alpha\in CF^0(K,L)$, $\beta\in CF^0(L,K)$ are isomorphism elements. Since $K$ and $L$ are connected, such isomorphism elements are unique up to sign (see the discussion after Definition \ref{def:equivalence}). Thus, Theorem \ref{thm:Weinstein-equivalence-simple-reprise} implies that these particular elements $\alpha$ and $\beta$ define simple isomorphisms. It follows that the maps $\mu^2(\alpha,~)$ and $\mu^2(~,\beta)$ are simple homotopy equivalences. Since $\psi_L$ is a homotopy equivalence, we conclude that $\psi_K$ is a homotopy equivalence as well. Therefore, we conclude that the inclusion $i:K\xhookrightarrow{}X$ induces an isomorphism on homology groups.

    Since we assume that $i:K\to X$ induces an isomorphism on fundamental groups, there exists a lift $\tilde{i}:\tilde{K}\to\tilde{X}$ to the universal covers. The spaces $\tilde{K}$ and $\tilde{X}$ are simply connected, and the lift $\tilde{i}$ induces an isomorphism between their homology groups. Therefore by the relative Hurewicz theorem, $\tilde{i}$ induces an isomorphism on all homotopy groups
    \begin{equation}
        \tilde{i}_*:\pi_n(\tilde{K})\cong\pi_n(\tilde{X})
    \end{equation}
    for $n\geq 2$. Since $i:K\to X$ also induces an isomorphism on the fundamental groups, it follows that $i_*:\pi_n(K)\to\pi_n(X)$ is an isomorphism for all $n\geq1$. By Whitehead's theorem \cite[Theorem 1]{Whitehead}, we conclude that the inclusion $K\xhookrightarrow{}X$ is a homotopy equivalence.
    
    Moreover, the Whitehead torsion of the induced map $\varphi$ of $K\xhookrightarrow{}X\to L$ on enhanced cellular complexes equals the Whitehead torsion of $\psi_L^{-1}\circ\psi_K$, which is chain homotopic to a simple homotopy equivalence. Since chain homotopic homotopy equivalences have the same Whitehead torsion, $\varphi$ is a simple homotopy equivalence.

    Finally, we appeal to Proposition \ref{prop:CW_cpx_he_simple}, which states that a continuous cellular map between CW complexes is a simple homotopy equivalence if and only if its induced map on the enhanced cellular complexes has trivial Whitehead torsion. Therefore we conclude that the map $K\xhookrightarrow{}X\to L$ is a simple homotopy equivalence.

    With the above outline in place, we now provide the details of the proof. Our first task is to construct a Morse function $H$ on $X$ which restricts to a Morse function $h$ on $L$, such that the chain map
    \begin{equation}
        \underline{C_*^{cell}(L)}\xhookrightarrow{}\underline{C^{cell}_*(X)}
    \end{equation}
    induced by the inclusion $L\xhookrightarrow{} X$ is homotopic to the chain map
    \begin{equation}
        \underline{CM_*(L;h)}\to\underline{CM^*(X;H)},
    \end{equation}
    given by an inclusion of chain complexes.

    The existence of such a Morse function $H$ is guaranteed by the following lemma:
    \begin{lemma} \label{lem:Morse-function}
        Let $N$ be a closed exact Lagrangian in a Liouville manifold $X$. Pick a sufficiently large Liouville subdomain $X_0$ of $X$ that contains $N$, and suppose that $X\setminus X_0$ is a cylindrical neighborhood of $X$ at infinity. Then there exists a Morse-Smale function $H:X\to\R$ and a metric $g$ on $X$ such that the following conditions hold:
    \begin{enumerate}
        \item $h=H\vert_N$ is a Morse function, and the gradient flow of $H$ starting at a point in $N$ stays in $N$.
        \item The gradient flow of $H$ in a point near $N$ repels away from $N$.
        \item The gradient flow of $H$ points outward to $\partial X_0$, and $H$ restricted to $\partial X_0$ is also a Morse function.
    \end{enumerate}
    \end{lemma}
    \begin{proof}
    Since $X$ is Liouville, there exists a Morse-Smale function $f$ such that the gradient flow of $f$ flows outward along $\partial X_0$. We may also assume that $f$ is Morse when restricted to the boundary of $X_0$. Now pick some Riemannian metric $g$ on $X$ for which $N$ is totally geodesic and $g$ is cylindrical at infinity. Now define
    \begin{equation}
        F(x)=h(\pi(x))+d_N(x)^2,
    \end{equation}
    where $h$ is a $C^2$-small, positive Morse function on $N$, $d_N$ is the distance to $N$ measured in the metric $g$, and $\pi(x)$ is the point in $N$ with least distance to $x$ with respect to $g$. By smoothly interpolating between $F$ and $f$, we obtain a Morse function $H$ that satisfies the conditions, possibly after a $C^1$-small perturbation.
    \end{proof}

    Returning to our setting, apply the above Lemma to construct a Morse function $H:X\to\R$ whose restriction $h=H\vert_L:L\to\R$ is Morse. Because the descending manifolds of $H$ starting in $L$ stay in $L$, we obtain an inclusion of chain complexes
    \begin{equation} \label{eqn:5.3:L-to-X}
        j_*:\underline{CM_*(L;h)}\xhookrightarrow{}\underline{CM_*(X_0;H)}\simeq\underline{CM_*(X;H)},
    \end{equation}
    where the metric $g$ is chosen as in Lemma \ref{lem:Morse-function}.

    Now the argument of \cite[Lemma 5.1]{AbouzaidKragh} applies: if we define $X_s$ to be the union of the descending manifolds of $(X,H,g)$, the inclusion $X_s\xhookrightarrow{}X_0$ is a composition of elementary collapses which collapses the cells of $\partial X_0$. Thus it is a simple homotopy equivalence, and therefore the inclusion $\underline{CM_*(X_s)}\xhookrightarrow{}\underline{CM_*(X_0)}$ is a simple homotopy equivalence. Moreover, since the inclusion $L\xhookrightarrow{}X_s$ is cellular, we can conclude that the map $j_*$ is chain homotopic to $\iota_*$, the induced map by the inclusion $L\xhookrightarrow{}X$. Thus, if (\ref{eqn:5.3:L-to-X}) is a (simple) homotopy equivalence, then so is $L\xhookrightarrow{}X$.

    Now if we define $\underline{CF^*(L,L)}$ from the Floer data chosen in the definition of the Fukaya category $\mathcal{F}(X)$, we have a simple homotopy equivalence $\underline{CF_*(L,L)}\simeq\underline{CM_*(L;h)}$. Composing with the chain map (\ref{eqn:5.3:L-to-X}), we define a chain map
    \begin{equation}
        \psi_L:\underline{CF_*(L,L)}\to \underline{CM_*(X;H)}.
    \end{equation}
    Similarly we may apply Lemma \ref{lem:Morse-function} to the Lagrangian $K$, and obtain another Morse function $H'$ on $X$. Since we have shown that both $\underline{CM_*(X,H,g)}$ and $\underline{CM_*(X,H',g')}$ are simple homotopy equivalent to $\underline{C_*^{cell}(X)}$, we may compose with this simple homotopy equivalence between the two Morse complexes and define a map
    \begin{equation}\label{eqn:5.3-Q-to-X}
        \psi_K:\underline{CF_*(K,K)}\to\underline{CM_*(X,H)}.
    \end{equation}

    We now compare $\psi_L$ and $\psi_K$. The result we need is proven in \cite[Proposition 5.2]{AbouzaidKragh}: we recall the statement in the following proposition.

    \begin{proposition} \label{prop:L-M-torsion-agree}
        The two maps below are chain homotopic:
    \begin{align}
        \psi_K&:\underline{CF_*(K,K)} \to \underline{CM_*(X;H)},\\
        \psi_L\circ\mu^2(~,\beta)\circ\mu^2(\alpha,~)&:\underline{CF_*(K,K)}\to\underline{CF_*(K,L)}\to\underline{CF_*(L,L)}\to\underline{CM_*(X;H)},
    \end{align}
    where $\alpha\in CF^0(K,L)$ and $\beta\in CF^0(L,K)$ are isomorphism elements.
    \end{proposition}
    The proof is given by counting a 1-parameter moduli space of pearly trajectories, as depicted in Figure \ref{fig:disc-config}. The construction of a Morse function as in Lemma \ref{lem:Morse-function} is necessary for this argument.
    
    By composing with a homotopy inverse of $\psi_L$ on both sides, we obtain the desired chain homotopy equivalence
    \begin{equation}
        \psi_L^{-1}\circ\psi_K\simeq\mu^2(~,\beta)\circ\mu^2(\alpha,~).
    \end{equation}
    Since $K$ and $L$ are connected, the isomorphism elements $\alpha$ and $\beta$ are unique up to sign (see the discussion after Definition \ref{def:equivalence}), and Theorem \ref{thm:Weinstein-equivalence-simple-reprise} implies that these particular choices of $\alpha$ and $\beta$ are simple isomorphism elements. Thus, the right-hand side is a simple homotopy equivalence, and thus $\psi_L^{-1}\circ\psi_K$ is as well. Therefore from the argument mentioned in the beginning of the proof, we conclude that $K\xhookrightarrow{}X\to L$ is a simple homotopy equivalence.
\end{proof}

\begin{figure}[p]
\centering
\begin{tikzpicture}

\draw (2,-2) -- (2,-1.7);
\draw (2,-1.55) node {$\alpha$};
\draw (0,-4) -- (0,-4.3);
\draw (0,-4.55) node {$\beta$};
\draw (0,-3) circle (1 and 1);
\draw (2,-3) circle (1 and 1);
\draw [postaction={decorate, decoration={markings, mark=at position 0.5 with {\arrow{<}}}}] (3,-3) -- (5,-3);
\draw [postaction={decorate, decoration={markings, mark=at position 0.5 with {\arrow{>}}}}] (-1,-3) arc [start angle=270, end angle=210, x radius=2, y radius=1.5];
\draw [postaction={decorate, decoration={markings, mark=at position 0.5 with {\arrow{>}}}}](-2.74,-2.25) arc [start angle=330, end angle=270, x radius=2, y radius=1.5];
\draw (-4.45,-3) arc [start angle=0, end angle=360, x radius=1, y radius=1];
\draw [postaction={decorate, decoration={markings, mark=at position 0.5 with {\arrow{>}}}}] (-6.45,-3) -- (-7.95,-3);
\draw [thick, dotted] (-6.45,-3) -- (-5.45,-3);
\draw (1,-4) node {$K$};
\draw (1,-2) node {$L$};
\draw (-1,-4) node {$L$};

\draw (4,-2.55) node {$K$};
\draw (4,-3.45) node {$\nabla h'$};
\draw (-2.75,-3.05) node {$L$};
\draw (-3.6,-2.45) node {$\nabla h$};
\draw (-1.9,-2.45) node {$\nabla h$};
\draw (-7.2,-3.45) node {$\nabla H$};
\draw (-7.15,-2.6) node {$X$};
\draw (-5.5,-1.75) node {$L$};

\draw (2,-5) -- (2,-4.7);
\draw (2,-4.55) node {$\alpha$};
\draw (0,-7) -- (0,-7.3);
\draw (0,-7.55) node {$\beta$};
\draw (0,-6) circle (1 and 1);
\draw (2,-6) circle (1 and 1);
\draw [postaction={decorate, decoration={markings, mark=at position 0.5 with {\arrow{<}}}}] (3,-6) -- (5,-6);
\draw [postaction={decorate, decoration={markings, mark=at position 0.5 with {\arrow{>}}}}] (-1,-6) -- (-4.45,-6);
\draw (-4.45,-6) arc [start angle=0, end angle=360, x radius=1, y radius=1];
\draw [postaction={decorate, decoration={markings, mark=at position 0.5 with {\arrow{>}}}}] (-6.45,-6) -- (-7.95,-6);
\draw [thick, dotted] (-6.45,-6) -- (-5.45,-6);

\draw (0.5,-8) -- (0.5,-7.7);
\draw (0.5,-7.55) node {$\alpha$};
\draw (-1.5,-10) -- (-1.5,-10.3);
\draw (-1.5,-10.55) node {$\beta$};
\draw (-1.5,-9) circle (1 and 1);
\draw (0.5,-9) circle (1 and 1);
\draw [postaction={decorate, decoration={markings, mark=at position 0.5 with {\arrow{<}}}}] (1.5,-9) -- (3.5,-9);
\draw (-2.5,-9) arc [start angle=0, end angle=360, x radius=1, y radius=1];
\draw [postaction={decorate, decoration={markings, mark=at position 0.5 with {\arrow{>}}}}] (-4.5,-9) -- (-6,-9);
\draw [thick, dotted] (-4.5,-9) -- (-3.5,-9);
\draw (2.5,-9.45) node {$\nabla h$};
\draw (-5.2,-9.45) node {$\nabla H$};
\draw (2.45,-8.6) node {$K$};
\draw (-5.2,-8.6) node {$X$};
\draw (-0.5,-10) node {$K$};
\draw (-2.5,-10) node {$L$};
\draw (-1.5,-7.75) node {$L$};

\draw (-1.5,-11.5) -- (-1.5,-11.2);
\draw (-1.5,-11.05) node {$\alpha$};
\draw (-1.5,-14.5) -- (-1.5,-14.8);
\draw (-1.5,-15.05) node {$\beta$};
\draw (-1.5,-13) circle (1.5 and 1.5);
\draw [postaction={decorate, decoration={markings, mark=at position 0.5 with {\arrow{<}}}}] (0,-13) -- (3.5,-13);
\draw [postaction={decorate, decoration={markings, mark=at position 0.5 with {\arrow{>}}}}] (-3,-13) -- (-6,-13);
\draw [thick, dotted] (-3,-13) -- (-1.5,-13);

\draw (-1.5,-17) circle (1.5 and 1.5);
\draw [postaction={decorate, decoration={markings, mark=at position 0.5 with {\arrow{<}}}}] (0,-17) -- (3.5,-17);
\draw [postaction={decorate, decoration={markings, mark=at position 0.5 with {\arrow{>}}}}] (-3,-17) -- (-6,-17);
\draw [thick, dotted] (-3,-17) -- (-1.5,-17);
\draw (0,-19) circle (1 and 1);
\draw (-4.45,-17.5) node {$\nabla H$};
\draw (1.7,-17.5) node {$\nabla h'$};
\draw (-4.45,-16.5) node {$X$};
\draw (1.7,-16.5) node {$K$};
\draw (0,-20.4) node {$L$};
\draw (0.84,-19.5) -- (1.05,-19.7);
\draw (-0.84,-19.5) -- (-1.05,-19.7);
\draw (1.2,-19.85) node {$\alpha$};
\draw (-1.24,-19.85) node {$\beta$};
\draw (1.24,-19) node {$K$};
\draw (-1.24,-19) node {$K$};
\draw (-3,-16) node {$K$};
\draw (-0,-16) node {$K$};

\end{tikzpicture}
\caption{The 1-parameter family of holomorphic maps that are counted in the proof of Proposition \ref{prop:L-M-torsion-agree}. The 1-parameter family of pseudoholomorphic curves counted in the second picture breaks into the first and third pictures, and the 1-parameter family in the fourth picture breaks into the third and fifth pictures.}
\label{fig:disc-config}
\end{figure}
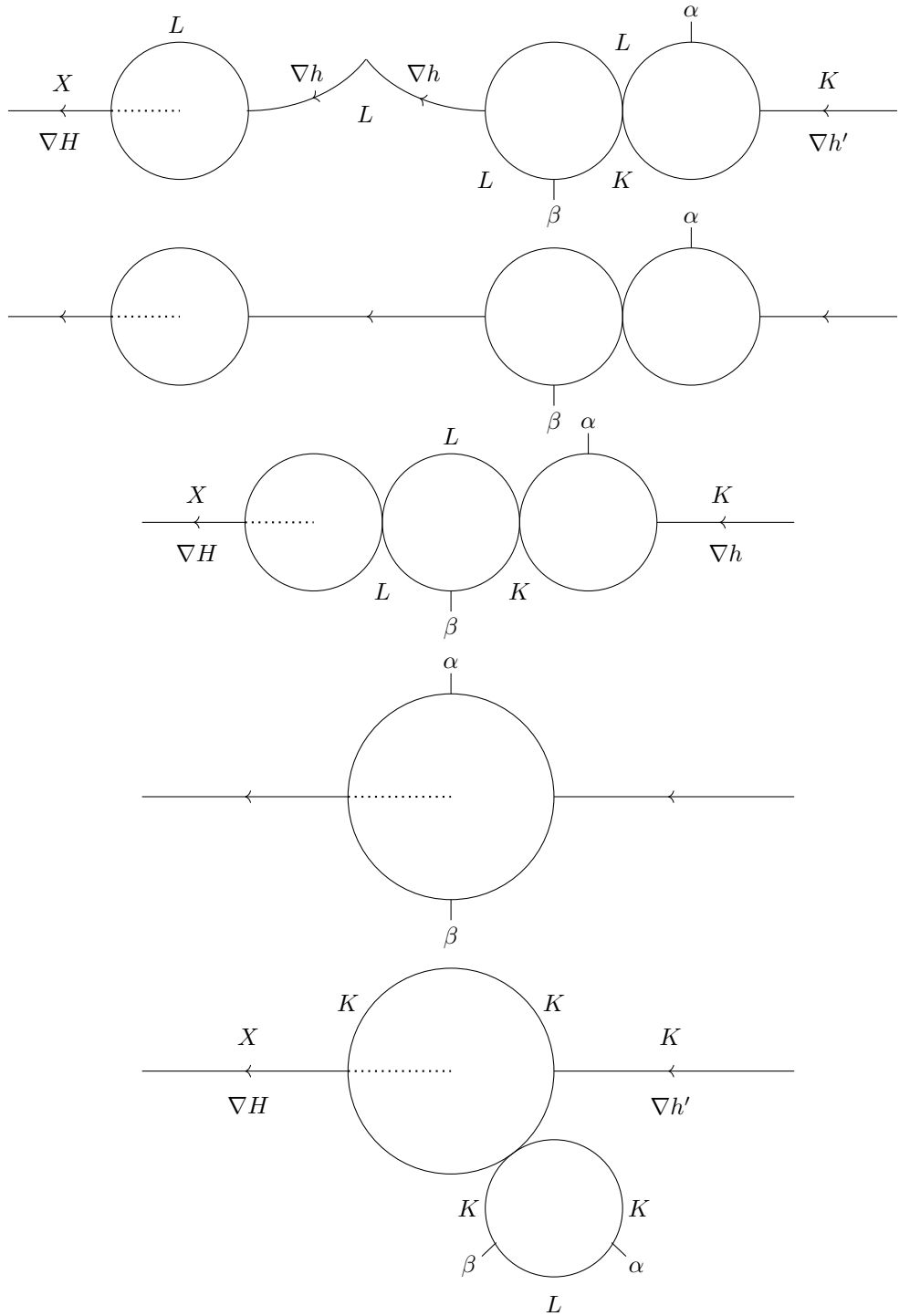

With the above theorem in hand, we now proceed to prove the applications stated in the introduction. First, we show the following proposition:

\begin{proposition} \label{prop:classify-lagrangians-wein1}
    Let $X = M \natural N$ be the Weinstein 1-handle connect sum of two Liouville manifolds $M$ and $N$, where $\pi_1(M) \neq 1$, $\pi_1(N) = 1$, and both $M$ and $N$ satisfy $c_1 = 0$. Then any connected closed exact Lagrangian brane $L \subset X$ whose inclusion induces an isomorphism $\pi_1(L) \cong \pi_1(X)$ is isomorphic, in the wrapped Fukaya category $\mathcal{W}(X)$, to an object in the $\mathcal{W}(M)$ summand under the quasi-equivalence
    \begin{equation}
        \mathcal{W}(X) \simeq \mathcal{W}(M) \oplus \mathcal{W}(N).
    \end{equation}
\end{proposition}
\begin{proof}
    We claim that $L$, as an object in $\mathcal{W}(X)$, lies in the summand corresponding to $\mathcal{W}(M)$ under the quasi-equivalence
    \begin{equation}
        \mathcal{W}(X)\cong\mathcal{W}(M)\oplus\mathcal{W}(N).
    \end{equation}
    For the sake of contradiction, assume that $L$ is isomorphic to an object $L'$ in the $\mathcal{W}(N)$ component.

    We now consider the wrapped Fukaya category $\mathcal{S}(X)$ of Lagrangians equipped with (possibly infinite-dimensional) local systems, with morphisms defined as in Subsection \ref{ssec:locsys}. Let $\mathcal{L}$ be the local system on $X$ corresponding to the regular representation $\Z\pi_1(X)$. Since $N$ is simply-connected, the local system $\mathcal{L}$ restricted to the objects in $\mathcal{W}(N)$ is trivial. On the other hand, since $\pi_1(L)\to\pi_1(X)$ is an isomorphism, the restriction $\mathcal{L}\vert_L$ corresponds to the regular representation of $\pi_1(L)$.

    Now consider the $A_\infty$-functor $\Phi:\mathcal{W}(X)\to\mathcal{S}(X)$ which sends an object $K$ to $(K,\mathcal{L}\vert_K)$. Since $L$ and $L'$ are isomorphic in $\mathcal{W}(X)$, their images $\Phi(L)$ and $\Phi(L')$ are also isomorphic in $\mathcal{S}(X)$. In particular, there is an isomorphism
    \begin{equation}
        CF^*(L,\Phi(L))\cong CF^*(L',\Phi(L')).
    \end{equation}

    Since $L'$ is a Lagrangian submanifold of $N$ which is simply-connected, $\mathcal{L}\vert_{L'}$ is trivial, and thus the object $\Phi(L')$ can be identified as $L'$ equipped with the trivial local system of rank $\lvert\pi_1(X)\rvert$. Therefore, the zeroth degree cohomology $HF^0(L',\Phi(L'))$ has rank $\lvert\pi_1(X)\rvert$. On the other hand, since $\Phi(L)$ is the Lagrangian $L$ equipped with the local system corresponding to the regular representation, we have an isomorphism
    \begin{equation}
        CF^*(L,\Phi(L))\cong C_{sing}^*(L;\mathcal{L}\vert_L)\cong C^*(\Tilde{L};\Z)
    \end{equation}
    for the universal cover $\Tilde{L}$ of $L$. Since $L$ is connected, $HF^0(L,\Phi(L))$ has rank 1, which is a contradiction. Therefore we conclude that $L$ is isomorphic to an object in $\mathcal{W}(M)$.
\end{proof}
\begin{remark}
    If one prefers to restrict to using only finite-dimensional representations, one can attempt to choose a finite-dimensional irreducible nontrivial representation instead of the regular representation. Since the local system corresponding to this representation will admit no sections, one can obtain the same conclusion again by comparing $H^0$. However, we remark that there exist groups (such as Thompson groups) for which there are no nontrivial finite-dimensional representations.
\end{remark}

Since 3-dimensional lens spaces are completely classified up to diffeomorphism by their simple homotopy type, we can finish our proof of Theorem \ref{thm:lens-space}:

\begin{thm}\label{thm:lens-space-reprise}
    Let $X$ be a simply-connected Weinstein manifold of dimension 6 with $c_1(TX)=0$, and let $Q=L(p,q)$ be a 3-dimensional lens space. Then any connected closed exact Maslov zero Spin Lagrangian submanifold $L$ in $M=T^*Q\natural X$ for which the inclusion $\pi_1(L)\xhookrightarrow{}\pi_1(M)$ is an isomorphism must be diffeomorphic to $L(p,q)$.
\end{thm}
\begin{proof}
    We take the universal cover $\tilde{M}$ of $M$, given by
    \begin{equation}
        \tilde{M}=T^*S^3\natural X\natural\cdots\natural X.
    \end{equation}
    Since $\pi_1(L)\to\pi_1(M)$ is an isomorphism, the preimage $\tilde{L}$ in $\tilde{M}$ is a connected closed Lagrangian.

    Because $X$ is simply-connected, Proposition \ref{prop:classify-lagrangians-wein1} implies that $L$, equipped with any brane structure, is isomorphic to a compact object in the wrapped Fukaya category $\mathcal{W}(T^*L(p,q))$, and hence is orthogonal to the $\mathcal{W}(X)$ summand in the direct sum decomposition up to quasi-equivalence
    \begin{equation}
        \mathcal{W}(M)\cong\mathcal{W}(T^*L(p,q))\oplus\mathcal
        W(X).
    \end{equation}
    In particular, the preimage $\tilde{L}$ is orthogonal to all the $\mathcal{W}(X)$ summands in the direct sum decomposition
    \begin{equation}
        \mathcal{W}(\tilde{M})\cong\mathcal{W}(T^*S^3)\oplus\mathcal
        W(X)^{\oplus p}
    \end{equation}
    up to quasi-equivalence. It follows that $\tilde{L}$ is isomorphic to an object in the $\mathcal{W}(T^*S^3)$ summand of $\mathcal{W}(\tilde{M})$ whose endomorphism algebra is cohomologically supported in nonnegative degrees.

    By Lemma \ref{lem:compact-objects-T^*S^n}, we conclude that $\tilde{L}$ is isomorphic to the zero section $S^3$ in $\mathcal{W}(\tilde{M})$, up to a possible degree shift. Applying the same argument as Theorem \ref{thm:auteq-reprise}, we conclude that $L$ is isomorphic to the zero section $Q=L(p,q)$ in $\mathcal{W}(M)$, again up to a degree shift. By Theorem \ref{thm:Weinstein-equivalence-simple-reprise}, $L$ and $Q$ are simply isomorphic objects, and therefore their enhanced cellular cochain complexes $\underline{C^*_{cell}(L)}$ and $\underline{C^*_{cell}(Q)}$ are simple homotopy equivalent. 

    It remains to prove that $L$ is diffeomorphic to $Q$. We first forget the $\pi_1$-action on the enhanced cellular cochain complexes for both $L$ and $Q$. This yields the standard cellular cochain complexes of their universal covers, which are isomorphic. In particular, by the Hurewicz and Whitehead theorems, we obtain a homotopy equivalence between the universal cover of $L$ and $S^3$, which in fact must be a diffeomorphism due to Perelman's proof of the smooth Poincaré conjecture in dimension 3. Thus $L$ is a quotient of $S^3$ by a free $\Z/p$-action, and hence is diffeomorphic to a lens space $L(p,q')$ for some $q'$.

    By our earlier computation of Reidemeister torsions for lens spaces, if the enhanced cochain complexes $\underline{C^*_{cell}(L(p,q))}$ and $\underline{C^*_{cell}(L(p,q'))}$ are simple homotopy equivalent, then $L(p,q)$ and $L(p,q')$ must be diffeomorphic. We conclude that $L$ is diffeomorphic to $Q$, as claimed.
\end{proof}

The existence of exotic spheres in higher dimensions, along with the presence of fake lens spaces in dimensions $\geq 5$, poses a fundamental obstruction to extending the above argument to higher-dimensional lens spaces. In particular, while higher-dimensional lens spaces (defined as quotients of $S^{2n-1}\subset\C^n$ by the action of a cyclic subgroup of $\U(n)$) are completely classified up to diffeomorphism by their Reidemeister torsions, there also exist cyclic quotients of $S^{2n-1}$ arising from non-linear group actions. These ``fake lens spaces'' can be simple homotopy equivalent to genuine lens spaces, but are not necessarily diffeomorphic.

We now conclude the section by revisiting the case of exotic Weinstein balls. Let $\Sigma^{2n}$ be an exotic Weinstein ball, constructed as the total space of a Lefschetz fibration as in \cite{McLean, AbouzaidSeidelLefschetz}. For any non-simply-connected Spin manifold $Q$, define $X$ to be the Weinstein 1-handle connect sum $T^*Q\natural \Sigma$.

\begin{thm} \label{thm:exotic-ball-reprise}
    Let $L$ be any connected closed exact Maslov zero Spin Lagrangian submanifold in $X = T^*Q \natural \Sigma$ such that the inclusion $L \hookrightarrow X$ induces an isomorphism on fundamental groups. Then $L$ is simple homotopy equivalent to $Q$.
\end{thm}

\begin{proof}
    Since the inclusion $Q \hookrightarrow X$ is a homotopy equivalence, Theorem \ref{thm:simple-he-Lagrangians-reprise} reduces the claim to showing that $L$ and $Q$ admit brane structures for which they define isomorphic objects in the Fukaya category $\mathcal{F}(X)$.

    By Proposition \ref{prop:classify-lagrangians-wein1}, the Lagrangian $L$, equipped with any brane structure, is isomorphic to some object $K$ in the wrapped Fukaya category $\mathcal{W}(T^*Q)$. Let $N$ be a cotangent fiber of $T^*Q$. Since $N$ split-generates $\mathcal{W}(T^*Q)$, it suffices to show that the Yoneda module $CF^*(K,N)$ is isomorphic to the Yoneda module of the zero section $Q$, equipped with a suitable $\Z$-local system.

    Set $\mathcal{A}=CW^*(N,N)$ and $\mathcal{M}=CF^*(K,N)$. Following the argument of \cite[Lemma C.1]{Abouzaid12b}, reduction to coefficients in $\Q$ and in $\mathbb{F}_p$ for all primes $p$ shows that $HF^*(K,N)$ is isomorphic to $\Z$, up to a degree shift. In particular, $HF^*(K,N)$ is free over $\Z$, and therefore the homotopy transfer theorem equips it with an $A_\infty$-module structure over $\mathcal{A}$.

    Because this cohomology is concentrated in a single degree, only $\mathcal{A}^0=\Z[\pi_1(Q)]$ can act nontrivially. Hence the resulting $A_\infty$-module is precisely the data of a $\Z$-local system on $Q$.

    We therefore conclude that $\mathcal{M}$ is isomorphic to the Yoneda module of the zero section $Q$ equipped with an appropriate $\Z$-local system. It follows that $L$ may be endowed with a brane structure for which it is isomorphic to $Q$ in $\mathcal{F}(X)$, completing the proof.
\end{proof}


\begin{appendices}
\section{Monotonicity Lemmas}

A standard estimate for the energy of a pseudoholomorphic curve is given by the monotonicity lemma proved in \cite{Sikorav}. In this appendix, we recall two versions of the monotonicity lemma from \cite{Sikorav}, and generalizations proved in \cite{CEL} and \cite{AbouzaidKragh}. 

Let $(X,\omega)$ be a symplectic manifold, and $J$ an almost complex structure. We define such a triple $(X,\omega,J)$ to be \emph{almost Kähler} if the almost complex structure $J$ is $\omega$-compatible. There is a canonical choice of Riemannian metric for an almost Kähler triple $(X,\omega,J)$ given by
\begin{equation}
    g_J(v_1,v_2)=\omega(v_1,Jv_2).
\end{equation}
Because $J$ is $\omega$-compatible, we may define the \emph{energy} of a $J$-holomorphic curve $u:S\to X$ to be
\begin{equation}
    E(u)=\int_S u^*\omega.    
\end{equation}
We also assume that $J$ has regularity $\mathcal{C}^r$ for some $r\geq 1$, and the $J$-holomorphic maps that appear later have regularity $\mathcal{C}^{r+1}$. The key assumption for the monotonicity lemma is a \emph{bounded geometry condition} on the triple $(X,\omega,J)$, which we define below.

\begin{defn} \label{def:bounded-geometry}
    Let $(X,\omega,J)$ be an almost Kähler triple. We define this triple to have \emph{bounded geometry} if the associated Riemannian metric $g_J$ has bounded curvature, and a lower bound on the injectivity radius.
\end{defn}

Any Liouville manifold equipped with a cylindrical compatible almost complex structure satisfies these conditions. We now state the monotonicity lemma.

\begin{lemma}[{\cite[Proposition 4.3.1(ii)]{Sikorav}}] \label{thm:motonocity-lemma}
    Suppose that $(X,\omega,J)$ is an almost Kähler triple with bounded geometry, and let $r_0$ be the global lower bound for the injectivity radius. Then there exists some constant $C>0$ such that the following holds: Consider a Riemann surface with boundary $(S,j)$ and a non-constant $J$-holomorphic map $u:S\to X$. If there exists some $r\leq r_0$ such that
    \begin{equation}
        u(S)\subset B(x,r),~u(\partial S)\subset\partial B(x,r)
    \end{equation}
    for some $x\in u(S)$, then there is a lower bound for the energy of $u$ given by
    \begin{equation}
        E(u)\geq Cr^2.
    \end{equation}
\end{lemma}

There is a version of the monotonicity lemma in \cite{Sikorav} for holomorphic curves with Lagrangian boundary conditions. Let $L$ be a properly embedded Lagrangian submanifold of $X$, which we assume to be either compact or cylindrical. Such Lagrangians satisfy the geometric conditions imposed in \cite[Definition 4.7.1]{Sikorav}.

\begin{lemma}[{\cite[Proposition 4.7.2(ii)]{Sikorav}}] \label{thm:monotonicity-lemma-lag-bdry}
    Suppose that $(X,\omega,J)$ is an almost Kähler triple with bounded geometry, and let $r_0$ be the lower bound for the injectivity radius. We assume that $L$ is either a compact or cylindrical Lagrangian. Then there exists some constant $C_L>0$ depending on $L$ such that the following holds: Consider a Riemann surface with boundary $(S,j)$ and a non-constant $J$-holomorphic map $u:S\to X$. If there exists some $r\leq r_0$ such that
    \begin{equation}
        (u(S),u(\partial S))\subset (B(x,r),\partial B(x,r)\cup L)
    \end{equation}
    for some $x\in u(S)$, then there is a lower bound for the energy of $u$ given by
    \begin{equation}
        E(u)\geq C_Lr^2.
    \end{equation}
\end{lemma}

Cieliebak, Ekholm, and Latschev prove a generalization of the above statement to the case where more than one Lagrangian boundary component is allowed. To be precise, they allow the Lagrangian $L$ to be an immersed Lagrangian with clean self-intersection along a compact submanifold $Z$. For our setting, it is enough to restrict to the case where $L$ is the union of two transversely intersecting Lagrangians $L_0,L_1$, and $Z$ is the union of the intersections $L_0\cap L_1$, which is a finite disjoint union of points. We also impose the following two conditions on $J$:
\begin{enumerate}
    \item $J$ is integrable in a neighborhood of $Z$.
    \item Near each point $x\in Z$, there exists a neighborhood $U_x$ with holomorphic coordinates $U_x\cong\C^n$ such that there is an identification
    \begin{equation*}
        L_0\cong\R^n,~L_1\cong i\R^n.
    \end{equation*}
\end{enumerate}
We further assume that there exists some subset $Y\subset X$ be a subset for which any holomorphic curve with boundary in $Y$ must be entirely contained in $Y$. Since we assume that $X$ is Liouville, such a $Y$ always exists.

\begin{lemma}[{\cite[Lemma 3.4]{CEL}}] \label{thm:monotonicity-lemma-CEL}
    Suppose that $(X,\omega,J)$ is an almost Kähler triple with bounded geometry, and let $r_0$ be the lower bound for the injectivity radius. Let $L$ be an immersed Lagrangian as above, and assume that $J$ satisfies the conditions (1), (2) above. Then there exists some constant $\epsilon_L, C'_L>0$ depending on $L$ such that the following holds: Consider a Riemann surface with boundary $(S,j)$ and a nonconstant $J$-holomorphic map $u:(S,\partial S)\to (X,L)$ passing through a point $x\in Y$ such that $u(S)\cap B(x,r)$ is compact for some $r<\epsilon_L$. Then there is a lower bound for the energy of $u$ given by
    \begin{equation}
        E(u)\geq C'_Lr^2.
    \end{equation}
\end{lemma}

The sketch of the proof is to use a local model $(\R^n,i\R^n)$ for the Lagrangian boundaries, and the map $(z_1,\cdots,z_n)\to(z_1^2,\cdots,z_n^2)$ to reduce the situation to a holomorphic map with boundary on $\R^n\subset\C^n$. 

Abouzaid and Kragh show an application of the monotonicity lemma that shows that under some assumptions, the image of a pseudoholomorphic strip is contractible. Their setup is as follows.

Let $K$ and $L$ be two Lagrangian submanifolds of a symplectic manifold $X$ together with a almost complex structure $J$. Let $x\in K\cap L$ be an isolated intersection point, and denote by $U_x$ an unspecified open neighborhood of $x$. Although it is not mentioned in the paper, one must also assume that $(X,J)$ satisfies the bounded geometry condition to apply the monotonicity lemma.

\begin{lemma}[{\cite[Lemma A.1]{AbouzaidKragh}}] \label{lem:monotonicity-AK}
    For every sufficiently small contractible neighborhood $U_x$ of $x$, there exists a constant $\delta>0$ such that for any $J$-holomorphic curve $u:D^2\to X$ satisfying the following two conditions
    \begin{enumerate}
        \item $u(1)\in U_x$, $u(-1)\not\in U_x$,
        \item the upper half of the boundary $\partial D^2$ is mapped to a $\delta$-neighborhood of $K$, and the lower half of the boundary $\partial D^2$ is mapped to a $\delta$-neighborhood of $L$,
    \end{enumerate}
    the energy $E(u)=\int u^*\omega$ satisfies the following lower bound
\begin{equation}
    E(u)>\delta.
\end{equation}
\end{lemma}

\section{Classifying representatives for twisted complexes} \label{app:generation}

In this appendix, we collect some homological algebra statements and proofs used in Subsection \ref{ssec:app-auteq}. The exposition will mostly follow \cite[Appendix A]{Abouzaid12b}, and \cite[Section 4]{AbouzaidSmith}. The base ring we consider will be $k$, a field. Alternatively, we may take the base ring to be $\Z$, but then the homological perturbation lemma (Lemma \ref{lem:homological-perturbation}) needs an additional condition.

\begin{defn}
    We define an $A_\infty$-algebra $\mathcal{A}$ to be \emph{minimal} if the differential $\mu^1_\mathcal{A}$ vanishes. We define an $A_\infty$-module $\mathcal{M}$ over $\mathcal{A}$ to be \emph{minimal} if the structure map $\mu_\mathcal{M}^0$ vanishes.
\end{defn}
The following \emph{homological perturbation lemma} states that for any $A_\infty$-algebra and $A_\infty$-module, one can find a quasi-isomorphic $A_\infty$-algebra and $A_\infty$-module that is minimal.

\begin{lemma}[{\cite[Lemma A.1]{Abouzaid12b}}] \label{lem:homological-perturbation}
    Let $\mathcal{A}$ be an $A_\infty$-algebra, and let $\mathcal{M}$ be an $A_\infty$-module over $\mathcal{A}$. Then there exists an $A_\infty$-structure on the cohomological algebra $H^*\mathcal{A}$ such that $\mathcal{A}\to H^*\mathcal{A}$ is a quasi-isomorphism of $A_\infty$-algebras. Similarly, there exists an $A_\infty$-module structure on $H^*\mathcal{M}$ that is quasi-isomorphic to $\mathcal{M}$.
\end{lemma}

If the base ring is $\Z$, the homological perturbation lemma holds when each degree component of the cohomology $H^*\mathcal{A}$, $H^*\mathcal{M}$ is a free $\Z$-module (\cite[Theorem 1]{Petersen}). Explicit formulae can be found in (\cite[Section 12]{Berglund}) and \cite[Section 6.4]{KontsevichSoibelman}.

Now consider the case when $\mathcal{A}$ is a minimal $A_\infty$-algebra that is supported on degrees less or equal to $0$. For example, this is satisfied when $\mathcal{A}$ is the wrapped Floer cohomology ring of a cotangent fiber, or direct sums of such algebras:
\begin{equation}
    \mathcal{A}=HW^*(T^*_{q_1}Q_1,T^*_{q_1}Q_1)\oplus HW^*(T^*_{q_2}Q_2,T^*_{q_2}Q_2).
\end{equation}
This follows from the isomorphism between wrapped Floer cohomology of a cotangent fiber and the homology of the based loop space (Proposition \ref{prop:AS}), as seen in the isomorphism below
\begin{equation}
        HW^*(T^*_qQ,T^*_qQ)\cong H_{-*}(\Omega Q).
\end{equation}

For a minimal $A_\infty$-algebras that are supported in non-positive degrees,
each degree component of minimal $A_\infty$-modules are its $A_\infty$-submodules:

\begin{lemma}[{\cite[Lemma 4.6]{AbouzaidSmith}}] \label{lem:degree-filtration}
    Suppose that $\mathcal{A}$ is a minimal $A_\infty$-algebra that is supported in non-positive degrees, and let $\mathcal{M}$ be an $A_\infty$-module over $\mathcal{A}$. Then each degree component is an $A_\infty$-submodule of $\mathcal{M}$ over $\mathcal{A}$, and is determined by the $A_\infty$-module structure over $\mathcal{A}^0$, up to quasi-equivalence.
\end{lemma}

It is possible to classify minimal $A_\infty$-modules $\mathcal{M}$ over a minimal $A_\infty$-algebra $\mathcal{A}$ supported in non-positive degrees, under finiteness assumptions:

\begin{lemma}[{\cite[Lemma 4.7]{AbouzaidSmith}}] \label{lem:classifcation-minimal-modules}
    Suppose that $\mathcal{A}$ is a minimal $A_\infty$-algebra that is supported in non-positive degrees. Let $\mathcal{M}$ be a minimal $A_\infty$-module of finite rank over $\mathcal{A}$, cohomologically supported in finitely many degrees. Then there is a twisted complex
    \begin{equation}
        \mathcal{L}=(\bigoplus_i k[i]\otimes M^{-i},\delta_{ij})
    \end{equation}
    equivalent to $\mathcal{M}$, where each $M^{-i}$ is an $\mathcal{A}^0$-module, and $\deg(\delta_{ij})>1$.
\end{lemma}

In particular, we may set $\mathcal{A}$ to be $\mathcal{A}=HW^*(T^*_{q_1}Q_1,T^*_{q_1}Q_1)\oplus HW^*(T^*_{q_2}Q_2,T^*_{q_2}Q_2)$ as above. Then we have that
\begin{equation}
    \mathcal{A}^0\cong H_0(\Omega Q_1;k)\oplus H_0(\Omega Q_2;k),
\end{equation}
and so $\mathcal{A}^0$-modules can be regarded as the $k$-representations of $\pi_1(Q_1)$ and $\pi_1(Q_2)$, which are local systems over $Q_1$ and $Q_2$.

Another lemma that will be useful for us is \cite[Lemma A.4]{Abouzaid12b}. It proves that any twisted complex $\mathcal{L}$ built from shifted copies of the same object $Q$ whose endomorphism algebra is supported in non-positive degrees can be shown to be actually supported in a single degree. One can trace the origins of the argument back to a spectral sequence argument in \cite{FukayaSeidelSmith}.

\begin{lemma} \label{lem:twisted-cpx-endo}
    Suppose that $\mathcal{P}$ is a minimal $A_\infty$-algebra such that $\mathrm{End}^*(\mathcal{P})$ is supported in non-negative degrees. Let $\mathcal{L}$ be a twisted complex of $\mathcal{P}$-modules of finite rank.
    \begin{enumerate}
        \item If $\mathrm{End}^*(\mathcal{L})$ is supported in non-negative degrees, then there exists a free $k$-module $V$ such that there is an isomorphism
        \begin{equation*}
            \mathcal{L}\cong V[i]\otimes\mathcal{P}
        \end{equation*}
        for some integer $i$ accounting for the degree shift.
        \item Moreover, if $H^0\mathrm{End}(\mathcal{L})$ is a free $k$-module of rank $1$, then $\mathcal{L}$ is equivalent to a single copy of $\mathcal{P}$, possibly up to some shift.
    \end{enumerate}
\end{lemma}
\begin{proof}
    The first part is exactly the content of \cite[Lemma A.4]{Abouzaid12b}. For the second part, note that
    \begin{equation}
        H^0\mathrm{End}(\mathcal{L},\mathcal{L})\cong\hom_k(V,V)\otimes_k H^0\mathrm{End}(\mathcal{P}).
    \end{equation}
    By comparing the ranks of both sides, one can show that $V$ has rank $1$.
\end{proof}

To clarify a possible point of confusion, this lemma will apply to the case where the $A_\infty$-algebra $\mathcal{P}$ is $HF^*(Q,Q)$ for some compact generating Lagrangian $Q$, while the previous lemmas apply to the case where the $A_\infty$-algebra $\mathcal{A}$ is $HW^*(T^*_qQ,T^*_qQ)$. In plain English, Lemma \ref{lem:twisted-cpx-endo} says that any twisted complex made out of a single compact connected Lagrangian $L$ that represents another compact connected Lagrangian $K$ is actually a single copy of $L$, up to shifts.
\end{appendices}

\end{document}